\documentclass{amsart}
\usepackage{amssymb}
\usepackage{amsmath, amscd}
\usepackage{amsthm}
\usepackage{mathtools}
\usepackage{mathrsfs}
\usepackage{perpage}
\usepackage[all,cmtip]{xy}
\usepackage{setspace}
\usepackage{amsbsy}
\usepackage{url}
\usepackage[T1]{fontenc}
\usepackage{verbatim}
\usepackage{mathrsfs}
\usepackage{ifthen}
\usepackage{blkarray}
\usepackage{multirow}
\usepackage{enumerate}
\usepackage{enumitem}
\usepackage[foot]{amsaddr}
\usepackage{bm}
\usepackage[left=2cm,head=13pt,
top=1.5cm,right=2cm, bottom=1.5cm]{geometry}

\AtEndDocument{\bigskip{\footnotesize%
\noindent (Wushi Goldring)
  \textsc{Department of Mathematics, Stockholm University, Stockholm SE-10691, Sweden} \par  
 \noindent \texttt{wushijig@gmail.com} \par
  \addvspace{\medskipamount}
\noindent (Jean-Stefan Koskivirta) \textsc{Department of Mathematics, South Kensington Campus, Imperial College London, London SW7 2AZ, UK} \par 
\noindent \texttt{j.koskivirta@imperial.ac.uk}
}}

\usepackage{lipsum}

\numberwithin{equation}{subsection}

\newtheorem{theorem}{Theorem}[subsection]
\newtheorem{lemma}[theorem]{Lemma}

\newtheorem{corollary}[theorem]{Corollary}
\newtheorem{definition}[theorem]{Definition}

\newtheorem{proposition}[theorem]{Proposition}
\newtheorem{condition}[theorem]{Condition}

\newtheorem{Assumption}[theorem]{Assumption}

\theoremstyle{remark}

\newtheorem{rmk}[theorem]{Remark}

\title{Strata Hasse invariants, Hecke algebras and Galois representations}

\newcommand{\GZip}{\mathop{\text{$G$-{\tt Zip}}}\nolimits}
\newcommand{\GoneZip}{\mathop{\text{$G_1$-{\tt Zip}}}\nolimits}
\newcommand{\GtwoZip}{\mathop{\text{$G_2$-{\tt Zip}}}\nolimits}
\newcommand{\GspZip}{\mathop{\text{$GSp(2g)$-{\tt Zip}}}\nolimits}

\newcommand{\GF}{\mathop{\text{$G$-{\tt ZipFlag}}}\nolimits}

\newskip\procskipamount
\newskip\interskipamount
\newskip\refskipamount
\newcommand{\procskip}{\vskip\procskipamount}
\newcommand{\interskip}{\vskip\interskipamount}
\newcommand{\refskip}{\vskip\refskipamount}

\newcommand{\procbreak}{\par
   \ifdim\lastskip<\procskipamount\removelastskip
   \penalty-100
   \procskip\fi
   \noindent\ignorespaces}

\newcommand{\titlebreak}{\par%
\ifdim\lastskip<\interskipamount\removelastskip%
\penalty10000%
\interskip\fi%
\noindent}%

\newcommand{\interbreak}{\par%
\ifdim\lastskip<\interskipamount\removelastskip%
\penalty-100%
\interskip\fi%
\noindent\ignorespaces}%

\newcommand{\refbreak}{\par%
\ifdim\lastskip<\refskipamount\removelastskip%
\penalty-100%
\refskip\fi%
\noindent\ignorespaces}%

\newcounter{listcounter}
\newcounter{deflistcounter}
\newcounter{equivcounter}

\newskip{\itemsepamount}
\newskip{\topsepamount}
\newenvironment{assertionlist}{%
  \begin{list}
    {\upshape (\arabic{listcounter})}
    {\setlength{\leftmargin}{18pt}
     \setlength{\rightmargin}{0pt}
     \setlength{\itemindent}{0pt}
     \setlength{\labelsep}{5pt}
     \setlength{\labelwidth}{13pt}
     \setlength{\listparindent}{\parindent}
     \setlength{\parsep}{0pt}
     \setlength{\itemsep}{\itemsepamount}
     \setlength{\topsep}{\topsepamount}
     \usecounter{listcounter}}}
  {\end{list}}

\newenvironment{definitionlist}{%
  \begin{list}
    {\upshape (\alph{deflistcounter})}
    {\setlength{\leftmargin}{18pt}
     \setlength{\rightmargin}{0pt}
     \setlength{\itemindent}{0pt}
     \setlength{\labelsep}{5pt}
     \setlength{\labelwidth}{13pt}
     \setlength{\listparindent}{\parindent}
     \setlength{\parsep}{0pt}
     \setlength{\itemsep}{\itemsepamount}
     \setlength{\topsep}{\topsepamount}
     \usecounter{deflistcounter}}}
  {\end{list}}

\newenvironment{equivlist}{%
  \begin{list}
    {\upshape (\roman{equivcounter})}
    {\setlength{\leftmargin}{18pt}
     \setlength{\rightmargin}{0pt}
     \setlength{\itemindent}{0pt}
     \setlength{\labelsep}{5pt}
     \setlength{\labelwidth}{13pt}
     \setlength{\listparindent}{\parindent}
     \setlength{\parsep}{0pt}
     \setlength{\itemsep}{\itemsepamount}
     \setlength{\topsep}{\topsepamount}
     \usecounter{equivcounter}}}
  {\end{list}}

% ...
\newcommand{\Acal}{{\mathcal A}}
\newcommand{\Bcal}{{\mathcal B}}
\newcommand{\Ccal}{{\mathcal C}}
\newcommand{\Dcal}{{\mathcal D}}

\newcommand{\Fcal}{{\mathcal F}}
\newcommand{\Gcal}{{\mathcal G}}
\newcommand{\Hcal}{{\mathcal H}}
\newcommand{\Ical}{{\mathcal I}}

\newcommand{\Kcal}{{\mathcal K}}
\newcommand{\Lcal}{{\mathcal L}}
\newcommand{\Mcal}{{\mathcal M}}
\newcommand{\Ncal}{{\mathcal N}}
\newcommand{\Ocal}{{\mathcal O}}
\newcommand{\Pcal}{{\mathcal P}}
\newcommand{\Qcal}{{\mathcal Q}}
\newcommand{\Rcal}{{\mathcal R}}

\newcommand{\Xcal}{{\mathcal X}}
\newcommand{\Ycal}{{\mathcal Y}}
\newcommand{\Zcal}{{\mathcal Z}}

\newcommand{\gfr}{{\mathfrak g}}

\newcommand{\lfr}{{\mathfrak l}}
\newcommand{\mfr}{{\mathfrak m}}
\newcommand{\nfr}{{\mathfrak n}}

\newcommand{\pfr}{{\mathfrak p}}
\newcommand{\qfr}{{\mathfrak q}}

\newcommand{\tfr}{{\mathfrak t}}

\newcommand{\zfr}{{\mathfrak z}}

\renewcommand{\AA}{\mathbf{A}}
\newcommand{\BB}{\mathbf{B}}
\newcommand{\CC}{\mathbf{C}}

\newcommand{\FF}{\mathbf{F}}
\newcommand{\GG}{\mathbf{G}}

\newcommand{\LL}{\mathbf{L}}

\newcommand{\NN}{\mathbf{N}}

\newcommand{\PP}{\mathbf{P}}
\newcommand{\QQ}{\mathbf{Q}}
\newcommand{\RR}{\mathbf{R}}
\renewcommand{\SS}{\mathbf{S}}
\newcommand{\TT}{\mathbf{T}}

\newcommand{\XX}{\mathbf{X}}

\newcommand{\ZZ}{\mathbf{Z}}

\newcommand{\mmu}{\bm{\mu}}

\DeclareMathOperator{\Pic}{Pic}

\DeclareMathOperator{\nonvanish}{nonvanish}
\DeclareMathOperator{\mult}{mult}
\DeclareMathOperator{\Fil}{Fil}
\DeclareMathOperator{\Gr}{Gr}

\newcommand{\Fscr}{{\mathscr F}}
\newcommand{\Gscr}{{\mathscr G}}

\newcommand{\Lscr}{{\mathscr L}}

\newcommand{\Sscr}{{\mathscr S}}

\newcommand{\Vscr}{{\mathscr V}}

\newcommand{\Zscr}{{\mathscr Z}}

\DeclareMathOperator{\tr}{tr}

\newcommand{\cent}{{\rm Cent}}

\newcommand{\frobv}{{\rm Frob}_v}

\newcommand{\pn}{\mathbf P^n}

\newcommand{\qbar}{\overline{\mathbf Q}}

\newcommand{\fp}{\mathbf F_p}

\newcommand{\galq}{{\rm Gal}(\qbar / \QQ)}
\newcommand{\galqv}{{\rm Gal}(\qvbar / \qv)}

\newcommand{\glnqpbar}{GL(n, \qpbar)}
\newcommand{\glmc}{GL(m, \CC)}

\newcommand{\glmqpbar}{GL(m, \qpbar)}

\newcommand{\leftexp}[2]{{\vphantom{#2}}^{#1}{#2}}
\newcommand{\lbfg}{\leftexp{L}{\GG}}
\newcommand{\lbfgv}{\leftexp{L}{\GG}_v}
\newcommand{\lbfgvz}{\leftexp{L}{\GG}^{\circ}_v}
\newcommand{\lbfgvzc}{\leftexp{L}{\GG}^{\circ}_v(\CC)}

\renewcommand{\lg}{\leftexp{L}{G}}

\newcommand{\lbfgc}{\lbfg (\CC)}

\newcommand{\lbfgz}{\leftexp{L}{\GG}^{\circ}}

\newcommand{\rfdsslbfgv}{R_{\fd}^{\sesi}(\lbfgv)}
\newcommand{\rfdlbfgv}{R_{\fd}(\lbfgv)}

\newcommand{\zplbfgv}{\ZZ_p[\lbfgv]}
\newcommand{\chargp}{X^*}
\newcommand{\chargpldom}{X^*_{+,\LL}(\TT)}

\newcommand{\chargpbft}{\chargp(\mathbf T)}

\newcommand{\frob}{{\rm Frob}}
\newcommand{\zgeqo}{\ZZ_{\geq 1}}
\newcommand{\zgeqz}{\ZZ_{\geq 0}}
\newcommand{\rgeqz}{\RR_{\geq 0}}

\DeclareMathOperator{\ad}{ad}
\DeclareMathOperator{\Ad}{Ad}
\DeclareMathOperator{\card}{Card}
\DeclareMathOperator{\class}{Class}
\DeclareMathOperator{\cd}{cd}

\newcommand{\dR}{{\rm dR}}
\DeclareMathOperator{\type}{type}
\DeclareMathOperator{\Span}{Span}
\DeclareMathOperator{\Ker}{Ker}
\DeclareMathOperator{\ima}{Im}
\DeclareMathOperator{\fd}{fd}

\DeclareMathOperator{\Gal}{Gal}

\DeclareMathOperator{\Hom}{Hom}

\DeclareMathOperator{\Lie}{Lie}
\newcommand{\logdr}{\mathop{\textnormal{log-dR} }\nolimits}
\newcommand{\logcrys}{\mathop{\textnormal{log-crys} }\nolimits}
\DeclareMathOperator{\imag}{Im}
\DeclareMathOperator{\cusp}{cusp}

\DeclareMathOperator{\orb}{Orb}

\DeclareMathOperator{\pseudo}{pseudo}

\DeclareMathOperator{\Ram}{Ram}
\DeclareMathOperator{\rk}{rk}
\DeclareMathOperator{\rec}{rec}

\DeclareMathOperator{\res}{Res}

\DeclareMathOperator{\Sh}{Sh}
\DeclareMathOperator{\spec}{Spec}

\DeclareMathOperator{\stab}{Stab}
\DeclareMathOperator{\std}{Std}
\DeclareMathOperator{\Sch}{Sbt}
\DeclareMathOperator{\Sym}{Sym}
\DeclareMathOperator{\transl}{Transl}
\DeclareMathOperator{\Transp}{Transp}

\newcommand{\zp}{\mathbf Z_p}

\newcommand{\qp}{\QQ_p}
\newcommand{\qv}{\QQ_v}
\newcommand{\gal}{{\rm Gal}}

\newcommand{\galf}{\gal(\overline{F}/F)}

\newcommand{\qpbar}{\overline{\QQ}_{p}}
\newcommand{\qvbar}{\overline{\QQ}_{v}}

\newcommand{\shgx}{\Sh(\mathbf G, \mathbf X)}

\newcommand{\egx}{E(\mathbf G, \mathbf X)}
\newcommand{\gx}{(\mathbf G, \mathbf X)}

\newcommand{\rescr}{\res_{\mathbf C/\mathbf R}}
\newcommand{\ccross}{\mathbf C^{\times}}

\newcommand{\gofr}{\GG(\RR)}
\newcommand{\gofc}{\GG(\CC)}

\newcommand{\rpipi}{R_{p,\iota}(\pi)}
\newcommand{\rpiineta}{R_{p,\iota}(i, n, \eta)}
\newcommand{\rpiinetar}{R_{p,\iota}(r;i, n, \eta)}

\newcommand{\gofaf}{\mathbf G(\mathbf A_f)}
\newcommand{\gofqp}{\mathbf G(\qp)}
\newcommand{\gofqv}{\mathbf G(\qv)}

\newcommand{\gofafp}{\mathbf G(\mathbf A_f^p)}

\newcommand{\End}{{\rm End}}

\newcommand{\id}{{\rm Id}}
\newcommand{\red}{{\rm red}}

\newcommand{\sesi}{{\rm ss}}

\newcommand{\Th}{{\rm Th.}}
\newcommand{\Ths}{{\rm Ths.}}
\newcommand{\Rmk}{{\rm Rmk.}}
\newcommand{\Rmks}{{\rm Rmks.}}
\newcommand{\Cor}{{\rm Cor.}}

\newcommand{\Lem}{{\rm Lem.}}

\newcommand{\Conj}{{\rm Conj.}}

\newcommand{\Chap}{{\rm Chap.}}

\newcommand{\Def}{{\rm Def.}}

\newcommand{\Prop}{{\rm Prop.}}
\newcommand{\Props}{{\rm Props.}}
\newcommand{\Ex}{{\rm Ex.}}

\newcommand{\App}{{\rm App.}}
\newcommand{\loccit}{{\em loc.\ cit. }}
\newcommand{\loccitn}{{\em loc.\ cit.}}
\newcommand{\cf}{{\em cf. }}
\newcommand{\ie}{i.e.,\ }
\newcommand{\eg}{e.g.,\ }
\newcommand{\viz}{{\em viz. }}

\newcommand{\diag}{{\rm diag}}

\newcommand{\fil}{{\rm Fil}}

\newcommand{\sub}{{\rm sub}}
\newcommand{\can}{{\rm can}}

\newcommand{\Shk}{\Sscr_{\Kcal}}

\newcommand{\Shko}{S_{\Kcal}}
\newcommand{\Shkn}{\Sscr_{\Kcal}^n}

\newcommand{\Shktor}{\Sscr_{\Kcal}^{\Sigma} }
\newcommand{\Shktorc}{\Sscr_{\Kcal, \CC}^{\Sigma} }

\newcommand{\Flktorn}{ \Fcal l_{\Kcal}^{\Sigma,n} }

\newcommand{\Shktorn}{\Sscr_{\Kcal}^{\Sigma,n} }

\newcommand{\Shktoro}{S_{\Kcal}^{\Sigma} }

\newcommand{\Shkmin}{\Sscr_{\Kcal}^{\min} }

\newcommand{\Shkmino}{S_{\Kcal}^{\min} }

\newcommand{\Flk}{ \Fcal l_{\Kcal}}
\newcommand{\Flktoro}{ Fl_{\Kcal}^{\Sigma} }

\newcommand{\Flktor}{ \Fcal l_{\Kcal}^{\Sigma} }

\newcommand{\veta}{\Vscr({\eta})}
\newcommand{\vlambda}{\Vscr({\lambda})}
\newcommand{\leta}{\Lscr({\eta})}

\newcommand{\vcan}{\Vscr^{\can}}
\newcommand{\lcan}{\Lscr^{\can}}

\newcommand{\vcaneta}{\vcan(\eta)}

\newcommand{\lcaneta}{\lcan(\eta)}

\newcommand{\vsub}{\Vscr^{\sub}}
\newcommand{\lsub}{\Lscr^{\sub}}

\newcommand{\vsubeta}{\vsub(\eta)}

\newcommand{\lsubeta}{\Lscr^{\sub}({\eta})}

\newcommand{\Shgk}{\Sscr_{g,\tilde \Kcal}}
\newcommand{\Shgkmin}{\Sscr_{g,\tilde \Kcal}^{\min}}
\newcommand{\Shgko}{S_{g,\tilde \Kcal}}

\renewcommand{\Im}{{\rm Im}}

\newcommand{\gmc}{\mathbf{G}_{m, \mathbf C}}

\newcommand{\xg}{\mathbf X_{g}}

\newcommand{\shdagsp}{(GSp(2g), \xg)}

\newcommand{\ox}{\mathcal O_X}

\newcommand{\iw}{\leftexp{I}{W}}

\newcommand{\rampi}{\Ram(\pi)}

\DeclareMathOperator{\bc}{BC}

\renewcommand{\div}{{\rm div}}

\newcommand{\gv}{\mathbf G_v}

\DeclareMathOperator{\Isomcal}{\mathscr{I}\!\!\mathit{som}}

\begin{document}

\author{Wushi Goldring and Jean-Stefan Koskivirta}

\begin{abstract}
We construct group-theoretical generalizations of the Hasse invariant on strata closures of the stacks $\GZip^{\mu}$.
Restricting to zip data of Hodge type, we obtain a group-theoretical Hasse invariant on every Ekedahl-Oort stratum closure of a general Hodge-type Shimura variety. A key tool is the construction of a stack of zip flags $\GF^\mu$, fibered in flag varieties over $\GZip^{\mu}$. 
It provides a simultaneous generalization of the "classical case" homogeneous complex manifolds studied by Griffiths-Schmid and the "flag space" for Siegel varieties studied by Ekedahl-van der Geer. 

Four applications are obtained: 
(1) Pseudo-representations are attached to the coherent cohomology of Hodge-type Shimura varieties modulo a prime power. (2) Galois representations are associated to many automorphic representations with non-degenerate limit of discrete series archimedean component.
(3) It is shown that all Ekedahl-Oort strata in the minimal compactification of a Hodge-type Shimura variety are affine, thereby proving a conjecture of Oort. 
(4) Part of Serre's letter to Tate on mod $p$ modular forms is generalized to general Hodge-type Shimura varieties.

\end{abstract}

\pagestyle{plain}
\maketitle
\setcounter{tocdepth}{1}
\tableofcontents
\section*{Introduction}  
\renewcommand{\thesubsection}{{I.\arabic{subsection}}}
\renewcommand{\thesubsubsection}{\thesubsection.\arabic{subsubsection}}

This work takes the first step in a program that connects two areas:

\begin{enumerate}[label=(\Alph*)]
\item {\em Automorphic Algebraicity}: The inherent algebro-geometric properties of automorphic representations, particularly those conjectured by the Langlands correspondence. \label{item-algebraicity-Langlands}
\item {\em $G$-Zip geometricity}: The geometry engendered by the theory of $G$-Zips, including the Ekedahl-Oort (EO) stratification of Shimura varieties, their flag spaces and Hasse invariants.  \label{item-zips}
\end{enumerate}

In contrast with previous work, one of the novel features of our approach is to consider the two areas above simultaneously. 
While each has separately undergone significant developments over the past fifteen-twenty years, there has been surprisingly little work relating Automorphic Algebraicity with $G$-Zip Geometricity. Most of the papers cited above on the Langlands correspondence only used the classical Hasse invariant, but not deeper aspects of the EO stratification. 
At the same time, the works which developed the theory of the EO stratification -- and more recently of $G$-Zips -- rarely studied applications to the Langlands correspondence.

We were also inspired by the possibility of connections with a third area:
\begin{enumerate}[label=(\Alph*),resume]
\item {\em Griffiths-Schmid Algebraicity}: Is there an algebro-geometric framework which applies to Griffiths-Schmid manifolds?  \label{item-Griffiths-Schmid}
\end{enumerate}
Recall that Carayol has pursued a program of relating Automorphic Algebraicity and Griffiths-Schmid Algebraicity over the last twenty years \cite{Carayol-LDS-1,Carayol-LDS-2,Carayol-LDS-3,Carayol-Knapp,Carayol-laumon-volume}. His program was developed further by Green-Griffiths-Kerr \cite{Green-Griffiths-Kerr-CBMS-Texas}.

The common pursuit of~\ref{item-algebraicity-Langlands},~\ref{item-zips} and~\ref{item-Griffiths-Schmid} rests on two themes developed by Deligne, Serre and their collaborators. These themes are (i) {\em geometry-by-groups} and (ii) {\em characteristic-shifting} -- back and forth between characteristic $0$ and $p$.

Regarding (i), geometry-by-groups manifests itself in two stages. The first is that geometric objects at the heart of \ref{item-algebraicity-Langlands}, \ref{item-zips}, \ref{item-Griffiths-Schmid} -- Shimura varieties, stacks of $G$-Zips and Griffiths-Schmid manifolds -- are all constructed from the same group-theoretic template:
 A pair $(G, \mu)$ consisting of a reductive group $G$ and a cocharacter $\mu$. The second stage is guided by the more general hypothesis that {\em all} objects constructed from reductive groups should admit a close-knit relationship with algebraic geometry. In our setting, two fundamental test-cases are  Automorphic Algebraicity and Griffiths-Schmid Algebraicity.
 
As for (ii), it is well-known that the method of characteristic-shifting applies throughout algebraic geometry; \eg  the characteristic $p$ approach of Deligne-Illusie to Kodaira vanishing and the degeneration of the Hodge-de Rham spectral sequence. Inspired by Deligne-Serre, this paper applies characteristic-shifting {\em in tandem} with geometry-by-groups: The most basic application is to $G$-Zip Geometricity. Joined to the latter, the duo is next applied to Automorphic Algebraicity. 

Further elaboration of our program is given in our papers \cite{Goldring-Koskivirta-zip-flags,Goldring-Koskivirta-global-sections-compositio}
and joint work in progress with B. Stroh and Y. Brunebarbe \cite{Brunebarbe-Goldring-Koskivirta-Stroh-ampleness}. 
Building on the current work, all of these papers compound evidence that $G$-Zip Geometricity is a "mod $p$ Hodge theory" which interacts with classical Hodge theory via characteristic-shifting. 
Specifically, our results suggest that stacks of $G$-Zips are mod $p$ analogues of period domains (and more generally Mumford-Tate domains, of which Griffiths-Schmid manifolds are quotients \cite{Green-Griffiths-Kerr-Mumford-Tate-Domains-book}). 
An idea along these lines was first put forth for $GL(n)$-Zips by Moonen-Wedhorn in the introduction of \cite{Moonen-Wedhorn-Discrete-Invariants}. The work of Griffiths-Schmid on the manifolds which bear their name \cite{Griffiths-Schmid-homogeneous-complex-manifolds} motivates the flag spaces which play a key role in this work, see Remark~\ref{rmk-GS}.

\subsection{Group-theoretical Hasse invariants} 
\label{sec-intro-gp-hasse}
The main technical result of this paper is the construction of group-theoretical Hasse invariants on strata closures in the stacks $\GZip^{\mu}$ (\Th~\ref{th-intro-gp-hasse}). 
As recalled below,  the theory of the EO stratification has evolved in three stages, progressively shifting from a combinatorial viewpoint to a group-theoretic one. Our work develops a fourth stage in this progression.  
\subsubsection{The Ekedahl-Oort stratification} 

Initially, Oort defined a stratification of $\Acal_g \otimes \fp$, the moduli space of principally polarized abelian varieties in characteristic $p>0$, by isomorphism classes of the $p$-torsion \cite{Oort-stratification-moduli-space-abelian-varieties}. 
He parametrized strata combinatorially by "elementary sequences". 
Later on, Moonen used the canonical filtration of a $BT_1$ with PEL-structure to give a group-theoretical classification of these \cite{moonen-gp-schemes}. He described isomorphism classes of $BT_1$'s over an algebraically closed field $k$ of characteristic $p$ as a certain subset ${}^I W$ of the Weyl group $W$ of the reductive group attached to the PEL-structure.

In the third stage, in a series of papers, Moonen, Wedhorn, Pink and Ziegler defined the algebraic stack $\GZip^\mu$, \cite{Moonen-Wedhorn-Discrete-Invariants,Wedhorn-ordinariness-Shimura-varieties,Pink-Wedhorn-Ziegler-zip-data,PinkWedhornZiegler-F-Zips-additional-structure}, whose $k$-points parametrize isomorphism classes of $BT_1$'s with $G$-structure over $k$. Viehmann and Wedhorn showed that the special fiber $\Shko$ of any PEL-type Shimura variety admits a universal $G$-zip of type $\mu$, which gives rise to a faithfully flat morphism of stacks $\zeta:\Shko\to \GZip^\mu$, \cite{Viehmann-Wedhorn}. 
By definition, the fibers of $\zeta$ are the EO strata of $\Shko$; this definition agrees with those of Ekedahl-Oort and Moonen in the Siegel and PEL cases respectively.
In his thesis, Zhang constructed a universal $G$-zip over the special fiber of a general Hodge-type Shimura variety. He proved that the induced map $\zeta$ is smooth \cite{ZhangEOHodge} (see also \cite{Wortmann-mu-ordinary}).

\subsubsection{Hasse invariants}
The theory of Hasse invariants in its modern form goes back to the algebro-geometric pursuit of modular forms modulo $p$, pioneered by Deligne, Katz, Serre and others in the late 1960's and early 1970's. Since then, many people have studied generalizations and applications of the classical Hasse invariant of an abelian scheme \cf \cite{Goren-partial-hasse,Ito-hasse,Emerton-Reduzzi-Xiao}. 

The work \cite{Goldring-Nicole-mu-Hasse} by the first author and Nicole was the first to produce a Hasse invariant on a general class of Shimura varieties -- those of PEL-type A -- whose classical ordinary locus is often empty.
The second author and Wedhorn extended this result to all Shimura varieties of Hodge-type, \cite{Koskivirta-Wedhorn-Hasse}.
Their work was the first to construct a group-theoretical generalization of the Hasse invariant on the stack of $G$-Zips. A partial group-theoretical result concerning smaller strata was also obtained by the second author \cite{Koskivirta-compact-hodge}. It is used as a starting point for the construction of Hasse invariants in this paper.  

One advantage of the $G$-Zip approach to Hasse invariants is that sections obtained by pull-back from this stack to a Shimura variety are automatically Hecke-equivariant. Another is that geometric properties of such sections can be read off from a root datum of $G$.  
\subsubsection{New results on Hasse invariants} \label{sec-intro-new-hasse}
Let $p$ be a prime ($p=2$ allowed). 
Let $G$ be a connected, reductive $\fp$-group and $\mu \in X_*(G)$ a cocharacter. 
The works of Moonen-Wedhorn \cite{Moonen-Wedhorn-Discrete-Invariants} and Pink-Wedhorn-Ziegler \cite{Pink-Wedhorn-Ziegler-zip-data,PinkWedhornZiegler-F-Zips-additional-structure} define the stack $\GZip^\mu$  of $G$-zips of type $\mu$ (\Def~\ref{def zip datum}) and attach to the pair $(G,\mu)$ the {\em zip group} $E$  (\S \ref{def zip datum}). This group acts naturally on $G$ and one has $\GZip^{\mu} \cong [E\backslash G]$. The $E$-orbits in $G$ are locally closed subsets $G_w$, parametrized by elements $w\in {}^I W$ (\S\ref{subsection strata GZip}).
Let $L=\cent(\mu)$ be the centralizer of $\mu$ in $G$. Every character $\chi \in X^*(L)$ gives rise to a line bundle $\Vscr(\chi)$ on $\GZip^\mu$ (\S\ref{Gvarstacks}). It satisfies $\Vscr(N\chi)=\Vscr(\chi)^N$ for all $N\ge 1$.
\begin{theorem}[Group-theoretical Hasse invariants, \Th~\ref{GTHI}]\label{th-intro-gp-hasse}
Let $G_w\subset G$ be an $E$-orbit;  $\overline{G}_w$ its Zariski closure\footnote{Unless otherwise stated, the Zariski closure is always equipped with the reduced scheme structure.}. Assume $\chi\in X^*(L)$ is $(p,L)$-admissible \textnormal{(\Def~\ref{def-orb-p-close})}. Then there exists $N_w\geq 1$ and a section $h_w\in H^0([E\backslash \overline{G}_w],\Vscr(\chi)^{N_w})$ whose non-vanishing locus is exactly the substack $[E\backslash G_w]$.
\end{theorem}
We call the sections $h_w$ afforded by \Th~\ref{th-intro-gp-hasse} {\em group-theoretical Hasse invariants}. The space $H^0([E\backslash G_w],\Vscr(\chi))$ has dimension $\leq 1$ for any $\chi\in X^*(L)$, hence the sections $h_w$ of \Th~\ref{th-intro-gp-hasse} are unique up to scalar.

Let $f:(G_1,\mu_1)\to (G_2,\mu_2)$ be a finite morphism of cocharacter data and $\widetilde{f}:\GoneZip^{\mu_1}\to \GtwoZip^{\mu_2}$
the induced map of stacks (\S\ref{cochardata}). Put $L_i=\cent(\mu_i)$.

\begin{corollary}[Discrete fibers, \Th~\ref{discrete_fibers}] \label{cor-intro-discrete-fibers} 
Assume there exists a $(p,L_2)$-admissible $\chi\in X^*(L_2)$, whose restriction to $L_1$ is orbitally $p$-close \textnormal{(\Def~\ref{def-orb-p-close})}. Then $\widetilde{f}$ has discrete fibers on the underlying topological spaces.
\end{corollary}

\begin{rmk}\label{rmk-intro-discrete-fibers}
The corollary generalizes \cite[\Cor~3]{Koskivirta-compact-hodge} about the $\mu$-ordinary locus. In \cite[\Th~2]{Goldring-Koskivirta-zip-flags}, we generalized  \Cor~\ref{cor-intro-discrete-fibers} above as follows: We proved that if $f:(G_1,\mu_1)\to (G_2,\mu_2)$ is a morphism with central scheme-theoretic kernel, then $\widetilde{f}$ has discrete fibers. The extra assumption of \Cor~\ref{cor-intro-discrete-fibers} is thus unnecessary.
\end{rmk}

\subsection{The EO stratification of Hodge-type Shimura varieties} \label{sec-intro-EO} For the rest of the introduction, assume $p>2$.
 Suppose $\gx$ is a Shimura datum of Hodge type. Assume $\GG$ is unramified at some prime $p$ and $\Kcal=\Kcal_p\Kcal^p$ is an open compact subgroup of $\gofaf$ with $\Kcal_p \subset \gofqp$ hyperspecial and $\Kcal^p \subset \gofafp$. 
Let $\Shk$ be the Kisin-Vasiu integral model at $p$, of level $\Kcal$, of the associated Shimura variety $\shgx$ (\S\ref{sec-shimura-integral}). Write $\Shko$ for the special fiber of $\Shk$ at a a prime $\pfr$ of the reflex field dividing $p$.

Let $[\mu]$ be the $\GG(\CC)$-conjugacy class of cocharacters deduced from $\XX$. 
By Zhang \cite{ZhangEOHodge}, there is a smooth morphism 
\[\zeta:\Shko \longrightarrow \GZip^{\mu}\] 
where $(G, [\mu])$ is the reduction modulo $p$ of an integral model of $(\GG_{\CC}, [\mu])$ (see \S\S\ref{sec-shimura-root-data-mu},\ref{sec-univ-Gzip}). Let $S_w:=\zeta^{-1}([E\backslash G_w])$ for $w\in {}^I W$.

We first describe the key applications of \Th~\ref{th-intro-gp-hasse} to the EO stratification of $S_{\Kcal}$ and then discuss extensions to compactifications. In particular, note that the following three corollaries are completely independent of the theory of compactifications.

Let $\omega$ be the Hodge line bundle on $\Shko$ associated to a symplectic embedding $\varphi:\gx \to \shdagsp$. Let $\eta_{\omega}$ be the character of $L$ satisfying $\omega=\Vscr(\eta_{\omega})$ (\S\ref{sec-vector-bundle-dictionary}). We have shown \cite[\Th~1.4.4]{Goldring-Koskivirta-quasi-constant} that $\eta_{\omega}$ is quasi-constant (\Def~\ref{def-orb-p-close}\ref{item-def-q-const}). Moreover, $\eta_{\omega}$ is easily seen to be $L$-ample (\Def~N.5.1). A quasi-constant character is orbitally $p$-close for all $p$ (\Rmk~\ref{rmk-cond-char}). Since $(p, L)$-admissible is defined as $L$-ample and orbitally $p$-close (\Def~\ref{def-orb-p-close}), $\eta_{\omega}$ is $(p,L)$-admissible for all $p$. Hence:

\begin{corollary}[\Cor~\ref{cor-hasse-hodge}]
\label{cor-intro-hasse-hodge} Suppose $(G, \mu)$ arises by reduction modulo $p$ from a Hodge-type Shimura datum and $\chi=\eta_{\omega}$. Then the conclusion of \textnormal{\Th~\ref{th-intro-gp-hasse}} holds for all primes $p$.
\end{corollary}

\begin{rmk}\label{intro-rmk-second-proof}
There is another proof of \Cor~\ref{cor-intro-hasse-hodge} that does not use the notion of "quasi-constant characters" and the result of \cite[\Th~1.4.4]{Goldring-Koskivirta-quasi-constant}. By \cite[\Cor~1]{Goldring-Koskivirta-zip-flags}, Corollary \ref{cor-hasse-hodge} holds for a more general class of pairs $(G,\mu)$, namely the "maximal" ones (\loccitn, \S2.4). The proof is based on the discrete fiber theorem (\loccitn, \Th~2).
\end{rmk}

Pulling back the sections of \Cor~\ref{cor-intro-hasse-hodge} along $\zeta$, we obtain Hasse invariants for all EO strata of $\Shko$. 
\begin{corollary}[EO Hasse invariants, \Cor~\ref{cor-hasse-shimura}]
\label{cor-intro-hasse-shimura}
 For every EO stratum $S_w \subset \Shko$, there exists $N_w\geq 1$ and a section $h_w\in H^0(\overline{S}_w,\omega^{N_w})$ whose non-vanishing locus is precisely $S_w$.
\end{corollary}
The sections $h_w$ are also $\gofafp$-equivariant, see \Cor~\ref{cor-hasse-shimura}. 
Since the nonvanishing locus of a section of an ample line bundle on a proper scheme is affine, we deduce:
\begin{corollary}[Affineness, compact case; \Cor~\ref{cor-affine-cpt-main-text}] 
\label{cor-affine-cpt} Assume $\gx$ is a Shimura datum of compact type. Then all EO strata in $\Shko$ are affine.
\end{corollary}

The following results concern extending the EO stratification and its Hasse invariants to compactifications. Let $\Shktoro$ be one of the proper toroidal compactifications of $\Shko$ constructed by Madapusi-Pera  (\S\ref{sec-toroidal-review}). In the next theorem, we do not assume $\Shktoro$ is smooth, but we do require a technical log-integrality assumption related to $\Sigma$, see \S\ref{sec-univ-semi-abel}, \S\ref{sec-G-zip-extension}.

\begin{theorem}[\Th~\ref{th-tor-ext-Gzip}] One has:   \label{th-intro-tor-min}   
\begin{enumerate}[label=(\alph*)]
\item \label{item-intro-zeta-extends}
The map $\zeta$ admits an extension
$\zeta^{\Sigma}_{\Kcal}: \Shktoro \longrightarrow \GZip^{\mu}$.
\item \label{item-intro-hasse-tor} Let $G_w \subset  G$ be an $E$-orbit; 
put $S_w^{\Sigma}=(\zeta^{\Sigma}_{\Kcal})^{-1}([E\backslash G_w])$ and $S_w^{\Sigma,*}=(\zeta^{\Sigma}_{\Kcal})^{-1}([E\backslash \overline{G}_w])$. Then the Hasse invariant $ h_w\in H^0(\overline{S}_w, \omega^{N_w}))$ of 
\textnormal{\Cor~\ref{cor-intro-hasse-shimura}} extends to $h_w^{\Sigma} \in H^0(S_w^{\Sigma,*}, \omega^{N_w})$ with non-vanishing locus $S_w^{\Sigma}$.
\end{enumerate}
\end{theorem}
Furthermore, the extension $\zeta^{\Sigma}_{\Kcal}$ is $\gofafp$-equivariant \eqref{eq-zeta-triangles-toroidal}. Let $S_{\Kcal}^{\min}$ be the minimal compactification of $S_{\Kcal}$ (\S\ref{sec-minimal-review}).
Define $S_w^{\min}$ as the image of $S_w^{\Sigma}$ by the natural map $\Shktoro\to \Shkmino$. We call $S_w^{\min}$ the extended EO strata in $\Shkmino$. Combining \Th~\ref{th-intro-tor-min} with a
 Stein factorization argument, we deduce:

\begin{corollary}[Affineness, noncompact case; \Prop~\ref{prop-min-cpt-strata}] \label{cor-affine-noncpt} Suppose $\gx$ is a Shimura datum of noncompact, Hodge type. Then the extended EO strata $S_w^{\min}$ are affine for all $w \in {}^I W$. 
\end{corollary}
In the Siegel case -- for $\Acal_g \otimes \fp$ with a suitable level structure -- \Cor~\ref{cor-affine-noncpt} was conjectured by Oort almost twenty years ago \cite[14.2]{Oort-stratification-moduli-space-abelian-varieties}. As far as we know, even this special case of \Cor~\ref{cor-affine-noncpt} is new. The affineness of the generic EO stratum for Hodge-type Shimura varieties was proved by the second author and T. Wedhorn (\cite[\Cor~2]{Koskivirta-Wedhorn-Hasse}).

In the restricted special case of Shimura varieties of PEL-type A and C, the thesis of G. Boxer  \cite{Boxer-thesis,Boxer-thesis-arxiv} obtained the Corollaries~\ref{cor-intro-hasse-shimura}, ~\ref{cor-affine-cpt},~\ref{cor-affine-noncpt} and a variant of \Th~\ref{th-intro-tor-min} simultaneously and independently from us.

\subsection{Applications to the Langlands correspondence} \label{sec-auto-to- galois} 
\subsubsection{Galois representations associated to automorphic representations} 
\label{sec-intro-galois-to-autom}
The Langlands correspondence (for number fields) predicts that the distinguished subclass of automorphic representations having integral infinitesimal character -- those termed $L$-algebraic in    \cite{Buzzard-Gee-conjectures} -- satisfy a number of algebraicity properties; see \loccit for some precise conjectures.  In particular, if $F$ is a number field, $G$ is a connected, reductive $F$-group and $\pi$ is an $L$-algebraic automorphic representation of $G$, then it is conjectured that, for every prime $p$, there exists an associated $L$-group-valued, continuous Galois representation \begin{equation} \tag{Gal}
R_{p,\iota}(\pi):\galf \longrightarrow \lg(\qpbar). 
\label{eq-langlands-correspondence}
\end{equation}

Following the pioneering works of Deligne  and Deligne-Serre on classical modular forms \cite{Deligne-Gal-Rep,Deligne-Serre-weight-1}, there has been an increasingly sustained effort by a growing number of people to construct the Galois representations~\eqref{eq-langlands-correspondence}. 
For a discussion and references concerning previous work on the association~\eqref{eq-langlands-correspondence}, \cf the first author's earlier papers \cite{Goldring-Galois-reps-HLDS,Goldring-Vancouver-expository}.
Thus far, most works have been limited to the (weaker) construction of $r \circ R_{p,\iota}(\pi) $, where $r:\lg \longrightarrow GL(n)$ is some low-dimensional representation of $\lg$. 
A notable exception is the recent preprint of Kret-Shin \cite{Kret-Shin-spin-valued-Galois-reps}, which constructs $GSpin(2g+1)$-valued Galois representations for certain $\pi$ of the split symplectic similitude group $GSp(2g)$.

Most of the work to date on~\eqref{eq-langlands-correspondence} has been restricted to cases when the archimedean component $\pi_{\infty}$ is regular. Prior to this work, the limited number of results which considered irregular $\pi_{\infty}$ were all in cases when $\pi_{\infty}$ is a holomorphic limit of discrete series (LDS); this is the mildest possible type of irregularity \cite{Taylor-GSp4,Jarvis-hilbert-LDS,Goldring-Galois-reps-HLDS,Goldring-Nicole-mu-Hasse}. For a detailed classification of archimedean components in terms of~\eqref{eq-langlands-correspondence}, see \cite{Goldring-LDS-nondeg-deg-functoriality}.
\subsubsection{Pseudo-representations associated to torsion} \label{sec-intro-torsion}
Starting with Ash \cite{Ash-mod-p-duke}, a number of {\em torsion} analogues of the Langlands correspondence have been proposed, where automorphic representations are replaced by systems of Hecke eigenvalues appearing in the cohomology (Betti or coherent) of a locally symmetric space with mod $p^n$ coefficients and the Galois representations~\eqref{eq-langlands-correspondence} are replaced by mod $p^n$-valued pseudo-representations. The interest in such ``torsion Langlands correspondences'' has grown considerably due to the pivotal role that they play in the Calegari-Geraghty program of pushing the Taylor-Wiles method beyond the regular case \cite{Calegari-Geraghty-beyond-taylor-wiles}. 

Scholze achieved a breakthrough in the Betti case, by associating pseudo-representations to the cohomology of the locally symmetric spaces of $GL(n)$ over  a $CM$ field, with mod $p^n$ coefficients \cite{Scholze-torsion}. 
By contrast, much less was known prior to this work concerning the coherent cohomology of Shimura varieties mod $p^n$.
The only case where pseudo-representations were associated to higher coherent cohomology (\ie $H^i$, $i>0$) was that of Hilbert modular varieties, due to Emerton-Reduzzi-Xiao \cite{Emerton-Reduzzi-Xiao}.   

\subsubsection{New results about the Langlands correspondence}
\label{sec-intro-new-results-langlands}

In this paper, we prove general results which associate pseudo-representations to the coherent cohomology of Hodge-type Shimura varieties modulo a prime power. Consequently, we deduce results about the existence of $r \circ R_{p,\iota}(\pi)$ when $G$ is a $\QQ$-group admitting a Hodge-type Shimura variety and $\pi_{\infty}$ is an arbitrary non-degenerate LDS. 

Let $\Sscr_{\Kcal}^{\Sigma}$ be a toroidal compactification of a Hodge type Shimura variety $\Sscr_{\Kcal}$ as in \S\ref{sec-intro-EO}. We now add the assumption that $\Sscr_{\Kcal}^{\Sigma}$ is smooth. Write $\Shktorn$ for its reduction mod $p^n$, so that $S_{\Kcal}^{\Sigma}=\Sscr_{\Kcal}^{\Sigma, 1}$ is its special fiber. 
Let $\vsubeta$ be the subcanonical extension to $\Shktor$ of an automorphic vector bundle $\veta$ on $\Shk$ (\S\ref{sec-torsors}). 
Let $\Hcal$ denote the (global, unramified, prime-to-$p$) Hecke algebra of $\GG$ and $\Hcal^{i,n}(\eta)$ its image in $\End (H^i(\Shktorn, \vsubeta))$ (\S\ref{sec-hecke-algebras}).

The following two theorems summarize our results; see \Ths~\ref{th-reduction-to-h0},~\ref{th-torsion-general},~\ref{th-nondeg-lds-general} for the precise statements. Unconditional analogues of \Th~\ref{th-intro-galois} for unitary groups are given in \Ths~\ref{th-torsion-unitary}, \ref{th-nondeg-lds-unitary}.   
\begin{theorem}[Factorization]  \label{th-intro-factor} Let $\veta$ be an automorphic vector bundle on $\Shk$ and $i\geq 0$ an integer. Suppose $\shgx$ is either of compact-type, or of PEL-type, or that $(S_{\Kcal}^{\Sigma}, \veta)$ satisfies \textnormal{Conditions~\ref{cond-zero-dim}, \ref{cond-hdi}}. Then, 
for every $\delta \in \rgeqz$ there exists a $\delta$-regular  weight $\eta'$ \textnormal{(\Def~\ref{def-delta-reg})} such that \[ \label{eq-factor-intro}
\Hcal \twoheadrightarrow \Hcal^{i,n}(\eta) \mbox{ factors through } \Hcal \twoheadrightarrow \Hcal^{0,n}(\eta').
\]
\end{theorem}

For Shimura varieties of PEL type $A$ or $C$, Boxer  independently and simultaneously obtained the weaker statement that $\Hcal \twoheadrightarrow \Hcal^{i,n}(\eta)$ factors through  $\Hcal \twoheadrightarrow \Hcal^{0,n}(\eta+a\eta_{\omega})$ for infinitely many $a$, \cite{Boxer-thesis,Boxer-thesis-arxiv}. In contrast with \Th~\ref{th-intro-factor}, the weights $\eta+a\eta_{\omega}$ may all be singular; see also \Rmk~\ref{rmk-assumptions-in-main-theorem}.
\begin{theorem}
\label{th-intro-galois} Keep the hypotheses of \textnormal{\Th~\ref{th-intro-factor}} and fix $\delta \in \rgeqz$. Assume $r:\lg \to GL(n)$ satisfies that $r \circ R_{p, \iota}(\pi')$ exists for all $C$-algebraic \textnormal{(\S\ref{harish-chandra iso})}, $\delta$-regular $\pi'$. Then 
\begin{enumerate}[label=(\alph*)]
\item \label{item-th-intro-pseudo}
For every triple $(i, n, \eta)$, there exists a continuous Galois pseudo-representation \[R_{p,\iota}(r;i,n, \eta): \galq \longrightarrow \Hcal^{i,n}(\eta),\] which for unramified $v$, $v \neq p$, maps $\frob_v^j$ to the Hecke operators $T_v^{(j)}$ defined in \textnormal{\S\ref{sec-satake}}.
\item \label{item-th-intro-autom} Assume $\pi$ is a cuspidal automorphic representation of $\GG$ with $\pi_{\infty}$ a non-degenerate, $C$-algebraic LDS and $\pi_p$ unramified. Then  $r \circ R_{p, \iota}(\pi)$ exists too.

\end{enumerate}
\end{theorem}

The case of cohomological degree $i>0$ reveals genuinely new obstacles (for both \Th~\ref{th-intro-factor} and \Th~\ref{th-intro-galois}). 
In terms of automorphic representations, $i>0$ corresponds to the condition that the non-degenerate LDS $\pi_{\infty}$ is non-holomorphic. 
When $i>0$, it is apparent that the original Deligne-Serre method of multiplication by a single mod $p$ automorphic form is insufficient.

Our new idea to overcome this problem is to use not just one mod $p$ automorphic form, but rather a whole family of such \viz the strata Hasse invariants $h_w$ and their toroidal extensions $h_w^{\Sigma}$. 
These are `generalized mod $p$ automorphic forms', in that they aren't defined on the whole special fiber of the Shimura variety, but only on EO strata closures. 
As pullbacks of sections on strata of $\GZip^\mu$, the $h_w, h_w^{\Sigma}$ are $\gofafp$-equivariant (\S\ref{sec-univ-Gzip},~\Th~\ref{th-tor-ext-Gzip}).  The reduction to the case $i=0$ is accomplished by studying the long exact sequences in coherent cohomology associated to technical modifications of the $h_w,h_w^{\Sigma}$.

Results similar to \Th~\ref{th-intro-factor} and \Th~\ref{th-intro-galois} were obtained simultaneously and independently by Pilloni-Stroh \cite{Pilloni-Stroh-CoherentCohomology}. Their method is completely different: It is based on Scholze's theory of perfectoid Shimura varieties. The analogue of \Th~\ref{th-intro-galois}\ref{item-th-intro-pseudo} in \loccit concerns the coherent cohomology of Scholze's integral models of $\shgx$ rather than the modular Kisin-Vasiu models studied here. Boxer also announced applications of his thesis to the construction of Galois representations, but as far as we know, no preprint containing such results has appeared.

\subsection{Serre's Letter to Tate on mod $p$ modular forms}

Let $N$ be prime to $p$ and let $X(N)$ (resp. $(X(N)^{\sesi}$) be the modular curve of full level $N$ (resp. its supersingular locus) in characteristic $p$.
In his letter to Tate \cite{Serre-Two-Letters-Modular-Forms}, Serre first showed that the systems of Hecke eigenvalues which appear in $ \bigoplus_{k \in \NN} H^0(X(N), \omega^k)$ are the same as those which  appear in $\bigoplus_{k \in \NN} H^0(X(N)^{\sesi}, \omega^{k})$. He went on to show that these systems of Hecke eigenvalues are also precisely those that appear in Gross' algebraic modular forms for the quaternion algebra $B_{p,\infty}$ ramified at $\{p,\infty\}$.

As a further application of group-theoretical Hasse invariants, we generalize Serre's first result to arbitrary Shimura varieties of Hodge type. Let $D$ be the boundary divisor of $\Sscr_{\Kcal}$ in $\Sscr_{\Kcal}^{\Sigma}$. Let $S_e$ be the (unique) zero-dimensional EO stratum of $\Shko$.

\begin{theorem}[see \Th~\ref{th-serre-letter}]
\label{th-intro-serre-letter} If $\gx$ is neither of compact-type, nor of PEL-type, then assume that there exists a a $\gofafp$-equivariant Cartier divisor $D'$ such that $D'_{\red}=D$ and $\omega^k(-D')$ is ample on $\Sscr_{\Kcal}^{\Sigma}$ for all $k \gg 0$.
As $\eta$ ranges over all weights, the systems of Hecke eigenvalues appearing in each of $\bigoplus_{\eta}H^0(\Shko, \veta)$, $\bigoplus_{\eta}H^0(S_{\Kcal}^{\Sigma}, \vsubeta)$ and $ \bigoplus_{\eta}H^0(S_e, \veta)$ are the same. In particular the number of such systems is finite.
     \end{theorem}
This application lies on the border between \ref{item-algebraicity-Langlands} Automorphic Algebraicity and \ref{item-zips} $G$-Zip Geometricity. The existence of such $D'$ is known in the PEL case by Lan \cite[\Th~7.3.3.4]{Lan-book-thesis} (see also \Rmk~\ref{rmk-cartier-ample} below) and it should follow similarly in the general Hodge case from the work of Madapusi-Pera \cite{MadapusiHodgeTor}.  In future work, we hope to return to Serre's second result in this level of generality. It is likely that it follows from the recent preprint \cite{Xiao-Zhu-geometric-satake}, but we have not checked this in detail.

No cases of \Th~\ref{th-intro-serre-letter} were previously known where the classical superspecial locus of $\Shko$ is empty\footnote{To avoid confusion, we specify that for us the classical superspecial locus means those points whose underlying abelian scheme is isomorphic to a product of supersingular elliptic curves. Some refer to this locus simply as the ``superspecial locus'', while others reserve the same term for the zero-dimensional EO stratum. The latter two conventions are jointly incompatible, already in the PEL case.}. Serre's two results had previously been generalized by Ghitza to Siegel modular varieties \cite{Ghitza-Siegel-Mod-p-Algebraic} and by Reduzzi to a (rather restricted) class of PEL-type Shimura varieties \cite{Reduzzi-PEL-Mod-p}. Unfortunately, these works seem to contain a nontrivial error, see \Rmk~\ref{rmk-Ghitza-mistake}. Both were limited to PEL cases where the classical superspecial locus of $\Shko$ is nonempty. Analogous to the distinction between the classical ordinary locus and the $\mu$-ordinary locus (=unique open EO stratum), examples of both PEL and Hodge-type abound where the classical superspecial locus is empty. For example, if $\gx$ is of unitary type with reflex field $E \neq \QQ$ and $p$ splits completely in $E$, then the classical superspecial locus of $\Shko$ is empty.

Using \Ths~\ref{th-intro-factor} and \Th~\ref{th-intro-serre-letter}, we deduce: 

\begin{corollary}[See \Cor~\ref{cor-serre-letter-finiteness}] Make the assumptions of \textnormal{\Th~\ref{th-intro-factor}}. 
As $i$ varies over all nonnegative integers and $\eta$ varies over all weights, the number of systems of Hecke eigenvalues appearing in $$ \bigoplus_{i, \eta}H^i(\Shktoro, \vsubeta) \oplus H^i(\Shktoro, \vcaneta) \oplus H^0(S_e, \veta)$$  is finite.
\label{cor serre letter} \end{corollary}

\subsection{Outline} \label{sec-outline} Following a preliminary \S N to fix notation, this paper is naturally divided into three parts:
The primary goal of Part~\ref{part-gp-hasse} (\S\S\ref{sec-Gzips}-\ref{sec hasse invariants}) is the construction of group-theoretical Hasse invariants (\Th~\ref{th-intro-gp-hasse}). 
To this end, we introduce the stack of zip flags (\S\ref{sec-zip-flags}). Parts \ref{part-shimura} and \ref{part-galois} are concerned with applications. Part \ref{part-shimura} contains those applications which are directly concerned with the geometry of the EO stratification of Shimura varieties, but which avoid Hecke operators. By contrast, Part~\ref{part-galois} regards three applications involving the Hecke algebra: 
(i) Factorization of the Hecke algebra action on the coherent cohomology of automorphic vector bundles (\S\S\ref{sec-fact-hecke}-\ref{sec-regularity-flag-space}), 
(ii) Association of Galois pseudo-representations to coherent cohomology modulo $p^n$ and Galois representations to automorphic representations of non-degenerate LDS-type (\S\ref{sec-galois-reps}), 
(iii) Generalization of Serre's Letter to Tate (\S\ref{sec-serre-letter}). We refer to the beginning of each section for a more detailed description of its contents. 

\section*{Notation} \label{sec-notation}
\renewcommand{\thesubsection}{{N.\arabic{subsection}}}
\renewcommand{\thesubsubsection}{\thesubsection.\arabic{subsubsection}}

\setcounter{subsection}{0}
\subsection{Ring theory}
\subsubsection{} Throughout, $k$ denotes an algebraically closed field. 
\subsubsection{Adeles}  The adele ring of $\QQ$ is written $\AA$, the finite adeles $\AA_f$. Given a prime $p$, $\AA_f^p$ denotes the finite adeles with trivial $p$-adic component.
\subsubsection{CM fields}
\label{sec-cm-fields}
A number field is a finite extension of $\QQ$. In this paper, `CM field' means a number field which is either totally real or a totally imaginary quadratic extension of a totally real field. 
\subsection{Scheme theory}
\subsubsection{Base Change} \label{sec-base-change-notation}
If $A$ is a commutative ring with $1$, $M$ is an $A$-module and $B$ an $A$-algebra, we write $M_B$ for the $B$-module $M \otimes_A B$.

If $S$ is a scheme, $f:X \to S$ is an $S$-scheme and $g:T \to S$ is a morphism of schemes, write $X_T$ for the base change $X \times_S T$. If $S=\spec(A)$ and $T=\spec(B)$, then we also write $X_B$ in place of $X_T$.
\subsubsection{Frobenius twist} Let $k$ be a field of characteristic $p$, and $\sigma:k\to k$ the map $x\mapsto x^p$. For a $k$-scheme $X$, we denote by $X^{(p)}$ the fiber product $X^{(p)}:=X\otimes_{k,\sigma} k$.

\subsubsection{Reduction modulo $p^n$} 
\label{sec-reduct-mod-pn}
Let $p$ be a prime, $K/\qp$ a finite extension with ring of integers $\Ocal_K$ and maximal ideal $\pfr$. If $\Xcal$ is an $\Ocal_K$-scheme, write $\Xcal^n$ for the base change of $\Xcal$ along $\spec \Ocal_K/\pfr^n \to \spec \Ocal_K$. In particular, $\Xcal^1$ is the special fiber of $\Xcal$.

Occasionally it will be useful to consider $\Xcal^n$ for all $n \in \zgeqo$, as well $\Xcal$ and its generic fiber $\Xcal \otimes_{\Ocal_K}K$. For this purpose we adopt the convention of writing $\Xcal^+:=\Xcal$ and $\Xcal^0:=\Xcal \otimes_{\Ocal_K}K$.

\subsubsection{Sections of line bundles} 
\label{sec-notation-line-bundle}
If $X$ is a scheme, $\Lscr$ is a line bundle on $X$ and $s \in H^0(X, \Lscr)$, write $Z(s)$ for the scheme of zeroes of $s$ in $X$  (\cf \cite[\App~A, C6]{Hartshorne-Alg-Geom}) and $\nonvanish(s)$ for the open subset $X \setminus Z(s)$.

We say $s \in H^0(X, \mathcal L)$ is \emph{injective} if the map of sheaves $\ox \to \mathcal L$ given by multiplication by $s$ is injective.\footnote{Such a section is sometimes referred to as {\em regular}  (\cf \cite[\Def~11.17]{stacks-project}), but we find the term {\em injective} less confusing.} Equivalently $s$ is injective if and only if $Z(s)$ is an effective Cartier divisor in $X$.
\subsubsection{Associated Reduced Scheme} If $X$ is any scheme, $X_{\red}$ will denote the associated reduced subscheme.

\subsection{Structure Theory and Root data}
\label{sec-structure-theory-root-data}
\subsubsection{Root data}
\label{sec-root-data}
Let $G$ be an algebraic $k$-group. Write $X^*(G)$ (resp. $X_*(G)$) for the group of characters (resp. cocharacters) of $G$. If $T$ is a $k$-torus, write $\langle, \rangle$ for the perfect pairing $X^*(T)\times X_*(T)\to \ZZ$.

For a connected, reductive $k$-group $G$ and a maximal torus $T\subset G$, write $\Phi:=\Phi(G,T)$ for the $T$-roots in $G$, $\Phi^{\vee}:=\Phi^{\vee}(G,T)$ for the $T$-coroots in $G$ and $\Rcal\Dcal(G,T)=(X^*(T), \Phi, X_*(T), \Phi^{\vee})$ for the root datum of $T$ in $G$. If $L$ is a Levi subgroup of $G$ containing $T$, set $\Phi_L:=\Phi(L,T)$ (resp. $\Phi_L^{\vee}:=\Phi^{\vee}(L,T)$).

\subsubsection{Based root data} \label{sec-based-root-data}
If $B$ is a Borel subgroup of $G$ containing $T$, define the system of positive roots $\Phi^+ \subset \Phi$ by the condition that $\alpha \in \Phi^+$ when the root group $U_{-\alpha}\subset B$. Put $\Phi^-:=-\Phi^+$. The subset of simple roots in $\Phi^+$ is denoted $\Delta$ and the set of simple coroots is denoted $\Delta^{\vee}$. Write $X^*_+(T)=\{\chi \in X^*(T) | \langle \chi, \alpha^{\vee} \rangle \geq 0 \mbox{ for all } \alpha \in \Delta\}$ for the cone of $\Delta$-dominant characters. Put 
\begin{equation}
\rho:=\frac{1}{2}\sum_{\alpha \in \Phi^+} \alpha \in X^*(T)_{\QQ}.
\end{equation}
If $L$ is a Levi subgroup of $G$, set $\Delta_L:=\Delta \cap \Phi_L$, $\Delta^{\vee}_L:=\Delta^{\vee} \cap \Phi^{\vee}_L$. Write $X^*_{+,L}(T)$ for the cone of $\Delta_L$-dominant characters. Similarly, put $\Phi^+_L:=\Phi^+(L,T)$ and $\Phi^-_L:=\Phi^-(L,T)$.

\subsubsection{Weyl group}
\label{sec-notation-weyl}
Let $W:=W(G,T)$ be the Weyl group of $T$ in $G$. For $\alpha \in \Phi$, let $s_{\alpha}$ denote the root reflection about $\alpha^{\vee}$. Write $\ell:W \to \zgeqz$ for the length function of the Coxeter group $(W, \{s_{\alpha}\}_{\alpha \in \Delta})$. The longest element of $W$ is denoted by $w_0$ and the identity element by $e$.

If $L\subset G$ is a Levi subgroup containing $T$, denote by $W_L\subset W$ the Weyl group of $L$. Write $w_{0,L}$ for the longest element in $W_L$. If $P\subset G$ is a parabolic subgroup, denote its type by $\type(P)\subset \Delta$. It is normalized as follows: If $P$ contains $B$ with Levi subgroup $L$ containing $T$, then $\type(P):=\Delta_L$.

Given $I\subset \Delta$, let $W_I \subset W$ be the subgroup generated by the elements $s_\alpha$ for $\alpha \in I$. Let $w_{0,I}$ denote the longest element of $W_I$. Define ${}^I W$ (resp. $W^I$) as the subset of elements $w\in W$ which are of minimal length in the coset $W_I w$ (resp. $wW_I$). The set ${}^I W$ (resp. $W^I$) is a set of representatives for the quotients $W_I \backslash W$ (resp. $W/W_I$). The longest element of ${}^I W$ (resp. $W^I$) is $w_{0,I}w_0$ (resp. $w_0 w_{0,I}$). For a Levi subgroup $L$, one has
\begin{align}
w_{0,L} \Phi^+_L&\subset \Phi_L^- \label{eq-z-roots1}\\
w\left(\Phi^+\setminus \Phi^+_L\right)&\subset \Phi^+ \quad \textrm{for all }w\in W_L.\label{eq-z-roots2}
\end{align}

\subsubsection{Weyl chambers}
\label{sec-weyl-chamber}
Write $\alpha^{\perp}$ for the hyperplane in $X^*(T)_{\RR}$ orthogonal to $\alpha^{\vee}$.
A Weyl chamber is a connected component $C$ of $X^*(T)_{\RR}\backslash \bigcup_{\alpha \in \Phi}\alpha^{\perp}$; thus $C$ is an open cone for the archimedean topology. Write $\overline{C}$ for the closure of $C$ and $\partial C:=\overline{C}\backslash C$ for its boundary. One says that $\alpha \in \Phi$ (resp. $\alpha^{\vee} \in \Phi^{\vee}$) is $C$-positive if $\langle \chi, \alpha^{\vee} \rangle>0$ for any (equivalently every) $\chi \in C$. The $C$-positive roots (resp. coroots) form a system of positive roots (resp. coroots). 

\subsubsection{Galois action}
\label{sec-galois-action}
Now suppose that $G,T$ are $\kappa$-groups for some subfield $\kappa 
\subset k$. The root datum and Weyl group of $(G, T)$ are by definition those of $(G_k,T_k)$. The group $\gal(k/\kappa)$ acts on both $W$ and $\Rcal \Dcal(G,T)$. The actions of $W$ and $\Gal(k/\kappa)$ on $X^*(T)$ are related by the formula
\begin{equation} \label{varphicirc}
{}^\sigma (w\lambda)={}^\sigma w \ {}^\sigma \lambda
\end{equation}
for all $w\in W$, $\sigma \in \Gal(k/\kappa)$ and $\lambda \in X^*(T)$.

If $B$ is a Borel subgroup of $G_{
\kappa'}$ for some finite extension $\kappa'/\kappa$, the based root datum of $(G, B, T)$ is that of the triple $(G_k, B_k, T_k)$.
\subsubsection{Ramification}
\label{sec-def-ramification}
In the notation of \S\ref{sec-galois-action}, suppose $\kappa$ is a number field. The set of finite places $v$ of $\kappa$ where $G$ ramifies, together with all the infinite places is denoted $\Ram(G)$. Similarly, if $\pi$ is an admissible representation of $G(\AA_{\kappa})$ given by a restricted tensor product of components $\pi_v$, then $\Ram(\pi)$ denotes the set of places $v$ where $\pi_v$ is ramified, together with all the archimedean places.  
%\subsection{Associated Sheaves}
%\label{sec-notn-assoc-sheaves}
%Our choice of positive roots $\Phi^+$ has the advantage that ample line bundles $\Lcal_{G/B}(\lambda)$ on the flag variety $G/B$ correspond precisely to weights $\lambda \in X^*(T)$ which are both dominant and regular.

\subsection{Quotients}
If an algebraic $k$-group $G$ acts on a $k$-scheme $X$, we denote by $[G\backslash X]$ the associated quotient stack. Denote by $\pi:X\to [G\backslash X]$ the natural projection. If $\cdot:X\times G \to X$ is a right action, we define $[X/G]$ as $[G\backslash X]$ for the action $g\star x := x\cdot g^{-1}$.

\subsubsection{Vector bundles on quotient stacks}\label{Gvarstacks}
Let $G$ act on $X$ and $\rho:G\to GL_{n,k}$ be an $n$-dimensional $k$-representation of $G$. There is an associated vector bundle $\Vscr(\rho)$ on the quotient stack $[G\backslash X]$, such that
\begin{equation} \label{sectionquot}
H^{0}\left(\left[G\backslash X\right],\Vscr(\rho)\right)=\{{f:X\to \AA^n} \ | \ f(g\cdot x)=\rho(g)f(x), \ \forall g\in G,x\in X \}.
\end{equation}
\subsection{Conditions on characters} \label{sec-conditions-characters}
Let ($G, B, T)$ as in \S\ref{sec-galois-action} and $L\subset G$ a Levi subgroup containing $T$. Let $I\subset \Delta$ be the type of the parabolic $P=LB$.
\begin{definition} We say that $\chi \in X^*(L)$ is $L$-\emph{ample} if it satisfies $\langle \chi, \alpha^{\vee} \rangle <0$ for all $\alpha \in \Delta \setminus I$.   \label{def-ample-character} \end{definition}
\begin{rmk}\label{rmk-ample} By \cite[II, (4.4)]{jantzen-representations} the following are equivalent: \begin{enumerate}[label=(\alph*)] \item $\chi$ is $L$-ample.
\item The associated line bundle $\Lscr(\chi)$ is anti-ample on $G/P$.
\end{enumerate} 
\label{rmk-ample-character}
\end{rmk}

\begin{definition}
\label{def-orb-p-close}
Let $\chi \in X^*(T)$. For every coroot $\alpha^{\vee}$ satisfying $\langle \chi , \alpha^{\vee} \rangle \neq 0$, put 
\[ \orb(\chi, \alpha^{\vee})=
\left\{ \left.  \frac{|\langle \chi, \sigma \alpha^{\vee} \rangle|}{|\langle \chi, \alpha^{\vee} \rangle |} ~\right| 
\sigma \in   W \rtimes \Gal(k/\kappa)   \right\}. \]
We say that $\chi$ is
\begin{enumerate}[label=(\alph*)] 
\item \label{item-def-orb-p-close} 
\emph{orbitally $p$-close} if  $\max \orb(\chi, \alpha^{\vee})\leq p-1$  for all $\alpha \in \Phi$ with  $\langle \chi, \alpha^{\vee} \rangle \neq 0$.
\item \label{item-def-q-const} 
\emph{quasi-constant} if $\orb(\chi, \alpha^{\vee}) \subset \{0,1\}$ for all $\alpha \in \Phi$ with  $\langle \chi, \alpha^{\vee} \rangle \neq 0$. 
\item 
\label{item-def-p-small} 
\emph{$p$-small} if 
$|\langle \chi, \alpha^{\vee} \rangle | \leq p-1$ for all $\alpha \in \Phi$.
\item
\label{item-def-p-adm}
\emph{$(p,L)$-admissible} if $\chi$ is orbitally $p$-close and $L$-ample.
\end{enumerate}
\end{definition}
\begin{rmk} 
\label{rmk-cond-char} Recall that $\chi \in X^*(T)$ is {\em minuscule} if $\langle \chi, \alpha^{\vee} \rangle \in \{-1,0,1\}$ for all $\alpha \in \Phi$. 
The notions introduced in \Def~\ref{def-orb-p-close} and the minuscule condition satisfy the following relations:
\[ 
\xymatrix@R=.1pc{
 & 
(\mbox{quasi-constant}) 
\ar@{}[rd]|-*[@]{\Longrightarrow} &
 \\ 
(\mbox{minuscule}) 
\ar@{}[ur]|-*[@]{\Longrightarrow}
\ar@{}[dr]|-*[@]{\Longrightarrow}
& & 
\left(\parbox{2.5cm}{\centering orbitally $p$-close \\ for all $p$ } \right)
\\ & 
\left(\parbox{1.5cm}{\centering $p$-small \\ for all $p$ } \right)
\ar@{}[ur]|-*[@]{\Longrightarrow} &
}\]
\end{rmk}
For more on quasi-constant (co)characters, including a general classification and applications, see  \cite{Goldring-Koskivirta-quasi-constant}.
\begin{definition}
\label{def-delta-reg} 
Let $\delta \in \rgeqz$. We say that $\chi \in X^*(T)$ is $\delta$-regular if $|\langle \chi, \alpha^{\vee} \rangle| >\delta$ for all $\alpha \in \Phi$.
\end{definition}
The case $\delta=0$ gives the usual notion of regularity.
\subsection{Hodge structures}
\label{sec-notn-hodge-structures}
Let $\SS:=\rescr \gmc$ be the Deligne torus. An $\RR$-Hodge structure is a morphism of $\RR$-algebraic groups $\SS \to GL(V_{\RR})$, where $V_{\RR}$ is an $\RR$-vector space. 
\subsubsection{Hodge bigrading convention}
\label{sec-notn-bigrading}
An $\RR$-Hodge structure $h: \SS \to GL(V_{\RR})$  induces a bigrading $V_{\CC} = \oplus_{a,b \in \ZZ}V^{a,b}$. Our convention is that of Deligne
\begin{comment}
Reference to Deligne checked
\end{comment}
\cite[1.1.6]{Deligne-Shimura-varieties}: $z \in \SS(\RR)$ acts on $V^{a,b}$ via $z^{-a}\bar z^{-b}$. In particular, a complex structure on $V_{\RR}$ is an $\RR$-Hodge structure of type $\{(0,-1),(-1,0)\}$.
 \subsubsection{Associated cocharacter}
 \label{sec-notn-cocharacter}
 There is a natural isomorphism
$
\SS_\CC \simeq \prod_{\Gal(\CC/\RR)} \gmc.
$
Let $\mu_0 : \gmc \to \SS_\CC$ be the natural injection onto the factor corresponding to $\id \in \Gal(\CC/\RR)$. If $G$ is a connected, reductive $\RR$-group and $h: \SS \to G$ is a morphism of $\RR$-groups, the cocharacter $\mu \in X_*(G)$ associated to $h$ is $\mu:=h_{\CC} \circ \mu_0$.

\renewcommand{\thesubsection}{{\thesection.\arabic{subsection}}}
\renewcommand{\thesubsubsection}{\thesubsection.\arabic{subsubsection}}

\part{Group-theoretical Hasse invariants} \label{part-gp-hasse}

Part \ref{part-gp-hasse} is devoted to group-theoretical objects. Section \S\ref{stack morphisms} recalls basic facts about morphisms of quotient stacks, used throughout this part. In \S\ref{def zip datum}, we recall the definition of the stack of $G$-Zips and introduce the notion of cocharacter data. Those of Hodge-type are singled out in \S\ref{subsection ZDHT}. In \S\ref{subsection strata GZip}, we recall the parametrization and properties of zip strata.

The topic of \S\ref{sec-zip-flags} is the stack of zip flags $\GF^\mu$, it is key for constructing Hasse invariants. The definition is given in \S\ref{sec-def-zip-flags}. We recall the definition and the properties of the `Schubert stack' in \S\ref{subsec-schubert}. It is used to define a stratification of $\GF^\mu$ in \S\ref{sec-stratification-GZipFlag}. Special strata termed `minimal' and `cominimal' are studied in \S\ref{sec-min-comin}.

We apply these tools in \S\ref{sec hasse invariants} to construct group-theoretical Hasse invariants. First, \S\ref{sec line bundles zip flags} studies line bundles on these various stacks. The main result of Part \ref{part-gp-hasse} is \Th~\ref{GTHI}, proved in \S\ref{subsec-group-HI}. We discuss functoriality and deduce \Th~\ref{discrete_fibers} in \S\ref{sec functor}. Finally, we introduce the cone of global sections in \S\ref{subsection global sec cone}; it will be used later in \S\ref{sec-increased-regularity}.

\section{The stack of $G$-zips} \label{sec-Gzips}
\subsection{Morphisms of quotient stacks} \label{stack morphisms}
Let $f:X\to Y$ be a morphism of schemes and let $G,H$ be group schemes such that $G$ acts on $X$ and $H$ acts on $Y$. Let $\alpha : G\times X \to H$ be a map of schemes satisfying
\begin{equation}
f(g\cdot x)=\alpha(g,x)\cdot f(x), \quad g\in G, \ x\in X\label{mos1}
\end{equation}
and the cocycle condition
\begin{equation}
\alpha(gg',x)=\alpha(g,g'\cdot x)\alpha (g',x)  \quad g,g'\in G, \ x\in X.\label{mos2}
\end{equation}
Then $(f,\alpha)$ induces a morphism between the groupoids attached to $(G,X)$ and $(H,Y)$, which yields in turn a map of stacks $\tilde{f}:\left[G\backslash X \right]\to \left[H\backslash Y \right]$, such that the following diagram commutes:
$$\xymatrix{
X \ar[r]^f \ar[d] &  Y \ar[d] \\
\left[ G\backslash X \right]\ar[r]_{\tilde{f}} & \left[ H\backslash Y \right]
}$$
where the vertical maps are the natural projections.

\subsection{The stack of $G$-zips} \label{def zip datum}

\subsubsection{Cocharacter datum} \label{cochardata}
We denote by $k$ an algebraic closure of $\fp$. A \emph{cocharacter datum} is a pair $(G,\mu)$ consisting of a connected reductive $\fp$-group $G$ and a cocharacter $\mu:\GG_{m,k}\to G_k$. We will denote by $\varphi :G\to G$ the Frobenius homomorphism.
\begin{definition}\label{defmorphchardata}
Let $(G_1,\mu_1)$ and $(G_2,\mu_2)$ be two cocharacter data. A morphism $f:(G_1,\mu_1)\to (G_2,\mu_2)$ is a morphism of groups $f:G_1\to G_2$ defined over $\fp$, such that $\mu_2=f\circ \mu_1$.
\end{definition}
We say that $f$ is an embedding (resp. finite) if the underlying map $f:G_1\to G_2$ is an embedding (resp. finite). Cocharacter data form a category.

A cocharacter datum $(G,\mu)$ gives rise to a pair of opposite parabolics $(P^-, P^+)$ in $G_k$ and a Levi subgroup $L:=P^-\cap P^+=\cent(\mu)$. The set $P^+(k)$ consists of those elements $g\in G(k)$ such that the limit 
\begin{equation}\label{limparab}
\lim_{t\to 0}\mu(t)g\mu(t)^{-1}
\end{equation}
exists, i.e such that the map $\GG_{m,k} \to G_{k}$,  $t\mapsto\mu(t)g\mu(t)^{-1}$ extends to a morphism of varieties $\AA_{k}^1\to G_{k}$. The Lie algebra of the parabolic $P^-$ (resp. $P^+$) is the sum of the non-positive (resp. non-negative) weight spaces of the action of $\GG_{m,k}$ on $\Lie(G)$ via $\mu$. We set $P:=P^-$, $Q:=\left( P^+ \right)^{(p)}$ and $M:= L^{(p)}$, so that $M$ is a Levi subgroup of $Q$. We denote by $U$ and $V$ the unipotent radicals of $P$ and $Q$, respectively. For a $k$-scheme $S$, one defines:

\begin{definition}[{\cite[\Def~1.4]{PinkWedhornZiegler-F-Zips-additional-structure}}]
\label{def-G-zip}
A $G$-zip of type $\mu$ over $S$ is a tuple $\underline{I}=(I,I_P,I_Q,\iota)$ where $I$ is a $G$-torsor over $S$, $I_P\subset I$ is a $P$-torsor, $I_Q\subset I$ is a $Q$-torsor, and $\iota: (I_P)^{(p)}/U^{(p)} \to I_Q/V$ an isomorphism of $M$-torsors.
\end{definition}

The category of $G$-zips over $S$ is denoted by $\GZip^{\mu}(S)$. This gives rise to a fibered category $\GZip^{\mu}$ over the category of $k$-schemes, which is a smooth algebraic stack of dimension 0 (\cite[\Th~1.5]{PinkWedhornZiegler-F-Zips-additional-structure}).

\subsubsection{Representation as a quotient stack}
The Frobenius restricts to a map $\varphi:L\to M$. The isomorphisms $L\simeq P/U$ and $M\simeq Q/V$ yield natural maps $P\to L$ and $Q\to M$ which we denote by $x\mapsto \overline{x}$. We define the zip group $E$ as:
\begin{equation}\label{zipgroup}
E:=\{(a,b)\in P\times Q \ | \ \varphi(\overline{a})=\overline{b}\}.
\end{equation}
The group $G\times G$ acts on $G$ by the rule $(a,b)\cdot g:=agb^{-1}$. By restriction, the groups $P\times Q$ and $E$ also act on $G$. By \loccit \Th~1.5, there is an isomorphism of stacks:
\begin{equation}\label{isomGE}
\GZip^{\mu}\simeq \left[E\backslash G \right].
\end{equation}
The association $(G,\mu)\mapsto \GZip^\mu$ is functorial : Let $f:(G_1,\mu_1)\to (G_2,\mu_2)$ be a morphism of cocharacter data and denote by $E_1$ and $E_2$ the associated zip groups. Using \eqref{limparab}, one sees that $f$ naturally induces a morphism of stacks $\left[E_1 \backslash G_1 \right] \to \left[E_2 \backslash G_2 \right]$, and hence a morphism $\GoneZip^{\mu_1}\to \GtwoZip^{\mu_2}$. 

\subsubsection{Conjugation}
For a cocharacter datum $(G,\mu)$ and $g\in G(k)$, the pair $(G,{}^g\mu)$ is again a cocharacter datum. Write $P',Q',E'$ for the groups attached to $(G,{}^g\mu)$ by \S\ref{cochardata}. One has $P'={}^gP:=gPg^{-1}$ and $Q'={}^{\varphi(g)}Q$. The maps $f: G\to G, \ x\mapsto gx \varphi(g)^{-1}$ and $\alpha : E \to E', \ (a,b)\mapsto \left(gag^{-1},\varphi(g)b\varphi(g)^{-1}\right)$ yield (\S\ref{stack morphisms}) an isomorphism:
\begin{equation}\label{conjisom}
\GZip^\mu \simeq \GZip^{{}^g \mu}.
\end{equation}

It will be convenient to make the following assumption:
\begin{Assumption}\label{AssumptionBorel}
There exists a Borel pair $(B,T)$ in $G$ defined over $\fp$ such that $B\subset P$.
\end{Assumption}

If $(G,\mu)$ is an arbitrary cocharacter datum, we can find $g\in G(k)$ such that $(G,{}^g\mu)$ satisfies Assumption \ref{AssumptionBorel}. By \eqref{conjisom}, it is harmless to assume that it is satisfied.

\subsection{Cocharacter data of Hodge type} \label{subsection ZDHT}

Let $(W,\psi)$ be a non-degenerate symplectic space over $\fp$ and $GSp(W,\psi)$ the symplectic group of $(W,\psi)$. Consider a decomposition $\Dcal$ : $W\otimes k=W_+ \oplus W_-$ where $W_+$ and $W_-$ are maximal isotropic subspaces. Define a cocharacter $\mu_\Dcal :\GG_{m,k}\to GSp(W,\psi)$ by letting $x\in \GG_{m,k}$ act trivially on $W_-$ and by multiplication by $x$ on $W_+$. The groups attached to the cocharacter datum $(GSp(W,\psi),\mu_\Dcal)$ are
$$P_\Dcal^+:=\stab_{GSp(W,\psi)}(W_+), \quad P_\Dcal=P^{-}_\Dcal:=\stab_{GSp(W,\psi)}(W_-), \quad Q_\Dcal:=\left(P^{+}_\Dcal\right)^{(p)},\quad L_\Dcal:=P^+_\Dcal \cap P^{-}_\Dcal,\quad M_\Dcal:=\left(L_\Dcal \right)^{(p)}.$$
Since $GSp(W,\psi)$ acts transitively on the set of all decompositions $\Dcal$ of $W$ into maximal isotropic subspaces, all cocharacter data obtained in this way are conjugate. We call $(GSp(W,\psi),\mu_\Dcal)$ a Siegel-type cocharacter datum. Define the \emph{Hodge character} $\eta_\omega:L_\Dcal \to \GG_m$ by $g\mapsto \det(g|W_+)^{-1}$.

\begin{definition} 
\label{defHodge} 
Let $(G,\mu)$ be a cocharacter datum.
\begin{enumerate}[label=(\alph*)]
\item \label{item-def-Hodgetype} We say that $(G,\mu)$ is of Hodge-type if there exists a Siegel-type cocharacter datum $(GSp(W,\psi),\mu_\Dcal)$ and an embedding $\iota : (G,\mu)\to (GSp(W,\psi),\mu_\Dcal)$.
\item \label{item-def-omegazip} Given an embedding $\iota$ as in \ref{item-def-Hodgetype}, we denote again by $\eta_\omega$ the restriction of $\eta_\omega$ to $L$, and call it the Hodge character of $(G, \mu)$ relative to $\iota$.
\end{enumerate}
\end{definition}
{\em A posteriori}, at least when $\iota$ arises from an embedding of Shimura data, the dependence of $\eta_{\omega}$ on $\iota$ is very minimal, see \Rmk~\ref{rmk-dependence-varphi}.
\subsection{Stratification} \label{subsection strata GZip}
Let $(G,\mu)$ be a cocharacter datum, and let $P,Q,L,M,E$ be the attached groups, as defined in \S\ref{def zip datum}. We assume that there exists a Borel pair $(B,T)$ satisfying Assumption \ref{AssumptionBorel}. Denote by $I,J\subset \Delta$ the type of $P,Q$ respectively (as defined in \S\ref{sec-notation-weyl}).

For $w\in W$, choose a representative $\dot{w}\in N_G(T)$, such that $(w_1w_2)^\cdot = \dot{w}_1\dot{w}_2$ whenever $\ell(w_1 w_2)=\ell(w_1)+\ell(w_2)$ (this is possible by choosing a Chevalley system, see \cite{SGA3}, Exp. XXIII, \S6). Define $z:=w_{0}w_{0,J}$. Then one has:
\begin{equation}\label{eqBorel}
{}^z \! B \subset Q \quad \textrm{and} \quad
\varphi(B\cap L) = B\cap M= {}^z \! B\cap M. 
\end{equation}

Note that the triple $({}^z \! B,T,\dot{z}^{-1})$ is a frame\footnote{Note that contrary to \loccitn, we use the convention $B\subset P$. This choice seems more natural from the point of view of Hodge theory and applications to
characteristic 0. It slightly modifies the parametrization of the $E$-orbits, and other results are changed accordingly.}, as defined in \cite[\Def~3.6]{Pink-Wedhorn-Ziegler-zip-data}. For $w\in W$, define $G_w$ as the $E$-orbit of $\dot{w}\dot{z}^{-1}$. The $E$-orbits in $G$ form a stratification of $G$ by locally closed subsets. By \Th~7.5 and \Th~11.2 in \cite{Pink-Wedhorn-Ziegler-zip-data}, the map $w\mapsto G_w$ restricts to two bijections:
\begin{align}\label{param}
{}^I W &\to \{E \textrm{-orbits in }G\} \\
W^J &\to \{E \textrm{-orbits in }G\}
\end{align}
Furthermore, for $w\in {}^I W \cup W^J$, one has
\begin{equation}\label{dimzipstrata}
\dim(G_w)= \ell(w)+\dim(P).
\end{equation}

\section{The stack of zip flags} \label{sec-zip-flags}
\subsection{Definition}\label{sec-def-zip-flags}
Fix a cocharacter datum $(G,\mu)$ and a Borel pair $(B,T)$ satisfying Assumption \ref{AssumptionBorel}.

\begin{definition}
\label{def-G-zip-flag}
A $G$-zip flag of type $\mu$ over a $k$-scheme $S$ is a pair $\hat{I}=(\underline{I},J)$ where $\underline{I}=(I,I_P,I_Q,\iota)$ is a $G$-zip of type $\mu$ over $S$, and $J\subset I_P$ is a $B$-torsor.
\end{definition}

We denote by $\GF^{\mu}(S)$ the category of $G$-zip flags over $S$. By similar arguments as for $G$-zips, we obtain a stack $\GF^{\mu}$ over $k$, which we call the stack of $G$-zip flags of type $\mu$. This stack is (up to isomorphism) independent of the choice of the pair $(B,T)$. There is a natural projection
\begin{equation}
\pi:\GF^{\mu}\to \GZip^\mu.
\end{equation}
Define an action of $E\times B$ on $G\times P$ by
\begin{equation}
((a,b),c)\cdot (g,r):=(agb^{-1},arc^{-1})
\end{equation}
for all $(a,b)\in E$, $c\in B$ and $(g,r)\in G\times P$. The first projections $G\times P \to G$ and $E\times B\to E$ give rise to a map of quotient stacks $\left[(E\times B) \backslash (G\times P)\right] \to \left[E \backslash G\right]$ as explained in \S\ref{stack morphisms}, denote it again by $\pi$.

\begin{theorem}\label{comdiagGF}
There is a commutative diagram, where the vertical maps are isomorphisms:
$$\xymatrix@1@M=5pt@R=0.7cm@C=1.2cm{
\GF^\mu \ar[r]^\pi \ar[d]^\simeq & \GZip^\mu \ar[d]^{\simeq} \\
\left[(E\times B) \backslash (G\times P)\right] \ar[r]^-{\pi} & \left[E \backslash G\right].
}$$
\end{theorem}

\begin{proof}
The proof is similar to that of \cite[\Prop~3.11]{PinkWedhornZiegler-F-Zips-additional-structure}. Let $S$ be a scheme over $k$. To $(g,r)\in (G\times P)(S)$, we attach a "standard" $G$-zip flag $\hat{I}_{(g,r)}=(\underline{I}_{g}, J_{(g,r)})$ over $S$. Let $\underline{I}_g$ be the $G$-zip over $S$ attached to $g$ by Construction 3.4 in \loccit Define $J_{(g,r)}$ as the image of $B\times S \subset P\times S$ under left multiplication by $r$.

For elements $(g,r),(g',r')\in (G\times P)(S)$, define the transporter as
\begin{equation}
\Transp((g,r),(g',r')):=\{(\epsilon,b)\in (E\times  B)(S) \ | \ (\epsilon,b)\cdot (g,r)=(g',r')\}.
\end{equation}

This gives a category $\Ycal$ fibered in groupoids over the category of $k$-schemes, such that for any $k$-scheme $S$, $\Ycal(S)$ is the category whose objects are $(G\times P)(S)$ and for $(g,r), (g',r')\in (G\times P)(S)$, the morphisms $(g,r)\to (g',r')$ are given by $\Transp((g,r),(g',r'))$. This is a prestack, whose stackification is the quotient stack $ \left[(E\times B) \backslash (G\times P)\right]$  (\cite[3.4.3]{Laumon-Moret-champs}).

We claim that $(g,r)\mapsto \hat{I}_{(g,r)}$ is then a fully faithful functor $\Ycal(S) \to \GF^\mu(S)$, i.e. that there is a natural bijection between the set of morphisms $\phi:\hat{I}_{(g,r)}\to \hat{I}_{(g',r')}$ and $\Transp((g,r),(g'r'))$. A morphism $\phi:\hat{I}_{(g,r)}\to \hat{I}_{(g',r')}$ consists in particular of a morphism $\varphi:\underline{I}_g\to \underline{I}_{g'}$. By \cite[\Lem~3.10]{PinkWedhornZiegler-F-Zips-additional-structure}, one can attach to $\varphi$ a pair $\epsilon=(p_+,p_-)\in E(S)$, (using the notation of \loccitn). By compatibility, the map $J_{(g,r)}\to J_{(g',r')}$ must be induced by left multiplication by $p_+$. Since $J_{(g,r)}=r(B\times  S)$ and $J_{(g',r')}=r'(B\times S)$, we have $b:=r'^{-1}p_+r \in B(S)$. We obtain a map $\phi \mapsto (\epsilon , b)$, and it is easy to check that it is a bijection.

To show that we obtain an isomorphism $\GF^\mu \simeq \left[(E\times B) \backslash (G\times P)\right]$, it remains to prove that any $G$-zip flag is \'etale locally of the form $\hat{I}_{(g,r)}$. Let $\hat{I}=(\underline{I},J)$ be a $G$-zip flag over a $k$-scheme $S$, where $\underline{I}=(I,I_P,I_Q,\iota)$ is a $G$-zip. By \cite[\Lem~3.5]{PinkWedhornZiegler-F-Zips-additional-structure}, the $G$-zip $\underline{I}$ is \'etale locally of the form $\underline{I}_g$. Furthermore, we may choose an \'etale extension which trivializes the $B$-torsor $J$. Since $J\subset I_P$ by definition, there exists an element $r\in P(S)$ such that $J$ is the image of $B\times S$ under left multiplication by $r$. Thus $\underline{I}=\hat{I}_{(g,r)}$. This completes the proof of the isomorphism. From our construction, it is clear that the diagram commutes.
\end{proof}

Define a subgroup $E'\subset E$ by
\begin{equation}
E':=E\cap (B\times G)=\{(a,b)\in E \ | \ a\in B\}.
\end{equation}
It is easy to see that $E'\subset B\times {}^z \! B$. The map $G\to G\times P$, $g\mapsto (g,1)$ and the inclusion $E'\subset E$ induce (\S\ref{stack morphisms}) an isomorphism of quotient stacks
\begin{equation}\label{EprimeGF}
\left[E' \backslash G\right]\simeq \left[(E\times B) \backslash (G\times P)\right].
\end{equation}
Note that we can also view this stack as the quotient $[E\backslash \left(G\times \left( P/B \right) \right)]$, where $E$ acts on the scheme $G\times \left( P/B \right)$ by
\begin{equation}\label{Eactionzipflag}
(a,b)\cdot (g,rB):=(agb^{-1},arB)
\end{equation}
for all $(a,b)\in E$ and $rB\in P/B$.

\subsection{The Schubert stack}\label{subsec-schubert}
Let $B\times B$ act on $G$ by $(a,b)\cdot g:=agb^{-1}$ for $a,b\in B$ and $g\in G$. Define the Schubert stack as the quotient stack $\Sch:=\left[(B\times B) \backslash G \right]$. The group $G$ is the disjoint union of the locally closed Schubert cells
\begin{equation}
C_w:=B\dot{w}B, \quad w\in W.
\end{equation}
One has $\dim(C_w)=\dim(B)+\ell(w)$. Denote by $\overline{C}_w$ the Zariski closure of $C_w$ in $G$, endowed with the reduced structure. Note that $\overline{C}_w$ is normal by \cite[\Th~3]{Ramanan-Ramanathan-projective-normality}. For each $w\in W$, define substacks 
\begin{equation}
\Sch_w:=\left[(B\times B) \backslash C_w \right] \quad \textrm{and} \quad \overline{\Sch}_w:=\left[(B\times B) \backslash \overline{C}_w \right].
\end{equation}
For each pair of characters $(\lambda,\mu)\in X^*(T)\times X^*(T)$, we have a line bundle $\Lscr_{\Sch}(\lambda,\mu)$ on $\Sch$ (\S\ref{Gvarstacks}). For $w\in W$, let $E_w$ denote the set of roots $\alpha\in \Phi^+$ such that $ws_\alpha <w$ and $\ell(ws_\alpha) =\ell(w)-1$.

\begin{theorem}\label{brion}
Let $w \in W$. One has the following:
\begin{enumerate}[label=(\alph*)]
\item \label{brion1} $H^0\left(\Sch_w,\Lscr_{\Sch}(\lambda,\mu)\right)\neq 0\Longleftrightarrow \mu = -w^{-1} \lambda$.
\item \label{brion2} $\dim_k H^0\left(\Sch_w,\Lscr_{\Sch}(\lambda,-w^{-1} \lambda) \right)=1$.
\item \label{brion3} For any nonzero $f\in H^0\left(\Sch_w,\Lscr_{\Sch}(\lambda,-w^{-1} \lambda) \right)$ viewed as a rational function on $\overline{C}_w$, one has
\begin{equation}\label{briondiv}
\div(f)=-\sum_{\alpha \in E_w} \langle \lambda , w\alpha^\vee \rangle \overline{C}_{w s_\alpha}.
\end{equation}
\end{enumerate}
\end{theorem}
\begin{proof}
We prove part~\ref{brion1}. Assume $f\in H^0\left(\Sch_w,\Lscr_{\Sch}(\lambda,-w^{-1} \lambda) \right)$ is a nonzero element. Then $f$ identifies with a regular function $f:C_w\to \AA^1$ (necessarily non-vanishing) satisfying the relation
\begin{equation}
f(bxb')=\lambda(b) \mu(b')^{-1}f(x)
\end{equation}
for all $b,b'\in B$ and for all $x\in C_w$. In particular, for all $t\in T$, we have
\begin{equation}
f(\dot{w}t)=\mu(t)^{-1} f(\dot{w})=\lambda(\dot{w}t\dot{w}^{-1})f(\dot{w})
\end{equation}
and the result follows (note that $f(\dot{w})\neq 0$). Part~\ref{brion2} follows by \cite[\Prop 1.18]{Koskivirta-Wedhorn-Hasse}. Part~\ref{brion3} is Chevalley's formula, see for example \cite[\S1, p. 654]{brion-lakshmibai-monomial}. The minus sign in formula \eqref{briondiv} accounts for our convention of positivity of roots (\S\ref{sec-based-root-data}).
\end{proof}

% As you said Chevalley's formula is for simply connected groups, so the root expressions should be for the simply connected covering of the derived subgroup. This is perfect, because it is easy to see that for a simply connected group, the characters of a Levi corresponding to a special node in the Dynkin diagram admit a minuscule, ample basis. %

\subsection{Schubert stratification of $\GF^\mu$} \label{sec-stratification-GZipFlag}
\Th~\ref{comdiagGF} and \eqref{EprimeGF} give an isomorphism $\GF^\mu \simeq \left[ E' \backslash G \right]$. It follows from \eqref{eqBorel} that $E'\subset B \times {}^z \! B$, so we have a natural projection $\beta:\left[ E' \backslash G \right] \to \left[ (B \times  {}^z \!B)\backslash G \right]$. The map $G\to G$, $x\mapsto x\dot{z}$ induces an isomorphism of quotient stacks $\alpha_z:\left[(B\times {}^z\!B) \backslash G\right]\xrightarrow{\simeq} \left[(B\times B) \backslash G\right]=\Sch$. We obtain a smooth morphism of stacks
\begin{equation}
\psi:\GF^{\mu} \longrightarrow \Sch, \quad \psi:=\alpha_z \circ \beta.
\end{equation}

For $w\in W$, define the flag stratum $\Xcal_w$ by $\Xcal_w:=\psi^{-1}(\Sch_w)$, and the closed flag stratum $\overline{\Xcal}_w:=\psi^{-1}(\overline{\Sch}_w)$, endowed with the reduced structure. The flag strata $\Xcal_w$ are smooth and locally closed. For $w\in W$, define a corresponding locally closed subvariety of $G\times (P/B)$ by
\begin{equation}
H_w:=\{(g,hB)\in G\times (P/B) \ | \ \widetilde{\psi}(g,h) \in C_w \}\end{equation}
where $\widetilde{\psi}(g,h):=h^{-1}g\varphi(\overline{h})\dot{z}$. The action of $E$ on $G\times (P/B)$ defined by \eqref{Eactionzipflag} restricts to an $E$-action on $H_w$ and one has an identification $\Xcal_w=\left[E \backslash H_w \right]$.

\begin{lemma}\label{lemma flag strata} \
\begin{enumerate}[label=(\alph*)]
\item \label{flag strata 1} The closed flag strata are normal and irreducible.
\item \label{flag strata 2}  The closed flag strata coincide with the closures of the flag strata.
\item \label{flag strata 3}  For all $w\in W$, one has $\dim (H_w)=\ell (w)+\dim (P)$.
\end{enumerate}
\end{lemma}

\begin{proof}
Since $\overline{C}_w$ is normal and $\psi$ is smooth, $\overline{\Xcal}_w$ is normal. The smooth morphism $\widetilde{\psi} : G\times P\to G$ has all fibers isomorphic to $P$. Since $C_w$ is irreducible and $\widetilde{\psi}$ is open with irreducible fibers, $\widetilde{\psi}^{-1}(C_w)\subset G\times P$ is irreducible. Hence $H_w$ is irreducible. Finally, the last two assertions follow from the smoothness of $\psi$.
\end{proof}

\subsection{Minimal, cominimal strata}\label{sec-min-comin}
Denote by $\tilde{\pi}:G\times P/B \to G$ the first projection. The map $\tilde{\pi}$ is $E$-equivariant for the action of $E$ described in \eqref{Eactionzipflag}. We have a Cartesian square:

$$\xymatrix@1@M=5pt{
G\times (P/B) \ar[d] \ar[r]^-{\tilde{\pi}} & G \ar[d] \\
[E\backslash \left(G\times (P/B)\right)] \ar[r]^-{\pi} & [E\backslash G]}$$

We introduce the map $\tilde{\pi}$ to reduce some proofs to scheme theory. For $w\in W$, the irreducible locally closed subset $\tilde{\pi}(H_w)$ is $E$-stable, so it is a union of $E$-orbits in $G$.

\begin{lemma}
For all $w\in W$, $\tilde{\pi}(H_w)$ is the union of $E$-orbits intersecting $B\dot{w}\dot{z}^{-1}$.
\end{lemma}

\begin{proof}
It is clear that $\tilde{\pi}(H_w)$ is the union of $E$-orbits intersecting $C_w\dot{z}^{-1}$. One has $C_w\dot{z}^{-1}=B( \dot{w}\dot{z}^{-1}) {}^z\!B$, so any such $E$-orbit intersects $B\dot{w}\dot{z}^{-1}$ and conversely.
\end{proof}

\begin{definition}
We call ${}^I W$ (resp. $W^J$) the set of minimal (resp. cominimal) elements of $W$ \textnormal{(\S\ref{sec-notation-weyl})}.
\end{definition}

Similarly, the flag stratum $\Xcal_w$ is called minimal (resp. cominimal) if $w$ is minimal (resp. cominimal).
Note that the longest minimal (resp. cominimal) element is $w_{0,I}w_0$ (resp. $w_0w_{0,J}$). The identity element $e$ is both the shortest minimal and cominimal element.

\begin{proposition}
\label{prop-image-min-comin}
Assume $w \in W$ is either minimal or cominimal. Then:
\begin{enumerate}[label=(\alph*)]
\item \label{mincomin1} One has $\tilde{\pi}(H_w)=G_w$ and $\tilde{\pi}(\overline{H}_w)=\overline{G}_w$.
\item \label{mincomin2} The $E$-action on $H_w$ is transitive.
\item \label{mincomin3} The preimage of $G_w$ by the morphism $\tilde{\pi}:\overline{H}_w\to \overline{G}_w$ is exactly $H_w$.
\item \label{mincomin4} The map $\tilde{\pi}:H_w\to G_w$ is finite.
\end{enumerate}
\end{proposition}

\begin{proof}
The first part of \ref{mincomin1} follows from \cite[\Th~5.14]{Pink-Wedhorn-Ziegler-zip-data} for $w$ minimal and from \cite[\Th~11.3]{Pink-Wedhorn-Ziegler-zip-data} for $w$ cominimal (with appropriate modifications due to our choice of frame $({}^zB,T,\dot{z}^{-1})$). The second part follows from the properness of $\tilde{\pi}$. To show \ref{mincomin2}, note that $\dim(H_w)=\dim(G_w)$ (Lemma \ref{lemma flag strata}\ref{flag strata 3} and \eqref{dimzipstrata}). Since the map $\tilde{\pi}:H_w\to G_w$ is $E$-equivariant and $G_w$ is an $E$-orbit, all its fibers are isomorphic. In particular, it is quasi-finite. Hence $H_w$ contains finitely many $E$-orbits. But if $Z\subset H_w$ is an $E$-orbit, it must map surjectively onto $G_w$ (as $G_w$ is an $E$-orbit). We deduce $\dim(Z)\geq \dim(G_w)=\dim(H_w)$. Hence any $E$-orbit of $H_w$ has dimension $\dim(H_w)$, which clearly shows that $H_w$ is an $E$-orbit. To prove \ref{mincomin3}, assume there exists $y\in \overline{H}_w\setminus H_w$ such that $\tilde{\pi}(y)\in G^w$. Let $w'\in W$ such that $y\in H_{w'}$. Then $H_{w'}\subset \overline{H}_w$ and $\pi(\overline{H}_{w'})=\overline{G}_w$. But this is impossible since $\dim H_{w'} < \dim(H_w)=\dim(G_w)$. Finally, \ref{mincomin3} shows that $\pi:H_w\to G_w$ is proper and quasi-finite, so it is finite, which proves \ref{mincomin4}.
\end{proof}

\begin{rmk}
It is proved in \cite[\Prop~2.2.1(ii)]{Koskivirta-Normalization} that the map $\tilde{\pi}:H_w\to G_w$ is also \'{e}tale.
\end{rmk}

\begin{corollary} \label{cor unique minimal cominimal}
For all $E$-orbits $S\subset G$, there is a unique minimal stratum $H$ and a unique cominimal stratum $H'$ such that $\pi(H)=\pi(H')=S$.
\end{corollary}

\section{Hasse invariants} \label{sec hasse invariants}

\subsection{Line bundles} \label{sec line bundles zip flags}

\begin{enumerate}
\item Identify $X^*(E)=X^*(P)=X^*(L)$ via the first projection $E\to P$ and the inclusion $L\subset P$. By \S\ref{Gvarstacks}, a character $\lambda\in X^*(L)$ gives rise to a line bundle $\Vscr(\lambda)$ on the quotient stack $\GZip^\mu\simeq\left[E \backslash G\right]$.

\item Similarly, identify $X^*(E')=X^*(B)=X^*(T)$ using the first projection $E'\to B$ and the inclusion $T\subset B$. For a character $\lambda\in X^*(T)$, we obtain by \S\ref{Gvarstacks} a line bundle $\Lscr(\lambda)$ on $\GF^\mu\simeq [E'\backslash G]$. 
\end{enumerate}

We first describe the relation between the line bundles $\Lscr_{\Sch}(\lambda, \mu)$, $\Lscr(\lambda)$ (for $\lambda,\mu \in X^*(T)$) and $\Vscr(\lambda)$ (for $\lambda \in X^*(L)$) via the maps
$$\xymatrix@1{
\GF^{\mu} \ar[r]^-{\psi} \ar[d]_{\pi} & \Sch \\
\GZip^{\mu} &
}$$

\begin{lemma} \  \label{pullbacks}
\begin{enumerate}[label=(\alph*)]
\item \label{item-PB1} For all $\lambda\in X^*(L)$, one has $\pi^*\Vscr(\lambda)=\Lscr(\lambda)$.
\item \label{item-PB2} For $\lambda,\nu \in X^*(T)$, one has
\begin{equation} \label{pbpsi}
\psi^*\Lscr_{\Sch}(\lambda,\nu)=\Lscr(\lambda+p \ {}^\sigma(z\nu)),
\end{equation}
where $\sigma:k\to k$ denotes the inverse of the map $x\mapsto x^p$.
\end{enumerate}
\end{lemma}

\begin{proof}
Part \ref{item-PB1} is clear as $\pi$ identifies with the natural projection $[E' \backslash G] \to [E \backslash G]$. Hence the pullback of $\pi^*\Vscr(\lambda)$ is the line bundle attached to the restriction of $\lambda$ to $E'\subset E$, which is $\Lscr(\lambda)$. To show part \ref{item-PB2}, recall (\S \ref{sec-stratification-GZipFlag}) that $\psi$ identifies with the natural projection $[E' \backslash G] \to [B\times {}^zB \backslash G]$ followed by the isomorphism $\alpha_z: [B\times {}^zB \backslash G] \to [B\times B \backslash G]$. Hence $\psi^{*}\Lscr_{\Sch}(\lambda,\nu)$ is the line bundle attached to the restriction of $(\lambda, \nu)$ via the inclusion $$E' \subset B\times {}^zB \simeq B\times B.$$
This restriction is given by $(a,b)\mapsto \lambda(a)(z\nu)(b)$ for $(a,b)\in E'$. Since $(a,b)$ satisfies $\varphi(\overline{a})=\overline{b}$ (where $\overline{a}, \overline{b}\in T$ are the torus components of $a,b$), one has
$\lambda(a)(z\nu)(b)=\lambda(\overline{a}) ({}^\sigma(z\nu))^p(\overline{a})$. This shows $\psi^*\Lscr_{\Sch}(\lambda,\nu)=\Lscr(\lambda+p \ {}^\sigma(z\nu))$.

\end{proof}

In view of \Th~\ref{brion}\ref{brion1}, we now restrict ourselves to line bundles of the form $\Lscr_{\Sch}(\lambda, \nu)$ where $\nu= -w^{-1}\lambda$ for $w\in W$. For $w\in W$ and $r\geq 1$, define $w^{(0)}=e$ and by induction $w^{(r)}:={}^\sigma (w^{(r-1)}w)$ for all $r\geq 1$.

\begin{lemma} \ \label{lemmapower}
\begin{enumerate}[label=(\alph*)]
\item\label{item-power-formula} For all $r,s\geq 1$ and $w\in W$, one has ${}^{\sigma^s}(w^{(r)})w^{(s)}=w^{(r+s)}$.
\item\label{item-exists-r} The set $R:=\{r\geq 0 \ | \ w^{(r)}=e\}$ is a non-trivial submonoid of $\ZZ_{\geq 0}$.
\item \label{item-powerminus1} If $w^{(r)}=e$ for $r\geq 1$, then $w^{(r-1)}=w^{-1}$.
\end{enumerate}
\end{lemma}
\begin{proof}
The first part follows from an easy induction. Hence $R$ is stable under addition. Since $W$ is finite, there exists $r>s\geq 0$ such that $w^{(r)}=w^{(s)}$. By \ref{item-power-formula}, we have $w^{(r-s)}=e$. Finally, \ref{item-powerminus1} is clear from the definition.
\end{proof}

By Lemma \ref{pullbacks}\ref{item-PB2}, we have $\psi^*\Lscr_{\Sch}(\lambda, -w^{-1}\lambda)=\Lscr(\lambda - p \ {}^\sigma \! (zw^{-1}\lambda))$. In the next section, we will use \Th~\ref{brion} to produce sections of certain line bundles $\Lscr_{\Sch}(\lambda, -w^{-1}\lambda)$ and then pull back these sections along the map $\psi$. Hence, we need to understand the map 
\begin{equation} \label{Dwmap}
D_w: X^*(T) \to X^*(T), \quad \lambda \mapsto \lambda - p \ {}^\sigma \! (zw^{-1}\lambda).
\end{equation}

\begin{lemma}\label{lemmainv} Let $w\in W$.
\begin{enumerate}[label=(\alph*)]
\item \label{item-lemmainv1} The map $D_w$ induces a $\QQ$-linear automorphism of $X^*(T)_\QQ$.
\item \label{item-lemmainv2} The inverse of $D_w$ is given as follows: Let $\chi\in X^*(T)$. Fix $r\geq 1$ such that $(zw^{-1})^{(r)}=e$ and let $m\geq 1$ such that $\chi$ is defined over $\FF_{p^m}$. Then one has $D_w(\lambda)=\chi$ for the quasi-character
\begin{equation}\label{lambdapullback}
\lambda =-\frac{1}{p^{rm}-1} \sum_{i=0}^{rm-1} p^i (zw^{-1})^{(i)} ({}^{\sigma^i}\chi)
\end{equation}\label{formula pullback}
Furthermore, the summand corresponding to $i=rm-1$ is $wz^{-1}({}^{\sigma^{-1}}\chi)$.
\end{enumerate}
\end{lemma}

\begin{proof}
Assertion \ref{item-lemmainv1} is clear, as $D_w$ is the identity modulo $p$. To show \ref{item-lemmainv2}, let $\lambda \in X^*(T)_\QQ$ be the unique quasi-character such that $D_w(\lambda)=\chi$. Then one shows by induction on $j\geq 1$ that
\begin{equation}
 \sum_{i=0}^{j-1} p^i (zw^{-1})^{(i)} ({}^{\sigma^i}\chi) = \lambda- p^j (zw^{-1})^{(j)} ({}^{\sigma^j}\lambda)
\end{equation}
For $j=rm$, we have $(zw^{-1})^{(j)}=e$ and ${}^{\sigma^j}\lambda=\lambda$, and the result follows.
\end{proof}

%((\varphi(z)^{-1}w^{-1})^{(rm-1)} z^{-1}\chi)\circ \varphi^{rm-1}=(\varphi^{-1}(w)\chi)\circ \varphi^{rm-1}=w(\chi\circ \varphi^{rm-1}).

\subsection{Group-theoretical Hasse invariants} \label{subsec-group-HI}
By \cite[\Prop~1.18]{Koskivirta-Wedhorn-Hasse}, the space $H^0(\left[E\backslash G_w\right],\Vscr(\chi))$ is at most $1$-dimensional for all $\chi\in X^*(L)$ and all $w\in {}^I W$. Moreover,  \cite[\Th~3.1]{Koskivirta-compact-hodge} shows that there exists $N\geq 1$ such that for all $w\in {}^IW$ and all $\chi \in X^*(L)$, the space $H^0(\left[E\backslash G_w\right],\Vscr(N\chi))$ has dimension $1$ over $k$. Fix the following:
\begin{itemize}
\item $r\geq 1$ such that $w^{(r)}=e$ for all $w\in W$.
\item $m\geq 1$ such that $T$ splits over $\FF_{p^m}$.
\item $N\geq 1$ such that $\dim_k H^0(\left[ E\backslash G_w\right],\Vscr(N\chi))= 1$ for all $w\in {}^IW$ and $\chi \in X^*(L)$.
\item  For $w\in {}^I W$ and $\chi\in X^*(L)$, let $h_{w,\chi}$ be a nonzero element of the line $H^0(\left[ E\backslash G_w\right],\Vscr(N\chi))$.
\item Similarly, for all $w\in W$ and $\lambda \in X^*(T)$, let $f_{w,\lambda}$ be a nonzero element of the line $H^0(\Sch_w,\Lscr_{\Sch}(\lambda,-w^{-1}\lambda))$ (see \Th~\ref{brion}).
\item For $w\in W$ and $\lambda \in X^*(T)$, set $f'_{w,\lambda}:=\psi^*(f_{w,\lambda})\in H^0(\Xcal_w,\psi^*\Lscr_{\Sch})$.
\end{itemize}
 
\begin{proposition} \label{prop HI}
Let $w\in {}^I W\cup W^J$ and $\chi \in X^*(L)$. The following assertions are equivalent:
\begin{enumerate}[label=(\alph*)]
\item \label{item HIzip} There exists $d\geq 1$ such that $h_{w,\chi}^d$ extends to $\left[ E\backslash \overline{G}_w\right]$ with non-vanishing locus $\left[ E\backslash G_w\right]$.
%\item\label{item HIflag}  The section $f'_{w,\chi}$ extends to $\overline{\Xcal}_w$ with non-vanishing locus $\Xcal_w$.
\item\label{item HIformula}  For all $\alpha \in E_w$, one has:
\begin{equation}\label{eqpropHI}
\sum_{i=0}^{rm-1} \langle(zw^{-1})^{(i)}({}^{\sigma^i}\chi),w\alpha^\vee\rangle p^i >0.
\end{equation}
\end{enumerate}
\end{proposition}

\begin{proof}
Let $\lambda\in X^*(T)_\QQ$ given by \eqref{lambdapullback} of Lemma \ref{lemmainv}. Put $C:=p^{rm}-1$, so that $C\lambda \in X^*(T)$. It satisfies $\psi^*(\Lscr_{\Sch}(C\lambda,-w^{-1}C\lambda))=\Lscr(C\chi)$, so we have $f'_{w,C\lambda}\in H^0(\Xcal_w,\Lscr(C\chi))$. Equation (\ref{eqpropHI}) is equivalent to $\langle \lambda, w\alpha^\vee \rangle >0$. Hence by \Th~\ref{brion}~\ref{brion3}, Property \ref{item HIformula} is satisfied if and only if $f_{w,C\lambda}$ extends to $\overline{\Sch}_w$ with non-vanishing locus $\Sch_w$. Equivalently, it holds if and only if $f'_{w,C\lambda}$ extends to $\overline{\Xcal}_w$ with non-vanishing locus $\Xcal_w$ (by smoothness of $\psi$). By \Prop~\ref{prop-image-min-comin}~\ref{mincomin1} combined with \cite[\Prop~ 1.18]{Koskivirta-Wedhorn-Hasse}, we have $\dim_k H^0(\Xcal_w,\Lscr(dC\chi)) \leq 1$, so $\pi^*(h_{w,\chi})^{dC}=(f'_{w,C\lambda})^{d}$ up to a nonzero scalar. Since $\overline{\Xcal}_w$ is normal, $f'_{w,C\lambda}$ extends to $\overline{\Xcal}_w$ with non-vanishing locus $\Xcal_w$ if and only if some power of it does, so we deduce that \ref{item HIformula} is equivalent to the fact that $\pi^*(h_{w,\chi})$ extends to $\overline{\Xcal}_w$ with non-vanishing locus $\Xcal_w$. In particular, \ref{item HIzip} implies \ref{item HIformula}. The converse follows from the next lemma.
\end{proof}

\begin{lemma}\label{descent lemma}
Let $f:X\to Y$ a proper surjective morphism of integral schemes of finite-type over $k$. Let $\Lscr$ be a line bundle on $Y$. Let $U\subset Y$ be a normal open subset and $h\in \Lscr(U)$ a non-vanishing section over $U$. Assume that the section $f^*(h)\in H^0(f^{-1}(U),f^*\Lscr)$ extends to $X$ with non-vanishing locus $f^{-1}(U)$. Then there exists $d \geq 1$ such that $h^d$ extends to $Y$, with non-vanishing locus $U$.
\end{lemma}

\begin{proof}
We will reduce to the case when $X,Y$ are affine, $\Lscr=\Ocal_Y$ and $f$ is the normalization of $Y$.

First, after replacing $X$ by its normalization $\widetilde{X}\to X$, we may assume that $X$ is normal. In this case, the map $f$ factors through the normalization $\pi:\widetilde{Y}\to Y$. Hence there exists $f':X\to \widetilde{Y}$ such that $f=\pi \circ f'$. The map $f'$ is again proper, so it is surjective. Write $\div(\pi^*(h))=\sum_{i=1}^r n_iZ_i$ where $n_i\in \ZZ$ and $Z_i\subset \widetilde{Y}\setminus\pi^{-1}(U)$ are codimension one irreducible subvarieties. If $n_i<0$ for some $i$, then $f^*(h)$ would have a pole (because $f'$ is surjective), hence $\pi^*(h)$ extends to $\widetilde{Y}$ with non-vanishing locus $\pi^{-1}(U)$. Thus we may assume $X=\widetilde{Y}$. Also, we may reduce to the case $Y=\spec(A)$, $X=\spec(B)$ for an integral domain $A$ and $B$ its integral closure, and $\Lscr=\Ocal_Y$. Write $I\subset A$ for the ideal sheaf of $Z:= X\setminus U$, endowed with the reduced structure. Hence
$U=D(I)=\{\pfr\in \spec A, I\nsubseteq \pfr\}$.
Replacing $h$ by a power, we may further assume that $h\in IB$.

We claim that for any $s\in IB$, there exists $n\geq 1$ such that $s^{p^n}$ lies in $A$. Since $k$ has characteristic $p$, we may assume that $s=gx$ for some $g\in I$ and $x\in B$, because the $p^n$ power map is additive. Since $U$ is normal, $f^{-1}(U)\to U$ is an isomorphism, so the map $A_g\to B_g$ is an isomorphism (since $D(g)\subset U$). Hence we can find $m\geq 1$ such that $g^m x \in A$. Since $A[x]$ is generated as an $A$-module by $1,x,...,x^r$  for some $r\geq 1$, it follows that we can find $m$ such that $g^m x^d\in A$ for all $d\geq 0$. Increasing $m$, we may assume that it is a power of $p$, say $m=p^n$. Taking $d=m$ gives $g^m x^m \in A$, which proves the claim. We have showed that there exists $d\geq 1$ such that $h^d\in A$. If $V\subset Y$ is the non-vanishing locus of $h^d$, then $f^{-1}(V)=f^{-1}(U)$, hence $U=V$. 
\end{proof}

Now we come to the main theorem of this part, the existence of group-theoretical Hasse invariants (\Th~\ref{th-intro-gp-hasse}):

\begin{theorem}\label{GTHI}
If $\chi\in X^*(L)$ is $(p,L)$-admissible \textnormal{(\Def~\ref{def-orb-p-close}\ref{item-def-p-adm})}, then there exists $d\geq 1$ such that for all $w\in W^J$, the section $h_{w,\chi}^d$ extends to $\left[ E\backslash \overline{G}_w\right]$ with non-vanishing locus $\left[ E\backslash G_w\right]$.
\end{theorem}

\begin{proof}
We show Property \ref{item HIformula} of \Prop~\ref{prop HI}. This expression has the form $\sum_{i=0}^{rm-1}\langle \chi, w_i \alpha^\vee \rangle p^i$ for some elements $w_i\in W$. Using the inequality
$$\sum_{i=0}^{rm-2}(p-1)p^i=p^{rm-1}-1 < p^{rm-1}$$
and the fact that $\chi$ is orbitally $p$-close, it suffices to check that the leading term $p^{rm-1}\langle \chi, w_{rm-1} \alpha^\vee \rangle$ is positive. By the second part of Lemma \ref{lemmainv}, this term is
$$p^{rm-1}\langle wz^{-1} ({}^{\sigma^{-1}}\chi), w\alpha^\vee \rangle=p^{rm-1}\langle \chi, {}^\sigma(z\alpha^\vee) \rangle.$$
Since $w\in {W^J}$ and $\alpha\in E_w$, we must have $s_\alpha \notin W_J$ by definition of $W^J$. The parabolic subgroup ${}^{z^{-1}}Q$ contains $B$ and has type $J$ and Levi subgroup ${}^{z^{-1}}M$. It follows that $\alpha$ is not a root of ${}^{z^{-1}}M$. By \eqref{eq-z-roots2}, $w_{0,J}\alpha$ is a positive root, hence $z\alpha=w_0 w_{0,J}\alpha$ is negative. We deduce $z\alpha \in \Phi^-\setminus \Phi(M,T)$, and finally
\begin{equation}
{}^\sigma(z\alpha)\in \Phi^-\setminus \Phi(L,T)
\end{equation}
Since $\chi$ is ample, we obtain $\langle \chi, {}^\sigma(z\alpha^\vee) \rangle>0$, which terminates the proof.
\end{proof}

\begin{rmk}
When $G_w$ is the open stratum of $G$, \Th~\ref{GTHI} was proved in \cite{Koskivirta-Wedhorn-Hasse} for $\chi$ any $L$-ample character. The assumption that $\chi$ is orbitally $p$-close is superfluous in this case.
\end{rmk}

\begin{lemma}
Let $(G,\mu)=(GSp(W,\psi),\mu_\Dcal)$ be a cocharacter datum of Siegel-type \textnormal{(\S \ref{subsection ZDHT})}. The Hodge character $\eta_\omega$ is $L$-ample and quasi-constant. In particular, $\eta_\omega$ is $(p,L)$-admissible for all $p$ \textnormal{(\Def~\ref{def-orb-p-close})}.
\end{lemma}

\begin{proof}
We prove first that $\eta_\omega$ is $L$-ample. Choose a Borel subgroup $B\subset P_\Dcal$ and a maximal torus $T\subset B$. Since $P_\Dcal$ is a maximal parabolic subgroup, the set $\Delta \setminus I$ consists of a single simple root $\alpha$. It is then easy to see from the definition of $\eta_\omega$ that $\langle \eta_\omega, \alpha^\vee\rangle =-1$, which shows that $\eta_\omega$ is $L$-ample.

We now show that $\eta_\omega$ is quasi-constant. For any positive root $\beta$ not in $\Phi(L_\Dcal,T)$, one has $|\langle \eta_\omega, \beta^\vee \rangle|=1$ if $\beta$ is a long root, and $|\langle \eta_\omega, \beta^\vee \rangle|=2$ if $\beta$ is a short root. Since $W\rtimes \Gal(k/\FF_p)$ preserves the length of roots, we deduce that $\orb(\eta_\omega,\beta^\vee)=\{0,1\}$ in all cases, which proves the result.
\end{proof}

\begin{corollary}\label{cor-Siegel-zip-Hasse}
Let $(G,\mu)=(GSp(W,\psi),\mu_\Dcal)$ be a cocharacter datum of Siegel-type. There exists $d\geq 1$ such that for all $w\in W^J$, there exists a section $h_w \in H^0(\left[E\backslash \overline{G}_w\right], \omega^d)$ with non-vanishing locus $\left[ E\backslash G_w\right]$.
\end{corollary}

We will explain later in  \S\ref{sec-hasse-shimura} that $\eta_\omega$ is $(p,L)$-admissible for all character data $(G,\mu)$ of Hodge-type coming from a Shimura datum of Hodge-type. In \cite{Goldring-Koskivirta-zip-flags}, we have shown the analogue of \Cor~\ref{cor-Siegel-zip-Hasse} for any $(G,\mu)$ of Hodge-type (not necessarily attached to a Shimura datum), and also for more general $(G,\mu)$ (those of maximal type).

\subsection{Functoriality of zip strata}\label{sec functor}
Let $f:(G_1,\mu_1)\to (G_2,\mu_2)$ be a finite morphism of cocharacter data and $\widetilde{f}:\GoneZip^{\mu_1}\to \GtwoZip^{\mu_2}$ the induced map of stacks. The underlying topological spaces of $\GoneZip^{\mu_1}$ and $\GtwoZip^{\mu_2}$ are finite, and the map $\widetilde{f}$ is continuous.

We denote by $P_1,L_1,E_1$ (resp. $P_2,L_2,E_2$) the subgroups attached to $\mu_1$ (resp. $\mu_2$), as in \S\ref{cochardata}. Let $f^*:X^*(L_2)\to X^*(L_1)$ denote the induced morphism.

\begin{theorem}[Discrete Fibers] \label{discrete_fibers}
Assume there exists a $(p,L_2)$-admissible character $\chi\in X^*(L_2)$, such that $f^*\chi$ is orbitally $p$-close. Then the map $\widetilde{f}$ has discrete fibers on the underlying topological spaces.
\end{theorem}

\begin{proof}
We claim first that $f:G_1\to G_2$ induces a finite morphism $G_1/P_1\to G_2/P_2$. It suffices to show that this morphism is quasi-finite. By definition, one has $P_1\subset f^{-1}(P_2)$, so it follows that $H:=f^{-1}(P_2)_{\rm red}$ is a parabolic subgroup of $G_1$ containing $P_1$. Using \eqref{limparab}, we find also $R_u(P_1)\subset f^{-1}(R_u(P_2))$ hence $R_u(P_1)\subset f^{-1}(R_u(P_2))_{\rm red}=R_u(H)$. It follows that $H=P_1$, so the map $G_1/P_1\to G_2/P_2$ is injective on $k$-points.

In particular, \Rmk~\ref{rmk-ample} shows that the $L_2$-ample character $\chi$ restricts to an $L_1$-ample character of $X^*(L_1)$ which we denote again by $\chi$. Hence we may apply \Th~\ref{GTHI} to $\chi$ on both stacks $\GoneZip^{\mu_1}$ and $\GtwoZip^{\mu_2}$. If the statement were false, there would exist two $E_1$-orbits $C_1$ and $C'_1$ in $G_1$ such that $C'_1\subset \overline{C}_1$ and $f(C_1)\subset C_2$, $f(C'_1)\subset C_2$ for some $E_2$-orbit $C_2\subset G_2$. There exists an integer $N\geq 1$ and a section $h_i\in H^0(\overline{C}_i, \Vscr(\chi)^{ N})$ (for $i=1,2$), whose non-vanishing locus is $C_i$. Since $\dim_k H^0(C_1, \Vscr(\chi)^{ N})=1$, one has $f^*(h_2)=h_1$ (up to nonzero scalar). This contradicts the fact that the non-vanishing locus of $h_1$ is $C_1$.
\end{proof}

We will see that \Th~\ref{discrete_fibers} applies in the context of embeddings of Shimura data of Hodge-type (\Cor~\ref{disc-fib-Shimura}). As explained in \Rmk~\ref{rmk-intro-discrete-fibers} of the introduction, the statement remains true without the cumbersome assumption that there exists
a $(p,L_2)$-admissible character $\chi\in X^*(L_2)$, such that $f^*\chi$ is orbitally $p$-close. This is proved in \cite[\Th~2]{Goldring-Koskivirta-zip-flags}. The trick is to replace the Frobenius $\varphi$ by higher powers of $\varphi$.

\subsection{The cone of global sections} \label{subsection global sec cone}
We will construct nonzero global sections for certain line bundles $\Lscr(\lambda)$. Actually, not all line bundles $\Lscr(\lambda)$ admit nonzero global sections. In general, it is hard to characterize the set of $\lambda \in X^*(T)$ which do.

In this section, $(G,\mu)$ is a cocharacter datum of Hodge-type. Recall that we defined a map $D_w:X^*(T)\to X^*(T)$ in \eqref{Dwmap}. Here, we consider the special case $w=w_0$. Hence, $D_{w_0}$ is defined by
\begin{equation}
\label{eq-def-map-cone}
D_{w_0}:X^*(T)\to X^*(T), \quad \lambda\mapsto \lambda-p ({}^{\sigma}(zw_0\lambda)).
\end{equation}
By \Th~\ref{brion}\ref{brion1}, the line bundles of the form $\Lscr_{\Sch}(\lambda,-w_0\lambda)$ are the ones admitting nonzero sections on the open stratum $\Sch_{w_0}\subset \Sch$. Among these line bundles, those admitting nonzero sections on $\Sch$ are given by the following lemma:

\begin{lemma}
One has $$H^0(\Sch, \Lscr_{\Sch}(\lambda,-w_0\lambda))\neq 0 \Longleftrightarrow -\lambda \in X_+^*(T).$$
\end{lemma}
\begin{proof}

By \Th~\ref{brion}\ref{brion3}, the divisor of any nonzero element $f\in H^0(\Sch_{w_0}, \Lscr_{\Sch}(\lambda,-w_0\lambda))$ (viewed as a rational function on $G$) is a sum of $B\times B$-orbits in $G$ with multiplicities $\langle \lambda, w_0 \alpha^\vee \rangle$ with $\alpha \in E_{w_0}$. One has $E_{w_0}=\Delta$, hence $f$ extends to $G$ if and only if $\langle \lambda, w_0 \alpha^\vee \rangle \geq 0$ for all positive roots $\alpha$. Since $w_0$ maps bijectively the positive roots to the negative ones, this is equivalent to $-\lambda \in X^*_+(T)$.
\end{proof}

This justifies the introduction of the following subset: Define $\Ccal_{w_0}\subset X^*(T)$ by:
\begin{equation}
\Ccal_{w_0}:=\{D_{w_0}(\lambda) \ | \ -\lambda \in X^*_+(T)\}.
\end{equation}

By Lemma \ref{pullbacks}\ref{item-PB2}, one has
$$\psi^*(\Lscr_{\Sch}(\lambda,-w_0\lambda))=\Lscr(D_{w_0}(\lambda)).$$
For any $\lambda \in -X_+^*(T)$, the pull back by $\psi:\GF^\mu \to \Sch$ of a nonzero global section of $\Lscr_{\Sch}(\lambda,-w_0\lambda)$ over $\Sch$ gives a nonzero global section of $\Lscr(D_{w_0}(\lambda))$ over $\GF^\mu$. Hence any $\chi \in \Ccal_{w_0}$ satisfies that $\Lscr(\chi)$ admits a nonzero global section.

\begin{lemma}
The subset $\Ccal_{w_0}\subset X^*(T)$ is a cone of maximal rank (i.e $\Span_\QQ(\Ccal_{w_0})=X^*(T)_\QQ$).
\end{lemma}

\begin{proof}
Since $D_{w_0}$ is linear and $-X^*_+(T)$ is a cone, it is clear that $\Ccal_{w_0}$ is a cone. It is of maximal rank because $X_{+}^*(T)$ is of maximal rank and $D_{w_0}$ induces an automorphism of $X^*(T)_\QQ$ (Lemma \ref{lemmainv}\ref{item-lemmainv1}).
\end{proof}

We assumed that the cocharacter datum $(G,\mu)$ is of Hodge-type. Choose a symplectic embedding $\iota:(G,\mu)\to (GSp(W,\psi),\mu_\Dcal)$. This embedding yields a Hodge character $\eta_\omega \in X^*(L)$ (\Def~\ref{defHodge}\ref{item-def-omegazip}). Define a second cone $\Ccal \subset X^*(T)$ as follows: First, fix an integer $N\geq 1$ such that $H^0([E \backslash G_{w_0}],\omega^N)\neq 0$ (as we did in \S \ref{subsec-group-HI}). Then put:
\begin{equation}
\label{eq-flag-cone}
\Ccal := \NN (N\eta_\omega) + \Ccal_{w_0}
\end{equation}
the cone generated by $N\eta_\omega$ and $\Ccal_{w_0}$. We introduce this cone because of the following proposition:

\begin{proposition}
For any $\chi \in \Ccal$, there exists a nonzero global section $h_\chi \in H^0(\GF^{\mu},\Lscr(\chi))$.
\end{proposition}

\begin{proof}
We already showed this for $\chi \in \Ccal_{w_0}$. Hence it suffices to show it for $N\eta_\omega$, i.e that $\omega^N$ admits a nonzero global section. The embedding $\iota$ induces a map of stacks $\GZip^\mu \to \GspZip^{\mu_\Dcal}$. The image of the unique open stratum of $\GZip^\mu$ maps to some stratum $S\subset \GspZip^{\mu_\Dcal}$ (not necessarily open). In particular, the image of $\GZip^\mu$ is contained in $\overline{S}$, by continuity.

By \Cor~\ref{cor-Siegel-zip-Hasse}, there exists a section $h_S\in H^0(\overline{S}, \omega^d)$ (some $d\geq 1$) with non-vanishing locus $S$. The pull-back of $h_S$ to $\GZip^\mu$ is then nonzero. This gives a nonzero global section $h$ of $\omega^d$ over $\GZip^\mu$. Now we claim that $\omega^N$ also admits a section. For this, let $f\in H^0([E \backslash G_{w_0}],\omega^N)\neq 0$ be nonzero. Then $f^d$ and $h^N$ are both sections of $\omega^{Nd}$ over the open stratum $[E \backslash G_{w_0}]$, hence coincide up to a nonzero scalar, as $\dim_k H^0([E \backslash G_{w_0}], \omega^{Nd})=1$. We deduce that $f^d$ (viewed as a rational function on $G$) extends to $G$. Since $G$ is normal, $f$ extends to $G$, which shows that $\omega^N$ also has a nonzero global section over $\GZip^\mu$, as claimed. This terminates the proof.
\end{proof}

If $G$ is split over $\FF_p$ then one has $N\eta_\omega \in \Ccal_{w_0}$, hence $\Ccal=\Ccal_{w_0}$, but we don't know if this remains true in general. To summarize, the above results provide nonzero global sections for $\Lscr(\chi)$ for $\chi$ in a cone $\Ccal \subset X^*(T)$ of maximal rank, which we tried to make as large as possible. This will be used later in \Th~\ref{th-reduction-to-h0}\ref{item-cone-maximal}.
%\begin{lemma}
%Assume $G$ is split over $\FF_p$. Then $\eta_\omega \in \Ccal_{w_0}$.
%\end{lemma}
%\begin{proof}
%Using the notations of Lemma \ref{lambdapullback}, we may take $r=2$ and $m=1$ for the element $w=w_0$. Formula \eqref{formula pullback} shows that $\psi^*(\lambda)=\eta_\omega$ with 
%\begin{equation}
%\lambda:=-\frac{1}{p^2-1}\left(p w_{0}\eta_\omega + w_0 w_{0,L} \eta_\omega \right).
%\end{equation}
%Hence it suffices to show that $p \langle \eta_\omega, \alpha^\vee\rangle+\langle \eta_\omega, w_{0,L}\alpha^\vee\rangle\leq 0$ for all $\alpha\in \Delta$. If $\alpha \in {}^z I$, then this expression is $0$ because $w_{0,L}$ stabilizes ${}^z I$ and $\langle \eta_\omega, \alpha^\vee\rangle=0$. If $\alpha \in \Delta \setminus {}^zI$, we have $\langle \eta_\omega, \alpha^\vee \rangle<0$. By theorem \ref{thm quasicst}, $\eta_\omega$ is quasi-constant, so the result follows.
%\end{proof}

\part{Strata Hasse invariants of Shimura varieties} 
\label{part-shimura}
In Part~\ref{part-shimura}, we apply the general results of Part~\ref{part-gp-hasse} on group-theoretical Hasse invariants to Hodge-type Shimura varieties and subschemes related to their EO stratification. 
\S\ref{sec-basic-objects-of-study} concerns Hodge-type Shimura varieties. 
\S\ref{sec-shimura-varieties-hodge-type} introduces notation for integral models of Hodge-type Shimura varieties, to be used throughout Parts~\ref{part-shimura}-\ref{part-galois}. 
\S\ref{sec-hasse-shimura} deduces the corollaries of \S\ref{sec-intro-gp-hasse} about Hasse invariants for the EO stratification. 

In \S\ref{sec-eo-compactification}, we study the extension of the EO stratification to compactifications. In \S\ref{sec-G-zip-extension}, we show that the universal $G$-Zip over $\Shko$ admits an extension to a $G$-Zip over a toroidal compactification $\Shktoro$. This is applied to the minimal compactification in \S\ref{sec-eo-minimal-cpt} where we prove the affineness statement of \Cor~\ref{cor-affine-noncpt}. The notion of length stratification is introduced in \S\ref{sec-length-strata} for a general scheme $X \to \GZip^{\mu}$. In \S\ref{sec-length-hodge-type}, the general considerations of \S\ref{sec-length-strata} are applied to $\Shktoro$. The length stratification will play a key role in \S\ref{sec-fact-hecke}, in the proof of the factorization theorem for Hecke algebras (\Th~\ref{th-intro-factor}, \Th~\ref{th-reduction-to-h0}).

The last two sections of Part~\ref{part-shimura} record auxiliary results that will be used in Part~\ref{part-galois}. In \S\ref{sec-vanish-strata}, we note how vanishing theorems for $\Shktor$ generalize to strata. \S\ref{sec-lifting-gluing} describes results about lifting and gluing powers of EO Hasse invariants. A key technical point in our arguments is working with the Cohen-Macaulay property to avoid embedded components. 

\section{Shimura varieties of Hodge type} 
\label{sec-basic-objects-of-study}
\subsection{Background} 
\label{sec-shimura-varieties-hodge-type}
\subsubsection{Rational theory} \label{sec-shimura-rational-theory}

Let $\gx$ be a Shimura datum \cite[2.1.1]{Deligne-Shimura-varieties}. Write $E=\egx$ for the reflex field of $\gx$ and $\Ocal_E$ for its ring of integers. 
Given an open compact subgroup $\Kcal \subset \gofaf$, write $\shgx_{\Kcal}$ for Deligne's canonical model at level $\Kcal$ over $E$ (see \loccitn). Every inclusion $\Kcal' \subset \Kcal$ induces a finite \'{e}tale projection $\pi_{\Kcal'/\Kcal}:\shgx_{\Kcal'} \to \shgx_{\Kcal}$. 
Let $\shgx$ be the resulting tower of $E$-schemes\footnote{Of course $\shgx_{\Kcal}$ will only be a stack for certain $\Kcal$; this does not matter for our purposes}. It admits a right $\gofaf$-action given by a compatible system of isomorphisms $g: \shgx_{\Kcal} \stackrel{\sim}{\to} \shgx_{g^{-1}\Kcal g}$ for $g \in \gofaf$.     
Define $d$ to be the common dimension of all the $\shgx_{\Kcal}$.

\subsubsection{Symplectic embedding}

Let $g\geq 1$ and let $(V,\psi)$ be a $2g$-dimensional, non-degenerate symplectic space over $\QQ$. Write $GSp(2g)=GSp(V, \psi)$ for the group of symplectic similitudes of $(V,\psi)$. Write $\XX_g$ for the double Siegel half-space \cite[1.3.1]{Deligne-Shimura-varieties}. The pair $\shdagsp$ is the Siegel Shimura datum; it has reflex field $\QQ$. Recall that $\gx$ is of Hodge type if there exists an embedding of Shimura data $\varphi:\gx \hookrightarrow \shdagsp$ for some $g \geq 1$. Henceforth, assume $\gx$ is of Hodge-type.

\subsubsection{Integral model} 
\label{sec-shimura-integral}

For the rest of this paper, fix a prime $p \not \in \Ram(\GG) \cup \{2\}$ (\S\ref{sec-def-ramification}). Let $\Gcal$ be a reductive, $\ZZ_{(p)}$-model of $\GG$ and $\Kcal_p:=\Gcal(\ZZ_p) \subset \gofqp$ the associated hyperspecial subgroup. 

Let $\pfr$ be a prime of $E$ above $p$ and let $\Ocal_{E,\pfr}$ be the localization of $\Ocal_E$ at $\pfr$. Write $E_\pfr$ for the completion of $E$ at $\pfr$ and $\Ocal_{\pfr}$ for its ring of integers. By Vasiu \cite[\Th~0]{Vasiu-Preabelian-integral-canonical-models} and 
Kisin \cite[\Th~1]{Kisin-Hodge-Type-Shimura}, as $\Kcal^p$ ranges over sufficiently small open compact subgroups of $G(\AA_f^p)$, the sub-tower of $E$-schemes $(\shgx_{\Kcal^p\Kcal_p})_{\Kcal^p}$ admits an integral canonical model $(\Sscr_{\Kcal^p\Kcal_p})_{\Kcal^p}$ with $\gofafp$-action over $\Ocal_{E,\pfr}$ in the sense of Milne \cite{Milne-integral-canonical-models}. 

For short, say that $\Kcal \subset \gofaf$ is a $p$-hyperspecial level when $\Kcal$ is an open, compact subgroup of $\gofafp$ of the form $\Kcal=\Kcal^p\Kcal_p$ with $\Kcal_p \subset \gofqp$ hyperspecial and $\Kcal^p \subset \gofafp$. The projections between levels
$\pi_{\Kcal'/\Kcal}:\Sscr_{\Kcal'} \to \Sscr_{\Kcal}$ and the right $\gofafp$-action
$g:\Sscr_{\Kcal} \to \Sscr_{g^{-1}\Kcal g}$ are denoted the same way as for the canonical model (\S\ref{sec-shimura-rational-theory}).

\subsubsection{Integral symplectic embedding}  \label{sec-integral-symplectic-embed} Let $\varphi:\gx \hookrightarrow \shdagsp$ be an embedding of Shimura data.
By \cite[2.3.1]{Kisin-Hodge-Type-Shimura}, there is a $\ZZ_{(p)}$-lattice $\Lambda \subset V$ such that $\GG \to GL(V)$ extends to a closed embedding of $\ZZ_{(p)}$-group schemes $\Gcal \to GL(\Lambda)$. By Zarhin's trick, we may assume that $\psi:\Lambda \times \Lambda \to \QQ$ is a perfect pairing with values in $\ZZ_{(p)}$ and that $\varphi$ extends to an embedding of $\ZZ_{(p)}$-group schemes
\begin{equation}
\label{eq-integral-symplectic-embedding}
\varphi:\Gcal \to GSp(\Lambda, \psi).
\end{equation}

Fix such $(\varphi, \Lambda)$ for the rest of the paper. Let $\tilde \Kcal_p:=GSp(\Lambda, \psi)(\ZZ_p)$, a hyperspecial subgroup of $GSp(V,\psi)(\QQ_p)$. For $\tilde \Kcal=\tilde \Kcal_p \tilde \Kcal^{p}$, write $\Shgk$ for the integral canonical model of $\Sh(GSp(2g), \XX_g)_{\tilde \Kcal}$ over $\ZZ_{(p)}$.

For $\Kcal^p$ sufficiently small, there exists $\tilde \Kcal^p\subset GSp(2g, \AA_f^p)$ open compact such that $\varphi(\Kcal^p)\subset \tilde \Kcal^p$ and there is a finite morphism of $\Ocal_{E,\pfr}$-schemes:
\begin{equation}
\label{integral-model-map}
\varphi^{\Sh}:\Shk \to \Sscr_{g, \tilde \Kcal}\otimes_{\ZZ_{(p)}} \Ocal_{E,\pfr}.
\end{equation}

\subsubsection{The cocharacter $\mu$} 
\label{sec-shimura-root-data-mu} 
Given $h \in \XX$, let $\mu \in X_*(\GG)$ be the associated minuscule cocharacter,  given by $\mu(z)=h_\CC \circ \mu_0$ (\S\ref{sec-notn-cocharacter}). The reflex field $E$ of $\gx$ is the field of definition of the $\gofc$-conjugacy class $[\mu]$ of $\mu$. 

For any algebraically closed extension $K/E$, the conjugacy class $[\mu]$ defines a unique conjugacy class $[\mu]_K$ of cocharacters of $\GG_K$ defined over $E$. In particular, we obtain a conjugacy class $[\mu]_{\overline{E}_\pfr}$ for a choice of an algebraic closure $\overline{E}_\pfr$ of $E_\pfr$. Since $\GG_{\QQ_p}$ is unramified, it is in particular quasi-split. Thus, there exists a representative $\mu \in X_*(\GG_{E_{\pfr}})$ of $[\mu]_{\overline{E}_\pfr}$ defined over $E_\pfr$.

Define $h_g \in \XX_g$ and $\mu_g \in X_*(GSp(2g))$ by $h_g=\varphi \circ h$ and $\mu_g:=\varphi \circ \mu$. The cocharacters $\mu, \mu_g$ determine parabolic subgroups, as explained in \S\ref{cochardata}. Specifically, we get the following subgroups:

\begin{enumerate}
\item A pair of opposite parabolic subgroups $(\PP^-,\PP^+)$ in $\GG_{E_\pfr}$ attached to $\mu_{E_\pfr}$, and a common Levi subgroup $\LL:=\PP^-\cap \PP^+=\cent(\mu_{E_\pfr})$. We set $\PP:=\PP^-$.
\item A pair of opposite parabolic subgroups $(P_g^-, P_g^+)$ in $GSp(2g)_{E_\pfr}$ attached to $\mu_{g, E_\pfr}$, and a common Levi subgroup $L_g:=P_g^-\cap P_g^+=\cent(\mu_{g,E_\pfr})$. We set $P_g:=P_g^-$. One has $\PP^\pm=P^\pm_g\cap \GG$ and $L_g \cap \GG=\LL$.
\end{enumerate}

\begin{comment}
Fix an embedding $\varphi:\gx \to \shdagsp$ of Shimura data, where $GSp(2g)=GSp(V,\psi)$ and $(V,\psi)$ is a symplectic space over $\QQ$.
\end{comment}

Fix a Borel pair $(\BB, \TT)$ in $\GG_{\QQ_p}$ such that $\BB_{E_\pfr}\subset \PP$. Write $I\subset \Delta$ for the type of $\PP$. Set $\BB_\LL=\BB \cap \LL$.

\subsubsection{Compatibility with the complex theory}
\label{sec-shimura-compatible-complex}
In \S\ref{sec-galois-reps}, we will fix an isomorphism $\iota: \qpbar \stackrel{\sim}{\to} \CC$. We choose $\iota$ compatibly with $\PP$ so that 
the $\CC$-parabolic  $\iota\PP_{\qpbar}$ deduced via $\iota$ is the stabilizer of the Hodge filtration associated to the $\RR$-Hodge structure $\Ad \circ h$.  
 Consequently, given $\alpha \in \Phi\setminus \Phi_{\LL}$, one has $\alpha \in \Phi^+$ if and only if the image of $\alpha$-root space of $\Lie(\GG)_{\CC}$ is non-zero in $\Lie(\GG)_{\CC}/\Lie(\iota \PP_{\qpbar})$, which is also the $(-1,1)$-part of the Hodge structure $\Ad \circ h$ (conventions \ref{sec-based-root-data}, \ref{sec-notn-bigrading}).

\subsubsection{Integral structure theory}
\label{sec-integral-structure-theory}
\begin{comment}
The conjugacy class $[\mu]$ over $E_\pfr$ extends to one of cocharacters of $\Gcal$; we may choose a representative $\mu$ defined over $\Ocal_{\pfr}$ since $\Gcal$ is quasi-split. 
This defines a model $\Lcal$ of $\LL$ over $\Ocal_{\pfr}$. 
\end{comment}

Since $\Gcal_{\ZZ_p}$ is also quasi-split, we may assume that the representative $\mu$ of $[\mu]_{\overline{E}_\pfr}$ extends to a cocharacter
$\mu: \GG_{m,\Ocal_{\pfr}}\to \Gcal_{\Ocal_{\pfr}}$.
The centralizer of $\mu$ is a Levi subgroup $\Lcal\subset \Gcal_{\Ocal_\pfr}$. Define $\mu_g : \GG_{m,\Ocal_{\pfr}} \to GSp(\Lambda,\psi)_{\Ocal_\pfr}$ by $\mu_g:=\varphi \circ \mu$. Since $\SS \stackrel{h_g}{\to} GSp(V_{\RR}, \psi) \to GL(V_{\RR})$ is an $\RR$-Hodge structure of type $\{(0,-1),(-1,0)\}$ (\S\ref{sec-notn-bigrading}),
the cocharacter $\mu_g$ defines a $\ZZ$-grading $\Lambda\otimes \Ocal_\pfr = \Lambda _0 \oplus \Lambda _{-1}$ (where $\GG_m$ acts through $\mu_g$ by $z\mapsto z^{-i}$ on $\Lambda_i$ for $i=0,-1$).

Since $\zp$ and $\Ocal_{\pfr}$ are discrete valuation rings, \cite[\Chap XXVI, 3.5]{SGA3} and the valuative criterion of properness applied to the schemes of parabolic subgroups of $\Gcal_{\Ocal_\pfr}$ imply that $\PP$ extends uniquely to a parabolic subgroup $\Pcal\subset \Gcal_{\Ocal_\pfr}$. Similarly, there is a unique extension of $\BB$ to a Borel subgroup $\Bcal \subset \Gcal_{\ZZ_p}$. Set $\Bcal_\Lcal:=\Lcal\cap \Bcal$, it is the unique extension of the Borel subgroup $\BB_\LL$ of $\LL$.

\subsubsection{Hodge bundles} 
\label{sec-hodge-line-bundle} 
Let $\Acal_g/\Shgk$ be the universal abelian scheme over $\Shgk$. 
Define the Hodge line bundle $\omega(\varphi)$ on $\Shko$ associated to $\varphi$  as in \cite[\Def~5.1.2]{MadapusiHodgeTor}:
In the Siegel case, set \begin{equation}
\label{def omega_g} 
\Omega_g=\fil^1H^1_{\dR}(\Acal_g/\Shgk) \mbox{ and } \omega_g =\det \Omega_g,  \end{equation} 
where $\fil^1$ refers to the Hodge filtration. Formation of $H^1_{\dR}(\Acal_g/\Sscr_g,\tilde{\Kcal})$ and its Hodge filtration $\Omega_g$ commutes with arbitrary base change, see \S\ref{sec-hdg-de-rham} where a more difficult version is explained for toroidal compactifications.

In general put $H^1_{\dR}(\Acal/\Shk):=\varphi^*H^1_{\dR}(\Acal_g/\Shgk)$,
$\Omega=\varphi^*\Omega_g$ and $\omega=\varphi^*\omega_g$   
(since the embedding $\varphi$ was fixed in \S\ref{sec-integral-symplectic-embed},  we omit $\varphi$ from the notation).

\begin{rmk}[Dependence on $\varphi$]
\label{rmk-dependence-varphi}
{\em A priori} both $\Omega$ and its determinant $\omega$ depend on the embedding $\varphi$. {\em A posteriori} it follows from \Th~\ref{th-quasi-const} below that when $\GG^{\ad}$ is $\QQ$-simple, the ray generated by $\omega$ in $\Pic(\Shk)$ is independent of the embedding $\varphi$, see \cite[\Cor~1.4.5]{Goldring-Koskivirta-quasi-constant}. 
However, this independence of $\varphi$ is not used in this paper. 
Embeddings of products show that the above invariance is best possible. 
It remains unclear to us to what extent the Hodge vector bundle $\Omega$ depends on $\varphi$.
\end{rmk}

\subsubsection{Torsors and automorphic bundles} \label{sec-torsors} 
Following \cite{Kisin-Hodge-Type-Shimura,MadapusiHodgeTor}, we recall how the vector bundle $H^1_{\dR}(\Acal/\Shk)$ yields a $\Gcal$-torsor $I_{\Gcal}$ on $\Shk$. We follow \cite{Kisin-Hodge-Type-Shimura} in the use of $H^1_{\dR}$, while \cite{MadapusiHodgeTor} uses its dual $H_{1,\dR}$. This choice seems easier to relate to the universal $G$-Zip over $\Shko$, see \S\ref{sec-G-zip-extension}.   
By \cite[1.3.2]{Kisin-Hodge-Type-Shimura}, the group $\Gcal$ is the pointwise stabilizer of a finite collection of tensors $(s_{\alpha}) \subset \Lambda^{\otimes}$.  
By 2.3.9 of \loccit (see also \cite[5.3.1]{MadapusiHodgeTor}), there are associated sections $(s_{\alpha, \dR}) \subset H^0(\Shk, \varphi^* H^1_{\dR}(\Acal_g/\Shgk)^{\otimes})$.  
Then \begin{equation}
\label{eq-def-G-torsor-1}
I_{\Gcal}:=\Isomcal_{\Shk}((H^1_{\dR}(\Acal/ \Shk), (s_{\alpha, \dR})), (\Lambda, (s_{\alpha})) \otimes \Shk)
\end{equation} is a $\Gcal$-torsor on $\Shk$. 
By 5.3.4 of \loccitn, the subsheaf $I_{\Pcal}\subset I_{\Gcal}$ consisting of isomorphisms which map the Hodge filtration of $H^1_{\dR}(\Acal/ \Shk)$ to $\Lambda_0\otimes \Shk$
is a $\Pcal$-torsor. Define an $\Lcal$-torsor $I_{\Lcal}$ on $\Shk$ as the quotient $I_{\Pcal}/R_u(\Pcal)$ where $R_u(\cdot)$ denotes the unipotent radical. 

Let $N$ be a finite extension of $E_{\pfr}$ with ring of integers $\mathcal O_N$ and prime $\wp$ lying over $\pfr$. 
Since $I_{\Pcal}/\Pcal \cong \Shk$,  every algebraic representation of $\Pcal$ on a finite free $\Ocal_{N}$-module $W$ gives rise to a vector bundle on $\Shk$ as in \S\ref{Gvarstacks}. By setting $R_u(\Pcal)$ to act trivially, any representation of $\Lcal$ on $W$ gives rise to one of $\Pcal$.

Let $\eta \in \chargpldom$. 
Let $\Lscr(\eta)$ be the associated $\LL_{\qpbar}$-equivariant (or $\LL_{\qpbar}$-linearized) line bundle on the flag variety $(\LL/\BB_\LL)_{\qpbar}$. Then there exists an extension $N$ as above such that $\Lscr(\eta)$ descends to an $\Lcal$-equivariant line bundle on $\Lcal/\Bcal_\Lcal \times_{\Ocal_{\pfr}} \Ocal_N$ over $\Ocal_N$. 
Continue to call the descended line bundle $\Lscr(\eta)$. Let $V_{\eta}$ be the representation of $\Lcal$ on $H^0(\Lcal/\Bcal_{\Lcal} \times_{\Ocal_{\pfr}} \Ocal_N, \Lscr(\eta))$. By the Borel-Weil  Theorem, $V_{\eta} \otimes \qpbar$ is irreducible of highest weight $\eta$. In general $V_{\eta}$ is the (possibly reducible) highest weight, induced module denoted $H^0(\eta)$ in \cite{jantzen-representations}. 

The automorphic vector bundle of weight $\eta$ is the vector bundle $\veta$ on $\Shk$  afforded by the torsor $I_{\Lcal}$ and the representation $V_{\eta}$. If different levels are in play, write $\Vscr_{\Kcal}(\eta)$ to specify the level. 

\subsubsection{$\gofafp$-equivariant objects} 
\label{sec-gofafp-objects-1}
Recall our convention \S\ref{sec-reduct-mod-pn}: Write $\Sscr_{\Kcal}^+:=\Shk$, $\Sscr_{\Kcal}^0=\Shk \otimes_{\Ocal_{E,\pfr}}E_{\pfr}$ for its generic fiber and $\Shkn:=\Shk \otimes_{\Ocal_{E,\pfr}} \Ocal_{E,\pfr}/\pfr^n$ for $n\geq 1$. Let $n \in \zgeqz \cup \{+\}$. A $\gofafp$-equivariant sheaf on the tower $(\Sscr^n_{\Kcal})_{\Kcal^p}$ is a system of sheaves $(\Fscr_{\Kcal})_{\Kcal^p}$ such that 
$\pi_{\Kcal'/\Kcal}^*\Fscr_{\Kcal}=\Fscr_{\Kcal'}$ and $g^*\Fscr_{g^{-1}\Kcal g} =\Fscr_{\Kcal}$ for all $\Kcal'^{p} \subset \Kcal^p$ and all $g \in \gofafp$, where $\Kcal'=\Kcal_p\Kcal'_p$. 
Similarly a morphism of $\gofafp$-equivariant sheaves $(\Fscr_{\Kcal}) \to (\Gscr_{\Kcal})$ consists of a family of morphisms 
$\Fscr_{\Kcal} \to \Gscr_{\Kcal}$ which is compatible with $\pi_{\Kcal'/\Kcal}$ and $g$.  
This applies to both \'etale sheaves (in particular \'etale torsors) and to (quasi-)coherent $(\Ocal_{\Sscr_{\Kcal}})_{\Kcal^p}$-modules. Given a $\gofafp$-equivariant sheaf $(\Fscr_{\Kcal})$, a $\gofafp$-equivariant section $t$ is a system $(t_{\Kcal} \in H^0(\Sscr^n_{\Kcal}, \Fscr))$ satisfying $\pi_{\Kcal'/\Kcal}^*t_{\Kcal}=t_{\Kcal'}$ and $g^* t_{g^{-1}\Kcal g}=t_{\Kcal}$.

More generally, define a $\gofafp$-system of schemes as a system of schemes $(\Zscr_{\Kcal})_{\Kcal^p}$ such that for every inclusion $\Kcal'^p \subset \Kcal^p$, one has a projection $\Zscr_{\Kcal'} \to \Zscr_{\Kcal}$ and for every $g\in \gofafp$ a system of isomorphisms $g:\Zscr_{\Kcal} \stackrel{\sim}{\to} \Zscr_{g^{-1} \Kcal g}$ (subject to the usual compatibility conditions). A $\gofafp$-equivariant morphism $\alpha: (\Zscr_{1,\Kcal}) \to (\Zscr_{2,\Kcal})$ is a system of morphisms $\alpha_{\Kcal}:\Zscr_{1,\Kcal} \to \Zscr_{2,\Kcal}$ such that $\Zscr_{1,\Kcal'} \to \Zscr_{1,\Kcal}$ (resp. $g:\Zscr_{1,\Kcal} \stackrel{\sim}{\to} \Zscr_{1,g^{-1} \Kcal g}$) is the pullback of $\Zscr_{2,\Kcal'} \to \Zscr_{2,\Kcal}$ (resp. $g:\Zscr_{2,\Kcal} \stackrel{\sim}{\to} \Zscr_{2,g^{-1} \Kcal g}$) along $\alpha$. In particular, a $\gofafp$-equivariant subscheme of $(\Sscr^n_{\Kcal})_{\Kcal^p}$ is a system of subschemes $(\Zscr_{\Kcal})_{\Kcal^p}$, where $\Zscr_{\Kcal}$ is a subscheme of $\Sscr^n_{\Kcal}$ such that $\pi_{\Kcal'/\Kcal}^{-1}(\Zscr_{\Kcal})=\Zscr_{\Kcal'}$ and $g^{-1}(\Zscr_{g^{-1}\Kcal g})=\Zscr_{\Kcal}$ (both scheme-theoretically). It follows directly from the definitions that a system of ideal sheaves $(\Ical_{\Kcal} \subset \Ocal_{\Sscr^n_{\Kcal}})$ is $\gofafp$-equivariant if and only if the corresponding system of zero-schemes is. 

By construction the universal abelian scheme $\Acal\to \Sscr$,  the vector bundle $H^1_{\dR}(\Acal/\Sscr)$ and the sections $s_{\alpha, \dR}$ are $\gofafp$-equivariant. It follows that also the bundles $\Vscr(\eta)$ and the torsors $I_{\Gcal},I_{\Pcal},I_{\Lcal}$ are all $\gofafp$-equivariant.

\subsubsection{Vector bundle dictionary} 
\label{sec-vector-bundle-dictionary}
Let $\std: GSp(V, \psi) \hookrightarrow GL(V)$ be the forgetful representation; it is an irreducible $GSp(V,\psi)$-module. Let $\eta_{g, \Omega}$ be the highest weight of $\std$ relative to our choice of Borel pair $(\BB, \TT)$ (\S\ref{sec-shimura-root-data-mu} in the case $\GG=GSp(V, \psi)$) and convention on positive roots (\S\ref{sec-based-root-data}).
Using the notation of \S\ref{sec-torsors}, put $\eta_{g,\omega}:=\det V_{\eta_{g,\Omega}}^{\vee}$. Recall that 
$\Vscr(\eta_{g, \Omega})^{\vee}\cong \Omega_g$
 and  
$\Vscr(\eta_{g, \omega}) \cong \omega_g$ \cf \cite[Proof of \Th~5.5.1]{Goldring-Galois-reps-HLDS} or \cite[p. 257, \Ex~(b)-(c)]{Chai-Faltings-book}. 
(Consider one standard set of choices/coordinates: $V=\QQ^{2g}$, $\psi=\left(\begin{array}{cc}
    0 & -I_g \\
    I_g &  0
\end{array} \right)$, $\TT$ is the diagonal ($\QQ$-split) maximal torus, identify $X^*(\TT)$ with $\{(a_1, \ldots a_g;c) \in \ZZ^{g+1} \ | \sum a_i \equiv c \pmod 2\}$ 
via $\diag(t_1z, \ldots ,t_gz, t_1^{-1}z, \ldots ,t_g^{-1}z) \mapsto t_1^{a_1} \cdots t_g^{a_g}z^c$, $\Delta=\{e_1-e_2, \ldots, e_{g-1}-e_g, 2e_g\}$ and $\Delta_L=\Delta \setminus \{2e_g\}$. Then $\eta_{g, \Omega}$ (resp. $\eta_{g, \omega}$) is identified with $\diag(0,\ldots,0,-1;-1)$ (resp. $\diag(-1, \ldots, -1; -g)$), \cf \cite[p. 306]{Taylor-GSp4}.)

In general, set $\eta_{\omega}=\varphi^*\eta_{g, \omega}$. Then $
\Vscr(\eta_{\omega}) \cong \omega$ as $\gofafp$-equivariant line bundles on $\Shk$.
We call $\eta_{\omega}$ the {\em Hodge character} associated to $\varphi$. This `modular' definition agrees with the group-theoretic one given in \Def~\ref{defHodge}\ref{item-def-omegazip}. 

\subsection{Universal $G$-zip over $S_\Kcal$} \label{sec-univ-Gzip} Let $S_{\Kcal}:=\Sscr_{\Kcal}^1$ be the special fiber of $\Sscr_{\Kcal}$.
We briefly review the construction (\cite[\Th~2.4.1]{ZhangEOHodge} and \cite[\S 5]{Wortmann-mu-ordinary}) of the universal $G$-zip $\underline{I}=(I,I_P,I_Q,\iota)$ over $\Shko$. Denote by $\kappa:=\Ocal_E/\pfr$ the residue field of $\pfr$. Define $G:=\Gcal\otimes \FF_p$ and write again $\mu:\GG_{m,\kappa}\to G_\kappa$ and $\mu_g=\mu \circ \varphi$ for the reduction of $\mu,\mu_g$. Recall (\S\ref{sec-integral-structure-theory}) that we have $\Lambda = \Lambda_{0} \oplus \Lambda _{-1}$. Define parabolic subgroups $P$, $Q$ of $G_\kappa$ as the stabilizers in $G_\kappa$ of $\textrm{Fil}_P:=\Lambda_{0,\kappa}$ and $\textrm{Fil}_Q:= {}^\sigma \Lambda_{-1,\kappa}$, respectively.

Denote by $(\overline{s}_\alpha)$ the reduction of $(s_\alpha)$ to $\Lambda_{\FF_p}$ and by $\overline{s}_{\alpha,\dR}$ the reduction of $(s_{\alpha,\dR})$ to $\Shko$. The sheaf $H^1_{\dR}(A/S_\Kcal):=H^1_{\dR}(\Acal/\Sscr_\Kcal)\otimes \kappa$ admits a Hodge filtration $\textrm{Fil}_{\rm H}:=\Omega_{A/S_\Kcal}$ and a conjugate filtration $\textrm{Fil}_{\rm conj}$. Furthermore, their graded pieces are related by the Cartier isomorphisms:
\begin{equation}
\iota_0 : \textrm{Fil}_{\rm H}^{(p)} \simeq H^1_{\dR}(A/\Shko)/\textrm{Fil}_{\rm conj} \quad \textrm{and} \quad
\iota_1 : (H^1_{\dR}(A/\Shko)/\textrm{Fil}_{\rm H})^{(p)} \simeq \textrm{Fil}_{\rm conj}.
\end{equation}
One then defines:
\begin{align}
I&:=\Isomcal_{\Shko}((H^1_{\dR}(A/\Shko),\overline{s}_{\dR}), (\Lambda,\overline{s})\otimes \Ocal_{\Shko}) \label{Gzipuniv1} \\
I_P&:=\Isomcal_{\Shko}((H^1_{\dR}(A/\Shko),\overline{s}_{\dR},\textrm{Fil}_{\rm H}), (\Lambda,\overline{s}, \textrm{Fil}_P)\otimes \Ocal_{\Shko}) \label{Gzipuniv2}\\
I_Q&:=\Isomcal_{\Shko}((H^1_{\dR}(A/\Shko),\overline{s}_{\dR},\textrm{Fil}_{\rm conj}),(\Lambda,\overline{s}, \textrm{Fil}_Q)\otimes \Ocal_{\Shko}) \label{Gzipuniv3}
\end{align}
and the isomorphism of $L^{(p)}$-torsors is given by $\iota=(\iota_0,\iota_1)$. This yields a universal $G$-zip $\underline{I}=(I,I_P,I_Q,\iota)$ over $\Shko$.
 It  gives rise to a morphism of stacks
\begin{equation}
\label{zeta}
\zeta:\Shko \to \GZip^{\mu}
\end{equation}
which is smooth by \cite[\Th~3.1.2]{ZhangEOHodge}.  Recall that $\GZip^\mu\simeq [E\backslash G]$. For every $w \in {}^I W$, define the EO-stratum $S_w$ in $\Shko$ by $S_w:=\zeta^{-1}([E\backslash G_w])$. Since $\zeta$ is smooth and $G_w$ is smooth, $S_w$ is a smooth, locally closed subscheme of $\Shko$.

Since the datum of a $G$-Zip consists of torsors and isomorphisms between them, one has the notion of $\gofafp$-equivariant $G$-Zip on $S_{\Kcal}$ by \S\ref{sec-gofafp-objects-1}. By construction,  $\underline{I}$ is $\gofafp$-equivariant: Given $\Kcal'^{p} \subset \Kcal^p$ and $g \in \gofafp$, one has commutative triangles
 \begin{equation}
 \label{eq-zeta-triangles}
 \xymatrixrowsep{.25pc}\xymatrixcolsep{2.5pc}\xymatrix{S_{\Kcal'} \ar[dd]_{\pi_{\Kcal'/\Kcal}} \ar[rd]^-{\zeta_{\Kcal'}} &  & & 
S_{\Kcal} \ar[rd]^-{\zeta_{\Kcal}}
\ar[dd]_g^*[@]{\sim}
&
\\ 
& \GZip^\mu & \mbox{ and } & & \GZip^{\mu}  \\
 S_{\Kcal} \ar[ru]_-{\zeta_{\Kcal}}   & & & S_{g^{-1} \Kcal g}\ar[ru]_-{\zeta_{g^{-1}\Kcal g}}
} \end{equation}
\subsection{Hasse invariants for Ekedahl-Oort strata}
\label{sec-hasse-shimura} 
We explain how to apply the general group theoretic \Th~\ref{GTHI} on Hasse invariants to $\Shko$. We give two ways to deal with the $(p,L)$-admissible hypothesis of \Th~\ref{GTHI}: (i) Show that it is satisfied by the Hodge character $\eta_{\omega}$ for all primes $p$, (ii) Show that for Hodge-type Shimura varieties, one can reduce to $G=GL(n)$, for which $\eta_{\omega}$ is minuscule. Regarding (i), one has:

\begin{theorem}[\cite{Goldring-Koskivirta-quasi-constant}, \Th~1.4.4]\label{th-quasi-const}
The Hodge character $\eta_{\omega}$ of a symplectic embedding $\varphi':\gx\hookrightarrow \shdagsp$ is quasi-constant \textnormal{(\Def~\ref{def-orb-p-close})}.
\end{theorem}
\begin{rmk} The theorem is proved using classical Hodge theory, specifically the methods of \cite{Deligne-Shimura-varieties}. In particular, it is valid for any symplectic embedding $\varphi'$, not just the special integral ones considered in \S\ref{sec-integral-symplectic-embed}.

\end{rmk}
As for (ii): Using \Th~\ref{GTHI} for $G=GL(n)$ and the generalization of the discrete fibers theorem \cite[\Th~2]{Goldring-Koskivirta-zip-flags} mentioned in \Rmk~\ref{rmk-intro-discrete-fibers}, we proved:
\begin{theorem}[\cite{Goldring-Koskivirta-zip-flags}, \Cor~6.2.2]
\label{th-zip-flags}
Let $(G,\mu)$ be a maximal cocharacter datum  (\cite[\S2.4]{Goldring-Koskivirta-zip-flags}, we do not assume $(G, \mu)$ arises by reduction mod $p$ from a Shimura datum). Let $\chi \in X^*(L)$ be a maximal character \textnormal{(\loccitn, \Def~2.4.3)}. Then there exists $N\geq 1$ such that for every $w \in {}^I W$, there is a section $h_w \in H^0([E\backslash \overline{G}_w], \Vscr(N\chi))$ with  $\nonvanish(h_w)=[E\backslash G_w]$.

\end{theorem}
Return to the setting where $(G,\mu)$ is associated to $\gx$ and $\varphi$ as in \S\ref{sec-univ-Gzip}. 
\begin{corollary} 
\label{cor-hasse-hodge} There exists $N\geq 1$ such that for every $w \in {}^I W$, there exists a section $h_w \in H^0([E\backslash \overline{G}_w], \Vscr(N\eta_{\omega}))$ whose non-vanishing locus is precisely $[E\backslash G_w]$. \end{corollary}
\begin{proof}[Proof 1:]
Since $\eta_{\omega}$ is easily seen to be $L$-ample,  \Th~\ref{th-quasi-const} shows that it is $(p,L)$-admissible for all primes $p$ (\S\ref{sec-conditions-characters}). Therefore the desired sections exist by \Th~\ref{GTHI}.
\end{proof}
\begin{proof}[Proof 2:]  The forgetful representation $GSp(V,\psi) \to GL(V)$ exhibits $(G, \mu)$ as a maximal cohcaracter datum and $\eta_{\omega}$ as a maximal character. Thus the corollary is a special case of \Th~\ref{th-zip-flags}. 

\end{proof}

\begin{corollary}[EO Hasse invariants]
\label{cor-hasse-shimura}
 For every EO stratum $S_w \subset \Shko$, the section $\zeta^*h_w\in H^0(\overline{S}_w,\omega^{N})$ is $\gofafp$-equivariant and $\nonvanish(\zeta^* h_w)=S_w$.
\end{corollary}

\begin{proof}
The section $\zeta^*h_w$ is $\gofafp$-equivariant by ~\eqref{eq-zeta-triangles} and $\nonvanish(\zeta^* h_w)=S_w$ follows from \Cor~\ref{cor-hasse-hodge}.
\end{proof}
Next we deduce the affineness of strata in the compact case (\Cor~\ref{cor-affine-cpt}). The generalization to the noncompact case (\Cor~\ref{cor-affine-noncpt}) is treated in \Prop~\ref{prop-min-cpt-strata}.
\begin{corollary}[Affineness, compact case]
\label{cor-affine-cpt-main-text}
 Suppose $\Shko$ is proper over $k$. Then the strata $S_w$ are affine for all $w \in {}^I W$.
\end{corollary}
\begin{proof}
 The Hodge line bundle $\omega$ is ample on $\Shko$ \cite[5.2.11(b)]{MadapusiHodgeTor}. Since the non-vanishing locus of a section of an ample line bundle on a proper scheme is affine, the result follows from \Cor~\ref{cor-hasse-shimura}. 
 \end{proof}

Combining \Th~\ref{th-quasi-const} with \Th~\ref{discrete_fibers} gives a `discrete fibers' result for symplectic embeddings of Shimura varieties:
\begin{corollary}[Discrete Fibers]
\label{disc-fib-Shimura}
Let $\varphi':\gx\hookrightarrow \shdagsp$ be an arbitrary symplectic embedding of Shimura data.  Then the induced morphism of stacks $\GZip^{\mu} \to \GspZip^{\mu_g}$ has discrete fibers. In other words, if two EO strata map to the same one under $\varphi'$, then there is no closure relation between them. \label{cor discrete fibers}\end{corollary}

For example, let $\varphi'$ be the tautological embedding of a PEL-type Shimura variety in its underlying Siegel variety. The naive stratification of $S_{\Kcal}$ is defined by taking preimages of Siegel EO strata; it is induced by the isomorphism class of the underlying $BT_1$ without its additional structure. \Cor~\ref{cor discrete fibers} states that a naive EO stratum is topologically a disjoint union of (true) EO strata.

\section{Lifting injective sections for $p$-nilpotent schemes and the length stratification}
\label{sec-lifting-gluing}

\subsection{Gluing sections}

%\subsection{The main lifting result}

\begin{comment}
\begin{theorem}\label{th-glue} Let $X$ be a finite type $\mathcal O_{\pfr}/\pfr^n$-scheme with special fiber $X^1$. 
Let $\mathcal L$ be a line bundle on $X$. 
Assume there is a decomposition \begin{equation} (X^1)_{\red}=X_1 \cup X_2 \cup \ldots \cup X_r \label{eq irred comp}, \end{equation} with each $X_i$ closed reduced subscheme of $(X^1)_{\red}$. 
There exists an integer $m_0$ such that for all $a\in \ZZ_{\geq 1}$ with $p^a\geq m_0$ and every $r$-tuple of sections $s_i \in H^0(X_i, \mathcal L)$, $1\leq i \leq r$, such that $s_i=s_j$ on $(X_i \cap X_j)_{\red}$, there exists a unique section $s \in H^0(X, \mathcal L^{p^{2a}})$ whose restriction to $X_i$ is $s_i^{p^{2a}}$ and which is Zariski-locally the $p^a$-power of a section $s'$ of $\Lcal^{p^a}$ which restricts to $s_i^{p^a}$ on $X_i$. 
Moreover, if $X$ has no embedded components, then the section $s$ is injective if and only if the sections $s_1, \ldots, s_r$ are all injective.  \end{theorem}
\end{comment}

\begin{theorem}
\label{th-glue} 
Let $X$ be a Noetherian scheme where $p$ is nilpotent. 
Let $\mathcal L$ be a line bundle on $X$. 
Let $\Dcal$ be a decomposition of $X_{\red}$: \begin{equation} X_{\red}=X_1 \cup X_2 \cup \ldots \cup X_r \label{eq irred comp}, \end{equation} where $X_i\subset X_{\red}$ are closed reduced subschemes. 
Then
\begin{enumerate}[label=(\alph*)]
\item 
\label{item-lift1} 
There exists an integer $m_0(X,\Lcal,\Dcal)$ such that for all integers $a\geq m_0(X,\Lcal,\Dcal)$ and every $r$-tuple of sections $s_i \in H^0(X_i, \mathcal L)$, $1\leq i \leq r$, such that $s_i=s_j$ on $(X_i \cap X_j)_{\red}$, there exists a unique section $s \in H^0(X, \mathcal L^{p^{2a}})$ which is Zariski-locally the $p^a$-power of a section $s'$ of $\Lcal^{p^a}$ which restricts to $s_i^{p^{2a}}$ on $X_i$. 

\item 
\label{item-lift2} 
If $X$ has no embedded components, then  $s$ is injective if and only if $s_1, \ldots, s_r$ are all injective.
\end{enumerate}
\end{theorem}

\begin{proof}
First we consider \ref{item-lift1}. Since $X$ is Noetherian, it is covered by finitely many open affine subsets $X=U_1\cup ... \cup U_d$ such that $\Lcal|_{U_j}=\Ocal_{U_j}$. For each $j=1,...,d$, let $\Dcal_j$ be the decomposition $(X_i\cap U_{j, \rm red})_i$ of $U_{j, \textrm{red}}$. If we prove the result for each $(U_j, \Lcal|_{U_j}, \Dcal_j)$, then we may take $m_0:=\sup_j \{m_0(U_j,\Lcal_j, \Dcal_j)\}$. Thus we may work locally and assume that $X=\spec(A)$ with $A$ Noetherian and $\Lcal=\Ocal_X$. By induction we reduce to $r=2$. Assertion \ref{item-lift1} follows from Lemma \ref{lemmaunique} below.

For~\ref{item-lift2}, we may also assume $X=\spec(A)$ with $A$ Noetherian. The result then follows from the claim that $s$ is injective if and only if $s|_{X_{\rm red}}$ is injective. This is an immediate consequence of the fact that the set of zero-divisors of $A$ is the union of the associated primes of $A$ (\cite[\Th~6.1]{Matsumura-commutative-ring-theory}).
\end{proof}

\begin{lemma}\label{lemmaunique}
Let $A$ be a Noetherian ring where $p$ is nilpotent. Let $\Ncal(A)$ be the nilradical and $I_1,I_2$ two radical ideals of $A$ such that $I_1\cap I_2=\Ncal(A)$. There exists an integer $m_0$ such that for all $n\geq m_0$ and for all $x_1,x_2\in A$ such that $x_1 -x_2\in \sqrt{I_1+I_2}$, there exists a unique $x\in A$ satisfying:
\begin{enumerate}[label=(\alph*)]
\item  \label{condicong1} $x\equiv x_1^{p^{2n}} \pmod{I_1}$ and  $x\equiv x_2^{p^{2n}} \pmod{I_2}$.
\item  \label{condicong2}  $x=y^{p^n}$ for some $y\in A$.
\end{enumerate}
\end{lemma}

\begin{proof}
Let $h,d,m\geq 1$ be integers such that
\begin{enumerate}
\item $p^h=0$ in $A$,
\item $(\sqrt{I_1+I_2})^{p^m}=0$ in $A/(I_1+I_2)$,
\item $\Ncal(A)^{p^d}=0$.
\end{enumerate}

Set $m_0:=\sup \{d+h-1,m\}$. There is an exact sequence:
\begin{equation}0 \longrightarrow  A/ \Ncal(A) \longrightarrow A/I_1 \oplus A/I_2 \longrightarrow A/(I_1+I_2) \longrightarrow 0 \label{eq x1 x2}\end{equation}
where the second map is $(g_1, g_2) \mapsto g_1-g_2$. Let $n\geq m_0$ and $x_1,x_2\in A$ such that $x_1-x_2 \in \sqrt{I_1+I_2}$. Then $x_1^{p^n}-x_2^{p^n}\in I_1 + I_2$ because the ring $A/(I_1+I_2)$ has characteristic $p$. Hence there exists $z\in A$ such that $z\equiv x_1^{p^n} \pmod{I_1}$ and  $z\equiv x_2^{p^n} \pmod{I_2}$. Therefore $x:=z^{p^n}$ satisfies the conditions.

To prove uniqueness, assume that $x,x'\in A$ satisfy \ref{condicong1} and \ref{condicong2}. In particular, we have $x-x'\in \Ncal(A)$. Write $x=y^{p^n}$ and $x'=y'^{p^n}$ for some $y,y'\in A$. The ring $A/\Ncal(A)$ has characteristic $p$, so $x-x'=y^{p^n}-y'^{p^n}=(y-y')^{p^n}=0$ in $A/\Ncal(A)$, hence $y-y'\in \Ncal(A)$. We deduce $(y-y')^{p^d}=0$ in $A$, hence also in $A/pA$. It follows that $y^{p^d}\equiv y'^{p^d} \pmod{pA}$. Thus $y^{p^{d+m-1}}\equiv y'^{p^{d+m-1}} \pmod{p^mA}$ for all $m\geq 1$. Hence $y^{p^{u}}=y'^{p^{u}}$ in $A$ for all $u\geq d+h-1$. In particular $x=x'$. 
\end{proof}

\subsection{Length stratification}
\label{sec-length-strata}

Retain the setting of \S\ref{sec-hasse-shimura}. In particular, $G$ is a reductive $\FF_p$-group and $\mu:\GG_{m,k}\to G_k$ is a cocharacter such that the cocharacter datum $(G,\mu)$ arises from a Shimura datum of Hodge-type (\S\ref{sec-hasse-shimura}). By \Cor~\ref{cor-hasse-hodge}, there exists an integer $N\geq 1$ and sections $h_w\in H^0([E\backslash\overline{G}_w],\omega^N)$ for each $w\in {}^I W$, such that the non-vanishing locus of $h_w$ is exactly the open substack $[E\backslash G_w]$.

Let $d:=\ell(w_{0,I} w_0)$. The length of any element $w\in {}^I W$ satisfies $0\leq \ell(w)\leq d$. For an integer $0\leq j \leq d$, define:
\begin{equation}
G_j=\bigcup_{\ell(w)=j}G_w \quad \textrm{and} \quad \overline{G}_j=\bigcup_{\ell(w)\leq j}G_w.
\end{equation}
Endow the locally closed subsets $G_j,\overline{G}_j$ with the reduced subscheme structure. We call $G_j$ the $j$th \emph{length stratum} of $G$.

\begin{lemma}\label{Weillemma}
One has $\overline{G}_j=\bigcup_{\ell(w) = j}\overline{G}_w$.
\end{lemma}

\begin{proof}
Recall that the underlying topological space of $\GZip^\mu$ is isomorphic to ${}^I W$ endowed with the topology induced by the partial order $\preccurlyeq$, which is in general finer than the restriction of the Bruhat order of $W$ (\cite[\Th~6.2]{Pink-Wedhorn-Ziegler-zip-data}).

Let $w\in \leftexp{I}{W}$ and let $w=s_{\alpha_1} \cdots s_{\alpha_r}$ be a reduced expression. Then $w':=s_{\alpha_1} \cdots s_{\alpha_{r-1}}\in \leftexp{I}{W}$ and $\ell(w')=r-1$. Furthermore, $w'\preccurlyeq w$ because $\preccurlyeq$ is finer than the Bruhat order. Hence any element in $\leftexp{I}{W}$ of length $r\geq 1$ has an element of $\leftexp{I}{W}$ of length $r-1$ in its closure. But since $w\mapsto w_{0,I}ww_0$ is an order-reversing involution, we deduce similarly that any element of length $\leq j$ in ${}^I W$ lies in the closure of an element of ${}^I W$ of length $j$.
\end{proof}

In the next proposition, we prove that the length stratification of $G$ is principally pure, i.e $G_j$ is the non-vanishing locus of a section over $\overline{G}_j$.

\begin{proposition}\label{prop-hj}
There exists an integer $N'\geq 1$ and for each $0\leq j \leq d$, a section $h_j \in H^0(\left[ E \backslash \overline{G}_j\right],\omega^{N'})$ such that the non-vanishing locus of $h_j$ is exactly $G_j$.
\end{proposition}
We call the $h_j$ {\em length Hasse invariants} of $\GZip^\mu$.
\begin{proof}
One has the decomposition $
\overline{G}_j=\bigcup_{\ell(w)=j}\overline{G}_w$ and for each $w$ with $\ell(w)=j$, the section $h_w$ afforded by \Cor~\ref{cor-hasse-hodge}. We may interpret each $h_w$ as a regular $E$-eigenfunction on $\overline{G}_w$ for the character $N\eta_\omega$. Using \Th~\ref{th-glue}, we obtain a function $\overline{G}_j\to \AA^1$ which restricts to $h^{r}_w$ on $\overline{G}_w$ for some $r\geq 1$. Hence $h_j$ is an $E$-eigenfunction for the character $Nr\eta_{\omega}$. It follows that $h_j$ identifies with an element of $H^0(\left[ E \backslash \overline{G}_j\right],\omega^{N'})$ for $N'=Nr$, and the result follows.
\end{proof}

Let $S$ be a scheme of finite type over $k$ and $f:S\to \GZip^\mu$ a morphism of stacks (not assumed smooth). For a character $\chi\in X^*(L)$, write $\Vscr_S(\chi):=f^*(\Vscr(\chi))$. For $w\in {}^I W$ and $j\in \{0,...,d\}$, define locally-closed subsets $S_w$, $S^*_w$, $S_j$, $S^*_j$ as the preimages by $f$ of $\left[E\backslash G_w \right]$,
$\left[E\backslash \overline{G}_w \right]$, $\left[E\backslash G_j \right]$, $\left[E\backslash \overline{G}_j \right]$ respectively. Endow each of them with the reduced subscheme structure. The $(S_w)_w$ form a locally-closed decomposition of $S$. It is false in general that $S^*_w$ is the Zariski closure of $S_w$ in $S$, but this holds for example when $f$ is smooth. One has
\begin{equation}
S_j=\bigsqcup_{\ell(w)=j}S_w \quad \textrm{and} \quad S^*_j=\bigsqcup_{\ell(w)\leq j}S_w.
\end{equation}
We call $S_j$ the $j$th \emph{length stratum} of $S$. Set $S_{j}=S^*_{j}=\emptyset$ for all $j<0$. Furthermore, by Lemma \ref{Weillemma} we have
\begin{equation}
S^*_j=\bigcup_{\ell(w) = j}S^*_w.
\end{equation}

\begin{proposition}
\label{prop-length-strata}
Assume the following:
\begin{enumerate}[label=(\alph*)]
\item \label{item-length-strata1} \label{item-length-equidim} The scheme $S$ is equi-dimensional of dimension $d$.
\item \label{item-length-strata2} \label{item-length-nonempty} The stratum $S_w$ is non-empty for all $w\in {}^I W$.
\item \label{item-length-strata3} The stratum $S_e=S_0$ is zero-dimensional.
\end{enumerate}
Then: \begin{enumerate}
\item
\label{item-S_j-equidim}
The schemes $S_j$ and $S^*_j$ are equi-dimensional of dimension $j$,
\item
\label{item-h_j-inject}
The sections $h_j$ are injective \textnormal{(\S\ref{sec-notation-line-bundle})}; equivalently $S_j$ is open dense in $S^*_j$.
\item
\label{item-S_w-equidim}
For $w\in {}^I W$, $S_w$ is equi-dimensional of dimension $\ell(w)$.
\end{enumerate} 
\end{proposition}

\begin{rmk} It is not claimed that $S_w$ is dense in $S^*_w$, nor that $S^*_w$ is equi-dimensional.
\end{rmk}

\begin{proof}
Let $j\geq 1$. Since $S^*_{j-1}$ is the set-theoretic vanishing locus $f^*(h_j)$ in $S^*_j$, we deduce that
\begin{equation}
\dim(S^*_j)-1\leq \dim(S^*_{j-1}) \leq \dim(S^*_j).
\end{equation}
Assumptions \ref{item-length-strata1} and \ref{item-length-strata3} imply that $\dim(S^*_j)=j$ for all $j=0,..., d$. We prove by decreasing induction on $j$ that $S^*_j$ is equi-dimensional of dimension $j$. This is true for $j=d$ by assumption \ref{item-length-strata1}. Assume it is true for $j\geq 1$. Since $S^*_{j-1}$ is the set-theoretic vanishing locus of $f^*(h_j)$ on $S^*_j$, which is equi-dimensional of dimension $j$, all irreducible components of $S^*_{j-1}$ have codimension $\leq 1$ in $S^*_j$. But  $\dim(S_{j-1}^*)=j-1$, so $S^*_{j-1}$  is equi-dimensional of dimension $j-1$. Since $S_j$ is open in $S^*_j$, it is also equi-dimensional of dimension $j$. This proves~\eqref{item-S_j-equidim}.

The section $f^*(h_j)$ does not restrict to zero on any irreducible component of $S^*_j$, so $S_{j}$ is dense in $S^*_j$ and $h_j$ is injective, which shows \eqref{item-h_j-inject}.

Finally, we show~\eqref{item-S_w-equidim}. If $w\in {}^IW$ satisfies $\ell(w)\leq j$ and $\dim(S_w)=j$, then $\ell(w)=j$. Conversely, let $w\in W$ be an element of length $j$. Since $S^*_j$ is pure of dimension $j$, there exists $w'\in {}^I W$ of length $j$ satisfying $\dim(S^*_{w'})=j$ such that $S_w$ intersects $S^*_{w'}$. The continuity of $f$ implies that $w'=w$, so $\dim(S^*_w)=j$. Hence $\dim(S_w)=j$ as well. Let $Z$ be an irreducible component of $S_w$, and assume $\dim(Z)<j$. Then $Z$ is contained in $S^*_{w'}$ for an element $w'\neq w$ of length $j$, because $S^*_j$ is pure of dimension $j$. Again, this contradicts the continuity of $f$. Hence $S_w$ is pure of dimension $j$.
\end{proof}

\section{Extension of Hasse invariants to compactifications}
\label{sec-eo-compactification}

\subsection{Compactifications} \label{sec-shimura-compactification}
We review the results of Madapusi-Pera \cite{MadapusiHodgeTor} that we shall use about integral models of toroidal and minimal compactifications of Hodge-type Shimura varieties. 
As recalled below, most of these results rely on the corresponding statements in the Siegel case, which are due to Chai-Faltings \cite{Chai-Faltings-book}. 
In-between the works of Chai-Faltings and Madapusi-Pera, generalizations to the intermediate PEL case were given in a series of works by Lan \cite{Lan-book-thesis,Lan-toroidal-kuga,Lan-higher-koecher,Lan-Ordinary-Loci}. We shall often refer to Lan for precise references in the Siegel case.

\subsubsection{Toroidal compactifications}
\label{sec-toroidal-review}
Recall the choice of hyperspecial $\tilde \Kcal_p \subset GSp(2g, \qp)$ (\S\ref{sec-integral-symplectic-embed}) and the map $\varphi^{\Sh}:\Shk \to \Shgk$ \eqref{integral-model-map}. In the Siegel case, let $\tilde \Sigma$ be a finite, admissible rpcd (rational, polyhedral cone decomposition) for $(GSp(2g), \XX_g,\tilde \Kcal)$  \cite[2.1.22-2.1.23]{MadapusiHodgeTor}. In general, let $\Sigma$ be a refinement of the finite, admissible rpcd obtained by pullback
$\varphi^*\tilde \Sigma$ for $(\GG, \XX,\Kcal)$.  
By \cite{Chai-Faltings-book} (resp. \cite[\Ths~1, ~4.1.5]{MadapusiHodgeTor})
\begin{comment}
Madapusi reference checked (arxiv 2018)
\end{comment}
there exists a toroidal compactification $\Sscr_{g, \tilde \Kcal}^{\tilde \Sigma}$ of $\Shgk$ (resp. toroidal compactifications $\Sscr_{\Kcal}^{\Sigma}$, $\Sscr_{\Kcal}^{\varphi^*\tilde \Sigma}$ of $\Shk$). The toroidal compactification $\Sscr_{\Kcal}^{\varphi^* \tilde \Sigma}$ is defined as the normalization of the Zariski closure of $\shgx_{\Kcal}$ in $\Sscr_{g,\tilde \Kcal}^{\tilde \Sigma}$. So one has maps $\Sscr_{\Kcal}^{\Sigma} \to \Sscr_{\Kcal}^{\varphi^*\tilde \Sigma}$ and $\Sscr_{\Kcal}^{\varphi^*\tilde \Sigma} \to \Sscr_{g, \tilde \Kcal}^{\tilde \Sigma}$.
Denote their composite by
\begin{equation}
\label{eq-varphi-tor}
\varphi^{\Sigma/\tilde \Sigma}:\Sscr_{\Kcal}^{\Sigma} \to \Sscr_{g, \tilde \Kcal}^{\tilde \Sigma}.
\end{equation}
If in addition $\tilde \Sigma$ is smooth (and such rpcd's do exist), then $\Sscr_{g, \tilde \Kcal}^{\tilde \Sigma}$ is smooth; if both $\tilde{\Sigma}$ and $\Sigma$ are chosen smooth (and again such choices do exist), then $\Sscr_{\Kcal}^{\Sigma}$ is smooth too. 
\subsubsection{Compactification of the universal semi-abelian scheme}
\label{sec-univ-semi-abel}
Let $\tilde{\Acal}_g$ be the universal semi-abelian scheme over $\Sscr_{g, \tilde \Kcal}^{\tilde \Sigma}$. Let $\beta:\overline{\Acal}_g \to \Sscr_{g, \tilde \Kcal}^{\tilde \Sigma}$ be one of the compactifications of $\tilde{\Acal_g}$ constructed in \cite[\S6.1]{Chai-Faltings-book} and \cite[\S2.B]{Lan-toroidal-kuga}. By  \cite[\Prop~3.19]{Lan-toroidal-kuga} (see also \cite[\Chap~VI, \S1, \Rmk~1.4]{Chai-Faltings-book}), after possibly refining $\tilde{\Sigma}$ (and a corresponding combinatorial datum used in the definition of $\overline{\Acal}_g)$ one can arrange that the structure map $\beta$ is log integral (by \cite[Prop~3.18]{Lan-toroidal-kuga} this is equivalent to $\beta$ being equidimensional). In particular, by refining $\Sigma$ one may ensure that $\Sigma$ is a refinement of $\varphi^*\tilde \Sigma$ for some $\tilde \Sigma$ which admits a $\beta$ which is log integral.  The log integrality of $\beta$ will be used to ensure that the Frobenius is well-behaved on the log de Rham cohomology of $\Sscr_{g, \tilde{\Kcal}}^{\tilde{\Sigma}}$, see the proof of Lemma~\ref{lem-FV}.

\subsubsection{Extension of de Rham and Hodge bundles}
\label{sec-extend-Hodge-bundle}
By \cite[\Chap~6, \Th~4.2 \& preceding examples]{Chai-Faltings-book} and \cite[\Th~2.15(d)]{Lan-toroidal-kuga}, $H^1_{\logdr}(\overline{\Acal}_g/\Sscr_{g, \tilde \Kcal}^{\tilde \Sigma})$ is a rank $2g$, locally free extension of $H^1_{\dR}(\Acal_g/\Sscr_{g, \Kcal})$.  
Let $H_{\dR}^{1,\can}(\Acal_g/\Shgk)$ be the canonical extension of $H_{\dR}^{1}(\Acal_g/\Shgk)$ to $\Sscr_{g, \tilde \Kcal}^{\tilde \Sigma}$ .
 By \loccitn,  $H_{\dR}^{1,\can}(\Acal_g/\Shgk)=H^1_{\logdr}(\overline{\Acal}_g/\Sscr_{g, \tilde \Kcal}^{\tilde \Sigma})$. 
 
 By \cite[\Prop~6.9(2)]{Lan-toroidal-kuga}  the natural pairing on $H_{\dR}^{1}(\Acal_g/\Shgk)$ extends (uniquely) to a perfect pairing on $H_{\dR}^{1,\can}(\Acal_g/\Shgk)$. The canonical extension  $\Omega^{\can}_g\subset H_{\dR}^{1, \can}(\Acal_g/\Shgk)$  is a maximal, totally isotropic, locally direct factor; it is also the pull-back along the identity section of $\Omega^1_{\tilde{\Acal_g}/\Shgk}$. 
 
In the general case, pull back along $\varphi^{\Sigma/\tilde \Sigma}$: Put $H^1_{\logdr}(\overline{\Acal}/\Sscr_{\Kcal}^{\Sigma}):=\varphi^{\Sigma/\tilde \Sigma,*}H^1_{\logdr}(\overline{\Acal}_g/\Sscr_{g, \tilde \Kcal}^{\tilde \Sigma})$ 
and $\Omega^{\can}:=\varphi^{\Sigma/\tilde \Sigma,*}\Omega_g^{\can}$; these give  locally free extensions of $H^1_{\dR}(\Acal_g/\Sscr_{\Kcal})$ and $\Omega$ respectively, see also \cite[5.1.1]{MadapusiHodgeTor}. Since the Hodge line bundle will be used so frequently, we abuse notation and continue to write $\omega$ for $\det \Omega^{\can}$ and $\omega_g$ for $\det \Omega^{\can}_g$.

\subsubsection{Degeneration of the Hodge-de Rham spectral sequence and base change: The Siegel case}
\label{sec-hdg-de-rham}
As in \cite[\Th~2.15(3)(a)]{Lan-toroidal-kuga}, let $$\overline{\Omega}^1_{\overline{\Acal}_g/\Sscr_{g, \tilde \Kcal}^{\tilde \Sigma}}:=\Omega^1_{\overline{\Acal}_g/\Ocal_{E,\pfr}}(d \log \infty)/\beta^* \Omega^1_{\Sscr_{g, \tilde \Kcal}^{\tilde \Sigma} /\Ocal_{E,\pfr}}(d \log \infty).$$ 
By \Th~2.15(3)(c) of \loccitn, the logarithmic  Hodge-de Rham spectral sequence \begin{equation}
\label{eq-hdg-dR}
E_1^{i,j}:=R^j\beta_*\overline{\Omega}^i_{\overline{\Acal}_g/\Sscr_{g, \tilde \Kcal}^{\tilde \Sigma}} \Rightarrow H^{i+j}_{\logdr}(\overline{\Acal}_g/\Sscr_{g, \tilde \Kcal}^{\tilde \Sigma})
\end{equation} degenerates at $E_1$ and the Hodge cohomology sheaves 
$R^j \beta_* \overline{\Omega}^i_{\overline{\Acal}_g/\Sscr_{g, \tilde \Kcal}^{\tilde \Sigma}}$ are locally free. As in \loccitn, it follows that $H^n_{\logdr}(\overline{\Acal}_g/\Sscr_{g, \tilde \Kcal}^{\tilde \Sigma})= \wedge^n H^1_{\logdr}(\overline{\Acal}_g/\Sscr_{g, \tilde \Kcal}^{\tilde \Sigma})$ for all $n \geq 0$. Therefore $H^n_{\logdr}(\overline{\Acal}_g/\Sscr_{g, \tilde \Kcal}^{\tilde \Sigma})$ is also locally free for all $n \geq 0$.  Applying \cite[\Cor~(8.6)]{Katz-nilpotent-monodromy-ihes}  gives:
\begin{lemma}
\label{lem-log-dR-base-change}
For all $n \geq 0$, formation of $H^n_{\logdr}(\overline \Acal_g/\Sscr_{g, \tilde \Kcal}^{\tilde \Sigma})$ commutes with arbitrary base change.
\end{lemma}

\subsubsection{Extension of torsors and bundles}
\label{sec-toroidal-bundles}
By \cite[\Prop~5.3.2]{MadapusiHodgeTor}, the sections $(s_{\alpha, \dR})$ of \S\ref{sec-torsors} extend to $H^1_{\logdr}(\Acal/\Sscr_{\Kcal}^{\Sigma})$. Repeating  the definition of $I_{\Gcal},I_{\Pcal}$ (\S\ref{sec-torsors})
with $\Sscr_{\Kcal}^{\Sigma}$, the extended sections and $H^1_{\logdr}(\Acal/\Sscr_{\Kcal}^{\Sigma})$ give extensions to torsors $I_{\Gcal}^{\Sigma},I_{\Pcal}^{\Sigma}$ on $\Sscr_{\Kcal}^{\Sigma}$. 
Again $I_{\Lcal}^{\Sigma}:=I_{\Pcal}^{\Sigma}/R_u(\Pcal)$ gives an $\Lcal$-torsor on $\Sscr_{\Kcal}^{\Sigma}$. 
Repeating the construction of \S\ref{sec-torsors} with $I_{\Lcal}^{\Sigma}$ in place of $I_{\Lcal}$  
associates a vector bundle $\vcaneta$ on $\Sscr_{\Kcal}^{\Sigma}$ with every $\eta \in \chargpldom$. Let $D=D_{\Kcal}^{\Sigma}$ be the boundary divisor of $\Sscr_{\Kcal}^{\Sigma}$ relative $\Sscr_{\Kcal}$. Set $\vsubeta:=\vcaneta(-D)$. Since $D$ is a relative, effective Cartier divisor in $\Sscr_{\Kcal}^{\Sigma}/\Ocal_{E,\pfr}$, the sheaf $\vsubeta$ is again locally free.

\subsubsection{The minimal compactification}
\label{sec-minimal-review}
Let $\Sscr_{g, \tilde \Kcal}^{\min}$ (resp. $\Shkmin$) be the minimal compactification of $\Shgk$ (resp. $\Shk$) constructed in \cite[Chap. 5, \Th~2.3]{Chai-Faltings-book} (resp. \cite[5.2.1]{MadapusiHodgeTor}). By construction, there are proper maps $\pi:\Sscr_{\Kcal}^{\Sigma} \to \Shkmin$ and $\pi_g:\Sscr_{g, \tilde \Kcal}^{\tilde \Sigma} \to \Shgkmin$ satisfying $\pi_* \Ocal_{\Sscr_{\Kcal}^{\Sigma}}=\Ocal_{\Shkmin}$ and $\pi_* \Ocal_{\Sscr_{g, \tilde \Kcal}^{\tilde \Sigma}}=\Ocal_{\Shgkmin}$. It is expected that the higher direct images of $\pi, \pi_g$ vanish in general; this is known in the PEL case, see Condition~\ref{cond-hdi} and \Rmk~\ref{rmk-hdi}. This issue will play an important role in  \S\S\ref{sec-vanish-strata}-\ref{sec-galois-reps}.

The Hodge line bundle $\omega$ on $\Sscr_{\Kcal}^{\Sigma}$ descends to an ample line bundle on $\Shkmin$ independent of the rpcd $\Sigma$ and still denoted $\omega$ \cite[5.1.3, 5.2.1]{MadapusiHodgeTor}. Applying the universal property of the minimal compactification \cite[Lemma 5.2.2]{MadapusiHodgeTor} to $\Shkmin$, the morphism $\pi_g \circ \varphi^{\Sigma/\tilde \Sigma}:\Sscr_{\Kcal}^{\Sigma} \to \Sscr^{\min}_{g, \tilde \Kcal}$ and the ample line bundle $\omega_g$ on $\Sscr^{\min}_{g, \tilde \Kcal}$ shows that $\pi_g \circ \varphi^{\Sigma/\tilde \Sigma}$ factors through $\Shkmin$. The result is a map 
\begin{equation}
\label{eq-varphi-min}
\varphi^{\rm min}:\Shkmin \to \Sscr^{\rm min}_{g, \tilde \Kcal}.
\end{equation} 

 Since $\pi_g \circ \varphi^{\Sigma/\tilde \Sigma}$ and $\pi$ are both proper, so is $\varphi^{\min}$. Observe that $(\varphi^{\min})^* \omega_g=\omega$ and both $\omega_g$ and $\omega$ are ample (see \S\ref{sec-hodge-line-bundle} and the proof of \cite[\Th~5.2.11(2)]{MadapusiHodgeTor}). Therefore $\varphi^{\min}$ is quasi-affine. As it is both quasi-affine and proper, $\varphi^{\min}$ is finite.

\subsubsection{$\gofafp$-equivariance for toroidal compactifications} 
\label{sec-gofafp-toro}
The notions of $\gofafp$-equivariant sheaf, torsor, subscheme, section and $G$-Zip from \S\S\ref{sec-gofafp-objects-1},~\ref{sec-univ-Gzip} generalize to the (double) tower $(\Sscr_{\Kcal}^{\Sigma})_{\Kcal^p, \Sigma}$ of toroidal compactifications. Later we shall consider sub-towers consisting of all $\Sscr^{\Sigma}_{\Kcal}$ such that $\Sigma$ is required to satisfy some property $P$ which is attainable by refinement (and no restriction is imposed on $\Kcal^p$); the examples of $P$ which will occurr are: "$\Sigma$ smooth", "$\Sigma$ a refinement of $\varphi^*{\tilde \Sigma}$ for some $\tilde \Sigma$ which admits $\beta$ log-integral" (\S\ref{sec-univ-semi-abel}) and their common intersection. Then $\gofafp$-equivariance is defined in the same way, except that
 we restrict to the morphisms 
$\pi^{\Sigma'/\Sigma}_{\Kcal'/\Kcal}:\Sscr_{\Kcal'}^{\Sigma'} \to \Sscr_{\Kcal}^{\Sigma}$ and 
$g^{-1}_{\Sigma^g/\Sigma}:\Sscr_{g^{-1} \Kcal g}^{\Sigma^g}\to \Sscr_{\Kcal}^{\Sigma}$, where ${\Sigma}$ has property $P$ and  $\Sigma'$ (resp. $\Sigma^g$) is a refinement of $\pi_{\Kcal'/\Kcal}^*\Sigma$ (resp. $(g^{-1})^* \Sigma)$) having property $P$. 

The system of boundary subschemes $(D^{\Sigma}_{\Kcal})$ is $\gofafp$-equivariant. The torsors $I_{\Gcal}^{\Sigma},I_{\Pcal}^{\Sigma},I_{\Lcal}^{\Sigma}$ and the automorphic vector bundles $\vcaneta,\vsubeta$ are all $\gofafp$-equivariant. 

\subsection{Extension of the universal $\GZip$} \label{sec-G-zip-extension} For the rest of \S\ref{sec-eo-compactification}, we assume that the rpcd $\Sigma$ for $(\GG, \XX, \Kcal)$ is a refinement of $\varphi^* \tilde \Sigma$, where the rpcd $\tilde{\Sigma}$ for $(GSp(2g), \XX_g, \tilde{\Kcal})$ has been chosen so that $\beta$ is log integral (\S\ref{sec-univ-semi-abel}). By contrast, the rpcd  $\Sigma$ for $(\GG, \XX, \Kcal)$ is not assumed to be smooth.

In this section, we prove:
\begin{theorem}[\Th~\ref{th-intro-tor-min}]
\label{th-tor-ext-Gzip}
The $G$-zip $\underline{I}$ \textnormal{(\S\ref{sec-univ-Gzip})} extends to a $\gofafp$-equivariant $G$-zip $\underline{I}^{\Sigma}$ over $S_{\Kcal}^{\Sigma}$. In particular,  $\zeta_{\Kcal}: \Shko\to \GZip^\mu$   extends to
\begin{equation}
\zeta_{\Kcal}^{\Sigma} : S_{\Kcal}^{\Sigma} \to \GZip^\mu,
\end{equation}
and~\eqref{eq-zeta-triangles} extends to commutative triangles \begin{equation}
 \label{eq-zeta-triangles-toroidal}
 \xymatrixrowsep{.25pc}\xymatrixcolsep{2.5pc}
\xymatrix{S^{\Sigma'}_{\Kcal'} \ar[dd]_{\pi^{\Sigma'/\Sigma}_{\Kcal'/\Kcal}} \ar[rd]^-{\zeta^{\Sigma'}_{\Kcal'}} &  & & 
S^{\Sigma}_{\Kcal} \ar[rd]^-{\zeta^{\Sigma}_{\Kcal}}
&
\\ 
& \GZip^\mu & \mbox{ and } & & \GZip^{\mu}.  
\\
 S^{\Sigma}_{\Kcal} \ar[ru]_-{\zeta^{\Sigma}_{\Kcal}}   
& & & 
S^{\Sigma^g}_{g^{-1} \Kcal g}
\ar[uu]^{g^{-1}}
\ar[ru]_-{\zeta^{\Sigma^g}_{g^{-1}\Kcal g}}
} \end{equation}
\end{theorem}
Define $S_w^{\Sigma}:=(\zeta_{\Kcal}^{\Sigma})^{-1}([E\backslash G_w])$ and $S_w^{\Sigma,*}:=(\zeta_{\Kcal}^{\Sigma})^{-1}([E\backslash \overline{G}_w])$ for all $w\in {}^I W$.
Before embarking on the proof, note the following immediate, important corollary: 
\begin{corollary}\label{exttorhw}
The Hasse invariant $h_w\in H^0(\overline{S}_w, \omega^{N_w})$ of 
\textnormal{\Cor~\ref{cor-hasse-shimura}} extends to a $\gofafp$-equivariant section $h_w^{\Sigma} \in H^0(S_w^{\Sigma,*}, \omega^{N_w})$ with non-vanishing locus precisely $S_w^{\Sigma}$.
\end{corollary}
 
 Let $A_g/S_{g, \tilde \Kcal}$ (resp. $\tilde{A}_g/S_{g, \tilde \Kcal}^{\tilde \Sigma}, \overline{A}_g/S_{g, \tilde \Kcal}^{\tilde \Sigma}$) be the special fiber of $\Acal_g/_{\Sscr_{g, \tilde \Kcal}}$ (resp. $\tilde{\Acal}_g/\Sscr_{g, \tilde \Kcal}^{\tilde \Sigma}, \overline{\Acal}_g/\Sscr_{g, \tilde \Kcal}^{\tilde \Sigma}$) and similarly for $H^{1,\can}_{\dR}(A_g/S_{g, \tilde \Kcal})$ etc. (\S\ref{sec-extend-Hodge-bundle}). 
 
The Frobenius  $F$ and Verschiebung $V$ of $A_g/S_{g, \tilde \Kcal}$ induce maps of $\Ocal_{S_{g, \tilde \Kcal}}$-modules $F:(H_{\dR}^{1}(A_g/\Shgko))^{(p)} \to H_{\dR}^{1}(A_g/\Shgko)$ and 
$V:H_{\dR}^{1}(A_g/\Shgko) \to H_{\dR}^{1}(A_g/\Shgko)^{(p)}$. These satisfy the usual conditions $\Ker(F)=\ima(V)$, $\Ker(V)=\ima(F)$. The polarization of $A_g$ induces perfect pairings on $H_{\dR}^{1}(A_g/\Shgko)$ and $H_{\dR}^{1}(A_g/\Shgko)^{(p)}$ under which $F,V$ are transpose to each other. 

\begin{lemma} \label{lem-FV}
There exist two maps of locally free $\Ocal_{\Shgko}$-modules
\begin{align}
\label{eq-F-toro}
F&:H_{\dR}^{1,\can}(A_g/\Shgko)^{(p)} \to H_{\dR}^{1,\can}(A_g/\Shgko) \\
\label{eq-V-toro}
V&:H_{\dR}^{1,\can}(A_g/\Shgko) \to H_{\dR}^{1,\can}(A_g/\Shgko)^{(p)} 
\end{align}
satisfying the following conditions
\begin{enumerate}[label=(\alph*)]
\item \label{item-FV1} $F,V$ extend the Frobenius and Verschiebung maps on $H^1_{\dR}(A_g/\Shgko)$.
\item \label{item-FV2} $\Ker(F)=\ima(V)$ and $\Ker(V)=\ima(F)$.
\item \label{item-FV3} $\Ker(V)$ and $\Im(V)$ are rank $g$ locally free and locally direct summands of $H_{\dR}^{1,\can}(A_g/\Shgko)$.
\end{enumerate}
\end{lemma}

\begin{proof}
The Frobenius $F:\overline{A}_g \to \overline{A}_g^{(p)}$ over $S_{g, \tilde \Kcal}^{\tilde \Sigma}$ induces $F:H_{\logdr}^{1}( \overline{A}_g^{(p)}/ S_{g, \tilde \Kcal}^{\tilde \Sigma}) \to H_{\logdr}^{1}( \overline{A}_g/ S_{g, \tilde \Kcal}^{\tilde \Sigma})$. Applying Lemma~\ref{lem-log-dR-base-change} to $F$ gives $H_{\logdr}^{1}( \overline{A}_g^{(p)}/ S_{g, \tilde \Kcal}^{\tilde \Sigma}) \cong H_{\logdr}^{1}( \overline{A}/ S_{g, \tilde \Kcal}^{\tilde \Sigma})^{(p)}$. Together this yields a map as in~\eqref{eq-F-toro}.

To simplify notation, write $\overline{\Mcal}=H_{\logdr}^{1}(\overline{A}_g/S_{g, \tilde{\Kcal}}^{\tilde{\Sigma}} )$. We have constructed an extension $F:\overline{\Mcal}^{(p)} \to \overline{\Mcal}$ over $S_{g, \tilde \Kcal}^{\tilde \Sigma}$. Define $V:\overline{\Mcal} \to \overline{\Mcal}^{(p)}$ as its transpose with respect to the perfect pairings on $\overline{\Mcal}$ and $\overline{\Mcal}^{(p)}$. Clearly, $V$ extends the Verschiebung map on $S_{g, \tilde \Kcal}$. It remains to check that $F,V$ satisfy conditions \ref{item-FV2} and \ref{item-FV3} of the lemma.

Since the relations $FV=0$ and $VF=0$ hold over $\Shgko$, which is open dense in $S_{g, \tilde \Kcal}^{\tilde \Sigma}$, they continue to hold over $S_{g, \tilde \Kcal}^{\tilde \Sigma}$. Hence $\ima(F)\subset \Ker(V)$ and $\ima(V)\subset \Ker(F)$. On the other hand, we claim that $(\Omega^{\can}_g)^{(p)} \subset \Ker(F)$. Indeed, $F$ induces an injective morphism $(\Omega^{\can}_g)^{(p)} / ((\Omega^{\can}_g)^{(p)} \cap \Ker(F)) \to \overline{\Mcal}$. Since we know the result over the open subscheme $\Shgko\subset S_{g, \tilde \Kcal}^{\tilde \Sigma}$, the sheaf $\Ncal:=(\Omega^{\can}_g)^{(p)} / ((\Omega^{\can}_g)^{(p)} \cap \Ker(F))$ restricts to zero on $\Shgko$. But $\Ncal$ is a subsheaf of a locally free sheaf, and since $S_{g, \tilde \Kcal}^{\tilde \Sigma}$ is reduced and $\Shgko$ is open dense, we deduce $\Ncal=0$. Thus $(\Omega^{\can}_g)^{(p)} \subset \Ker(F)$. 

Next, we claim that $\ima(F)$ is a locally free sheaf on $S_{g, \tilde \Kcal}^{\tilde \Sigma}$. It suffices to show $\dim_{k(x)}(\ima(F)\otimes k(x))=g$ for all $x\in S_{g, \tilde \Kcal}^{\tilde \Sigma}$. Since $(\Omega^{\can}_g)^{(p)}$ is locally a direct summand of $\Mcal^{(p)}$, we have an injection $(\Omega^{\can}_g)^{(p)} \otimes k(x) \to \Mcal^{(p)}\otimes k(x)$, which shows $\dim_{k(x)} (\ima(F)\otimes k(x)) \leq g$. The converse inequality follows simply from Nakayama's lemma because this dimension is $g$ generically on $S_{g, \tilde \Kcal}^{\tilde \Sigma}$. This proves that $\ima(F)$ is locally free of rank $g$.

For all $x\in S_{g, \tilde \Kcal}^{\tilde \Sigma}$, there is a natural surjection $\ima(F)\otimes k(x) \to \ima(F_x)$. In particular, $\rk(F_x)\leq g$. We claim that $\rk(F_x)=g$. 
To see this, consider $\Mcal:=H^1_{\logcrys}(\overline{A}_g / S^{\tilde \Sigma}_{g, \tilde \Kcal})$, endowed with its crystalline Frobenius $\Phi$. One has $\Mcal \otimes_{W(k)} k \simeq \overline{\Mcal}$ and $\Phi \otimes_{W(k)}k=F$ \cite[\Rmk~1.5]{Lan-Suh-general}. Recall that $\tilde{\Sigma}$ was chosen so that $\beta:\overline{\Acal}_g \to \Sscr_{g, \tilde{\Kcal}}^{\tilde{\Sigma}}$ is log integral. Using the log integrality of $\beta$ it is checked in \cite{Lan-Suh-general} that $\Sscr_{g, \tilde \Kcal}^{\tilde \Sigma}$ satisfies the two assumptions of \Prop~1.7 of \loccit (see esp. \Prop~1.13, Lemma~7.22 and its proof). In particular, the first implies that the top exterior power
$(\wedge^{2g} \Mcal, \wedge^{2g} \Phi)$ is the $\Phi$-span $\Ocal_{S_{g, \tilde \Kcal}^{\tilde \Sigma}/W(k)}(-g)$ as defined by Ogus\footnote{Ogus calls his spans $F$-spans since he uses $F$ for the crystalline Frobenius.} \cite[\Def~5.2.1]{Ogus-book-F-crystal-transversality-Hodge}. 
This means that the map $\wedge^{2g}\Phi:\wedge^{2g}\Mcal^{(p)} \to \wedge^{2g} \Mcal$ is Zariski locally on $S_{g, \tilde \Kcal}^{\tilde \Sigma}$ equal to  $p^g$ times  an isomorphism.  \begin{comment}
$$\wedge^{2g}\Phi:\wedge^{2g}H^{1}_{\logcrys}(\overline{A}_g/S_{g, \tilde{\Kcal}}^{\tilde{\Sigma}})^{(p)} \to \wedge^{2g}H^{1}_{\logcrys}(\overline{A}_g/S_{g, \tilde{\Kcal}}^{\tilde{\Sigma}})$$
\end{comment}

In particular, at the point $x$, the linear map $\Phi_x: \Mcal_x^{(p)} \to \Mcal_x$ satisfies that $\wedge^{2g} \Phi_x$ is $p^g$ times an isomorphism. By the theory of invariant factors, there exist suitable bases of $\Mcal_x$ in which $\Phi_x$ is given by the diagonal matrix $\diag(p^{a_1}, \ldots, p^{a_{2g}})$ for some integers $a_i$ satisfying $a_1 \geq \cdots \geq a_{2g} \geq 0$. Since the unique invariant factor of  $\wedge^{2g} \Phi_x$ is $p^g$, the only possibility is that $a_1=\cdots =a_g=1$ and $a_{g+1}=\cdots =a_{2g}=0$. Since  the reduction modulo $p$ of $(\Mcal_x, \Phi_x)$ is $(\Mcal\otimes k(x), F_x)$, one has $\rk(F_x)=g$ as claimed. 

Since $F_x$ and $V_x$ are transpose to each other, they have the same rank, so we deduce $\Ker(F_x)=\ima(V_x)$ and $\Ker(V_x)=\ima(F_x)$.
Since $\ima(F)$ is locally free, $\Ker(F)$ is locally a direct summand of $\overline{\Mcal}^{(p)}$. Hence $\Ker(F)\otimes k(x)\simeq \Ker(F_x)=\ima(V_x)$ for all $x\in S_{g, \tilde \Kcal}^{\tilde \Sigma}$. In particular, the inclusion $\ima(V)\subset \Ker(F)$ induces a surjective map $\ima(V)\otimes k(x)\to \Ker(F)\otimes k(x)$. So $(\Ker(F)/\ima(V))\otimes k(x)=0$. Hence $\Ker(F)=\ima(V)$ by Nakayama's lemma. Thus $\Ker(V)$ is locally a direct factor of $\Mcal$. By similar arguments, $\Ker(V)=\ima(F)$. This terminates the proof.
\end{proof}

Put $\Fil^{\tilde \Sigma}_{g, \rm conj}:=\Ker(V)$.  
By Lemma \ref{lem-FV}, $0 \subset \Fil_{g, \rm conj}^{\tilde \Sigma} \subset H^{1,\can}_{\dR}(A_g/S_{g, \tilde \Kcal}^{\tilde \Sigma})$  extends the conjugate filtration (\S\ref{sec-univ-Gzip}) to $S_{g, \tilde \Kcal}^{\tilde \Sigma}$. Set $\Fil^{\Sigma}_{\rm conj}:=\varphi^{\Sigma/\tilde \Sigma,*}\textrm{Fil}_{g, \rm conj}$. Then $\Fil^{\Sigma}_{\rm conj}$ provides an extension of the conjugate filtration on  $H_{\dR}^{1,\can}(A/S_{\Kcal}^{\Sigma})$ which is locally free and locally a direct summand. Replacing $\textrm{Fil}_{\rm conj}$ with 
$\Fil^{\Sigma}_{\rm conj}$ in \eqref{Gzipuniv3} gives an \'{e}tale sheaf $I_Q^{\Sigma}$ on $\Shktoro$. Furthermore, $I_Q^{\Sigma}$ is endowed with a natural $Q$-action.
\begin{lemma} 
\label{lem-q-torsor}
The \'etale sheaf $I_Q^{\Sigma}$ is a $Q$-torsor on $\Shktoro$
\end{lemma}
\begin{proof}
We thank Torsten Wedhorn for his help with this. Choose an \'{e}tale cover $U\to \Shktoro$ which trivializes the $G$-torsor $I^{\Sigma}$. Hence we have an isomorphism $I^{\Sigma}|_U\simeq G\times U$. The stabilizer of $\textrm{Fil}_{\rm conj}^{\Sigma}$ is a parabolic subgroup $\Qcal$ of $G\times U$, and we need to show that it is \'{e}tale locally on $U$ isomorphic to $Q$. To check this, it suffices to show that $\Qcal\otimes k(x)\simeq Q\otimes k(x)$ at every point $x$. This follows from the fact that the type of the parabolic $\Qcal\otimes k(x)$ is a locally constant function, and that the result holds generically (on the open subset $\Shko$). 
\end{proof}

\begin{proof}[\textnormal{Proof of \Th~\ref{th-tor-ext-Gzip}:}] By \S\ref{sec-toroidal-bundles}
and Lemma~\ref{lem-q-torsor}, we have torsors $I_{G}^{\Sigma},I_P^{\Sigma},I_Q^{\Sigma}$ for $G,P,Q$ extending $I_G, I_P,I_Q$ respectively. It remains to extend the isomorphisms $\iota_0$, $\iota_1$ to $\Shktoro$ (\S\ref{sec-univ-Gzip}). On $\Shgko$, these isomorphisms are naturally induced by the maps $F,V$. Since these maps extend to $S_{g, \tilde \Kcal}^{\tilde \Sigma}$ by Lemma \ref{lem-FV}, so do $\iota_0$ and $\iota_1$. One obtains similar isomorphisms on $\Shktoro$ by pull-back from $S_{g, \tilde \Kcal}^{\tilde \Sigma}$. This completes the construction of the $G$-zip $\underline{I}^{\Sigma}$ over $\Shktoro$.
\end{proof}

\subsection{Affineness in the minimal compactification} 
\label{sec-eo-minimal-cpt}
Let $\pi: \Shktoro\to \Shkmino$ be the natural projection. For $w\in {}^I W$, define
\begin{align}
S^{\min}_w&:=\pi(S^{\Sigma}_w) \\
S^{\min,*}_w&:=\pi(S^{\Sigma,*}_w).
\end{align} 
By~\eqref{eq-zeta-triangles-toroidal} (with $\Kcal'=\Kcal$) the above definitions are invariant under a refinement $\Sigma_0 \subset \Sigma$ (corresponding again to a $\beta$ which is log integral (\S\ref{sec-univ-semi-abel})). Hence $S^{\min}_w$ and $S^{\min,*}_w$ are independent of $\Sigma$.
Part~\ref{item-min-cpt-strata-affine} of the following proposition is \Cor~\ref{cor-affine-noncpt}.
\begin{proposition} \label{prop-min-cpt-strata}
\ 
\begin{enumerate}[label=(\alph*)]
\item 
\label{item-min-cpt-strata-disjoint} 
The $S^{\min}_w$ are pairwise disjoint and locally closed. 
\item \label{item-stein} The Stein factorization of $\pi:\overline{S}^{\Sigma}_w\to \overline{S}^{\min}_w$ is given by $\pi=f\circ g$ where $g:\overline{S}^{\Sigma}_w\to T$ is proper with connected fibers such that $g_{*}\Ocal_{\overline{S}^{\Sigma}_w}=\Ocal_T$, and $f:T\to \overline{S}^{\min}_w$ is a finite morphism which is a universal homeomorphism.
\item \label{item-min-cpt-strata-affine} The stratum $S^{\min}_w$ is affine.
\end{enumerate}
\end{proposition}

\begin{rmk}\label{remarkminimal}  
 \Prop~\ref{prop-min-cpt-strata}\ref{item-min-cpt-strata-disjoint} affords a decomposition of $\Shkmino$ into disjoint, locally closed subsets:
\begin{equation}\label{decompmin}
\Shkmino=\bigsqcup_{w\in {}^I W}S^{\min}_w.
\end{equation} 
For Siegel-type Shimura varieties, it follows from \cite[\S5]{Ekedahl-Geer-EO} that \eqref{decompmin} is a stratification; for  PEL Shimura varieties of type A and C this, as well as the fact that $h^{\Sigma}_w$ descends to  $\overline{S}^{\min}_w$ is proved in \cite[\Th~6.1.6, \Th~6.2.3]{Boxer-thesis-arxiv}. In the general Hodge case, these properties should follow from \cite{Lan-Stroh-stratifications-compactifications}.
\end{rmk}

\subsubsection{Comparison with  Ekedahl-van der Geer \cite{Ekedahl-Geer-EO}}
We shall first prove \Prop~\ref{prop-min-cpt-strata}\ref{item-min-cpt-strata-disjoint} in the Siegel case and then reduce the general case to the Siegel one. For the Siegel case, we need to compare the extension of the EO stratification to $S_{g, \tilde\Kcal}^{\tilde \Sigma}$ afforded by  \Th~\ref{th-tor-ext-Gzip} with the extension defined in \cite[\S5]{Ekedahl-Geer-EO}.

Recall that the Weyl group $W_g$ of $GSp(2g)$ is realized concretely as a subgroup of the symmetric group $S_{2g}$ via $$W_g=\{w \in S_{2g} \ | \ w(j)+w(2g+1-j)=2g+1 \textnormal{ for all } 1\leq j \leq g\}.$$
Let $\rho_i:W_{g-i} \to W_g$ be the natural embedding, defined  for $w' \in W_{g-i}$ by $\rho_i(w')(j)=i+w'(j)$ for $1 \leq j \leq g-i$ and $\rho_i(w')(j)=g+j$ for all $g-i+1 \leq j \leq g$. It satisfies  $\rho_i(W_{g-i}) \cap {}^IW_g=\rho_i({}^IW_{g-i})$.

Every $k$-point $ x \in S_{g, \tilde \Kcal}^{\tilde \Sigma}$ admits an underlying semi-abelian scheme $\tilde{A}$, which is an extension
\begin{equation}
\label{eq-semi-abelian-def}
1 \to T \to \tilde{A} \to A \to 1
\end{equation}
 of an abelian scheme $A$ by a torus $T$.
Define $S_{g, \tilde \Kcal}^{\langle i \rangle}$ to consist of $ x \in S_{g, \tilde \Kcal}^{\tilde \Sigma}$ such that the dimension of the torus part is $i=\dim(T)$.   
 The EO stratification of $S_{g, \tilde \Kcal}$ is extended to $S_{g, \tilde \Kcal}^{\tilde \Sigma}$ in  \cite[\S5]{Ekedahl-Geer-EO} by defining the stratum $'S_{g,w}^{\langle i \rangle} \subset S_{g, \tilde \Kcal}^{\langle i \rangle}$ corresponding to $w \in {}^I W_g$ to consist of $ x \in S_{g, \tilde \Kcal}^{\langle i \rangle}$ such that, if the EO stratum of the abelian part is given by $w' \in {}^I W_{g-i}$, then $\rho_i(w')=w$. Then define $'S_{g,w}^{\tilde \Sigma}$ to be the union of the $'S_{g,w}^{\langle i \rangle}$ over all $i$.

\begin{lemma}[Comparison] 
\label{lem-compare-EvdG-zeta-tor} Let $S_{g,w}^{\tilde \Sigma}$ be the EO strata defined  by pulling back the  stratification of $\GspZip^{\mu_g}$ along $\zeta_{\tilde g, \Kcal}^{\tilde \Sigma}:S_{g, \tilde \Kcal}^{\tilde \Sigma} \to \GspZip^{\mu_g}$. Then
 $'S_{g,w}^{\tilde \Sigma}= S_{g,w}^{\tilde \Sigma}$.
 \end{lemma}
\begin{proof}
It suffices to prove that the two extensions of the EO stratifiation agree one boundary stratum at a time, \ie to show that $'S_{g,w}^{\tilde \Sigma} \cap S_{g , \tilde \Kcal}^{\langle i \rangle} =S_{g,w}^{\tilde \Sigma} \cap S_{g ,\tilde \Kcal}^{\langle i \rangle}$ for all $i$.  As in \cite[4.3.1]{MadapusiHodgeTor}, there is a universal, polarized $1$-motive over $S_{g ,\tilde \Kcal}^{\langle i \rangle}$; let $H_{\dR}^{1, \langle i \rangle}$ denote its (contravariant) de Rham realization (\loccitn, 1.1.3 and \cite[4.3]{andreatta-viale-crystalline-1-motives}). It is a locally free sheaf of rank $2g$ equipped with a non-degenerate symplectic pairing, induced from the polarization of the $1$-motive. By the characterization of $H_{\dR}^{1, \langle i \rangle}$ as a canonical extension in \cite[4.3.1]{MadapusiHodgeTor} and the characterization of $H_{\dR}^{1,\can}(A_g/\Shgko)=H^1_{\logdr}(\overline{A}_g/S_{g, \tilde \Kcal}^{\tilde \Sigma})$ in \cite[\Th~2.15(3)(d)]{Lan-toroidal-kuga}, the restriction of $H^1_{\logdr}(\overline{A}_g/S_{g, \tilde \Kcal}^{\tilde \Sigma})$ to $S_{g ,\tilde \Kcal}^{\langle i \rangle}$ is $H_{\dR}^{1, \langle i \rangle}$.

The de Rham realization $H^{1, \langle i \rangle}_{\dR}$ admits a $3$-step weight filtration by locally free sheaves 
\begin{equation}
\label{eq-weight-filtration}    
(0) \subset  W_0H^{1, \langle i \rangle}_{\dR} \subset W_1 H^{1, \langle i \rangle}_{\dR} \subset W_2 H^{1, \langle i \rangle}_{\dR}=H^{1, \langle i \rangle}_{\dR}.
\end{equation}
The weight filtration is symplectic in the sense that the orthogonal $(W_jH^{1, \langle i \rangle}_{\dR})^{\perp}$ of $W_jH^{1, \langle i \rangle}_{\dR}$ is $W_{2-j}H^{1, \langle i \rangle}_{\dR}$. Let $\Gr^W_j:=\Gr^W_jH^{1, \langle i \rangle}_{\dR}$ denote the associated graded. The symplectic pairing on $H^{1, \langle i \rangle}_{\dR}$ induces non-degenerate pairings on the complementary graded pieces, \ie \begin{equation}
\label{eq-graded-pairings}
\Gr^W_1 \times \Gr^W_1 \to \Ocal_{S_{g ,\tilde \Kcal}^{\langle i \rangle}} \textnormal{ and }  \Gr^W_0 \times \Gr^W_2 \to \Ocal_{S_{g ,\tilde \Kcal}^{\langle i \rangle}}.
\end{equation}

The middle graded piece $\Gr^W_1$ is given by the de Rham cohomology of the abelian part of the universal $1$-motive, \ie $\Gr^W_1=H^1_{\dR}(A_{g-i}/S_{g-i, \overline{\Kcal}})$, where $\overline{\Kcal} \subset GSp(2(g-i),\AA_f)$ is the corresponding level subgroup and as before $A_{g-i}$ denotes the universal abelian scheme over $S_{g-i, \overline{\Kcal}}$.
By definition, a polarization of a $1$-motive induces one of its abelian part \cite[1.1.1]{MadapusiHodgeTor}, so the pairing on $\Gr^W_1=H^1_{\dR}(A_{g-i}/S_{g-i, \overline{\Kcal}})$ induced from $H^{1, \langle i \rangle}_{\dR}$ agrees with the one induced from the polarization of $A_{g-i}/S_{g-i, \overline{\Kcal}}$.

Consider the fiber of~\eqref{eq-weight-filtration} over a $k$-point $x \in S_{g, \tilde \Kcal}^{\langle i \rangle}$ with underlying semi-abelian scheme $\tilde{A}$~\eqref{eq-semi-abelian-def}. We get a filtration of $k$-vector spaces stable by $F$ and $V$, viewed as $\sigma$-linear and $\sigma^{-1}$-linear maps respectively.  We claim that there is a (non-canonical) splitting which is stable under both $F$ and $V$. Set $H^{1, \langle i \rangle}_{\dR,x}:=H^{1, \langle i \rangle}_{\dR} \otimes k(x)$,  $W_{j,x}:=W_jH^{1, \langle i \rangle}_{\dR} \otimes k(x)$ and $\Gr^W_{j,x}:=\Gr^W_j \otimes k(x)$. The action of $F$ (resp. $V$) on $\Gr^W_{0,x}$ (resp. $\Gr^W_{2,x}$) is invertible (\cf \cite[4.3]{andreatta-viale-crystalline-1-motives}). 

Recall the following elementary fact of $\sigma$-linear algebra: Assume $N$ is a $k$-vector space, $f:N \to N$ is a $\sigma$-linear map and $N'$ is a subspace stable by $f$. If $f$ is invertible on either $N'$ or the quotient $N/N'$, then there is a (noncanonical) $f$-stable splitting $N= N' \oplus N/N'$. 

Applying the $\sigma$-linear algebra fact with $N=W_{1,x}$, $N'=W_{0,x}$ and $f=F$ gives an $F$-stable complement $\widetilde \Gr^W_{1,x}$ to $W_{0,x}$ in $W_{1,x}$. Since the first pairing in~\eqref{eq-graded-pairings} is perfect and also induced from that of $H^{1, \langle i \rangle}_{\dR}$, the restriction of the pairing on $H^{1, \langle i \rangle}_{\dR,x}$ to $\widetilde \Gr^W_{1,x}$ is perfect. Therefore one has 
\begin{equation}
\label{eq-orthogonal-decomposition}
H^{1, \langle i \rangle}_{\dR,x}=\widetilde \Gr^W_{1,x} \oplus (\widetilde \Gr^W_{1,x})^{\perp}. 
\end{equation} 
We claim that both summands in~\eqref{eq-orthogonal-decomposition} are both $F$-stable and $V$-stable.
Since $F,V$ are transpose to each other, the orthogonal $(\widetilde \Gr^W_{1,x})^{\perp}$ is $V$-stable. Further, $W_{1,x} \cap (\widetilde \Gr^W_{1,x})^{\perp}=W_{0,x}$, so $(\widetilde  \Gr^W_{1,x})^{\perp}/W_{0,x}=\Gr^W_{2,x}$. Since $V$ is invertible on $\Gr^W_{2,x}$, there exists a $V$-stable complement $\widetilde \Gr^W_{2,x}$ to $W_{0,x}$ in $(\Gr^W_{1,x})^{\perp}$. The relation $FV=0$ on $H^{1, \langle i \rangle}_{\dR,x}$ implies that $F$ is zero on $\widetilde \Gr^W_{2,x}$. Since $W_{0,x}$ is $F$-stable, we conclude that $(\widetilde \Gr^W_{1,x})^{\perp}$ is $F$-stable. Applying the transpose relation again gives that $\widetilde \Gr^W_{1,x}$ is $V$-stable. 

Therefore the weight filtration~\eqref{eq-weight-filtration} at the point $x$ admits a (noncanonical) splitting 
\begin{equation}
\label{eq-splitting-weight-filtration}
H_{\dR,x}^{1, \langle i \rangle} =\widetilde \Gr^W_{1,x} \oplus (W_{0,x} \oplus \widetilde \Gr^W_{2,x})
\end{equation}
stable by $F$ and $V$ and compatible with the pairing on $H^{1, \langle i \rangle}_{\dR,x}$.
In other words~\eqref{eq-splitting-weight-filtration} shows that the symplectic similitude $F$-Zip $H^{1, \langle i \rangle}_{\dR,x}$ is the sum of the two symplectic similitude $F$-Zips 
$\widetilde \Gr^W_{1,x}$ and $W_{0,x} \oplus \widetilde \Gr^W_{2,x}=(\widetilde \Gr^W_{1,x})^{\perp}$. The first $\widetilde \Gr^W_{1,x}$ is isomorphic to the symplectic similitude $F$-zip attached to the abelian part $A$ and the second $W_{0,x} \oplus \widetilde \Gr^W_{2,x}$ is ordinary. It now follows easily that $x \in S_{g,w}^{\tilde \Sigma}$ if and only if the element $w' \in {}^IW_{g-i}$ classifying the EO stratum of the abelian part $A$ satisfies $\rho_i(w')=w$.
\end{proof}
\begin{rmk}
\label{rmk-comparison-EO-tor}
The proof of Lemma~\ref{lem-compare-EvdG-zeta-tor} is similar to \cite[Lemmas 3.2.6, 3.4.3]{Lan-Stroh-stratifications-compactifications} which uses the language of Raynaud extensions instead of $1$-motives. Lemma~\ref{lem-compare-EvdG-zeta-tor} is also implicit in \cite[\S5]{Ekedahl-Geer-EO} but it seems to us that the arguments in \cite[\S5]{Ekedahl-Geer-EO} may be incomplete. We thank the referee for suggesting to us that there might be an issue with \cite[\S5]{Ekedahl-Geer-EO}.
\end{rmk}

\begin{proof}[Proof of \textnormal{\Prop~
\ref{prop-min-cpt-strata}:}]
We prove that the $S_w^{\min}$ are disjoint and locally closed by first explaining why it is true in the Siegel case and then reducing to that case. 
Since  
$ \pi$ maps a semi-abelian $k$-scheme $\tilde{A}$ to its abelian part $A$, in view of the definition of the $'S_w^{\tilde \Sigma}$ recalled above, 
the disjointness of the images $\pi('S_w^{\tilde \Sigma})$ is just the injectivity of $\rho_i:W_{g-i} \to W_g$. 
Thus the disjointness of the $S_w^{\min}$ in the Siegel case follows from the comparison Lemma~\ref{lem-compare-EvdG-zeta-tor}.

Let $\mu_g:\GG_{m,k}\to GSp(2g)$ denote the cocharacter $\varphi\circ \mu$ and $\bar \varphi^{\rm zip}:\GZip^\mu\to \GspZip^{\mu_g}$ the induced morphism of stacks. 
Write $\bar \varphi^{\Sh}, \bar \varphi^{\Sigma/\tilde \Sigma}, \bar \varphi^{\min}$ for the special fiber of $\varphi^{\Sh}, \varphi^{\Sigma/\tilde \Sigma}, \varphi^{\min}$. Combining \S\S\ref{sec-toroidal-review}-\ref{sec-minimal-review} with \Th~\ref{th-tor-ext-Gzip} gives rise to the commutative diagram:

$$\xymatrix{
& \GZip^\mu \ar[rr]^{\bar \varphi^{\rm zip}} && \GspZip^{\mu_g} \\
& \Shktoro \ar[u]_{\zeta_{\Kcal}^{\Sigma}} \ar[rr]^{\bar \varphi^{\Sigma/\tilde \Sigma}} \ar[dd]^(.3){\pi}|!{[d];[rr]}\hole && S^{\tilde \Sigma}_{g, \tilde \Kcal} \ar[u]_{\zeta_{g, \tilde \Kcal}^{\tilde \Sigma}} \ar[dd]^{\pi_{g}} \\
\Shko \ar[ru] \ar[rd] \ar[rr]^(.3){\bar \varphi^{\rm Sh}} && S_{g, \tilde \Kcal} \ar[ru] \ar[rd] \\
& \Shkmino  \ar[rr]^{\bar \varphi^{\rm min}}  && S^{\rm min}_{g, \tilde \Kcal} }$$

Let $x\to \Shkmino$ be a geometric point, and let $X:=\pi^{-1}(x)$ be its fiber in $\Shktoro$. 
Note that $X$ is connected. Combining the disjointness of the strata of $S^{\min}_{g, \tilde \Kcal}$ in the Siegel case with the commutativity of the diagram, this implies that $\bar \varphi^{\Sigma/\tilde \Sigma}(X)$ is contained in a single EO stratum of $S^{\tilde \Sigma}_{g, \tilde \Kcal}$. Hence $\zeta_{\Kcal}^{\Sigma}(X)$ is a connected subset of $\GZip^\mu$ that maps to a single point by $\bar \varphi^{\rm zip}$. Since $\bar \varphi^{\rm zip}$ has discrete fibers by \Th~\ref{discrete_fibers}, we deduce that $\zeta_{\Kcal}^{\Sigma}(X)$ is a singleton, hence $X$ is contained in a single EO stratum. This shows that
\begin{equation}
\pi^{-1}(S^{\min}_w)=S^{\Sigma}_w  \mbox{ and } 
\pi^{-1}(S^{\min,*}_w)=S^{\Sigma,*}_w.\label{fiberstormin}
\end{equation}
In particular, the $S^{\min}_w$ are pairwise disjoint. 

Next we show that the $S^{\min}_w$ are locally closed. Since $\zeta^{\Sigma}_{\Kcal}$ is continuous, $S_w^{\Sigma, *}$ is closed in $S_{\Kcal}^{\Sigma}$. The complement of $S_w^{\Sigma}$ in $S_w^{\Sigma,*}$ is also closed, as it is the finite union of the $S_{w'}^{\Sigma,*}$ for all $w'\preccurlyeq w$ in ${}^IW$. Since $\pi$ is proper, the images $\pi(S_w^{\Sigma, *})$ and $\pi(S_w^{\Sigma,*}\setminus S_w^{\Sigma})$ are both closed in $S_{\Kcal}^{\min}$. Since the $S_w^{\min}$ are disjoint, $\pi(S_w^{\Sigma})=S_w^{\min}$ is the complement of $\pi(S_w^{\Sigma,*}\setminus S_w^{\Sigma})$ in $\pi(S_w^{\Sigma,*})$. Hence $\pi(S_w^{\Sigma})=S_w^{\min}$ is locally closed.

The Stein factorization of the map $\pi:\overline{S}^{\Sigma}_w\to \overline{S}^{\min}_w$ gives a scheme $T$ with a finite map $f:T\to \overline{S}^{\min}_w$, such that $\pi$ factors through a proper map $g:\overline{S}^{\Sigma}_w\to T$ with connected fibers, satisfying $g_{*}\Ocal_{\overline{S}^{\Sigma}_w}=\Ocal_T$. By \eqref{fiberstormin}, we see that $\pi$ has connected fibers, hence $f$ has connected fibers. It follows that $f$ is a universal homeomorphism. This proves \ref{item-stein}.

Finally, we prove \ref{item-min-cpt-strata-affine}. Denote again by $\omega$ the line bundle $f^*\omega$ on $T$. Note that $\omega$ is again ample on $T$. Since $g_{*}\Ocal_{\overline{S}^{\Sigma}_w}=\Ocal_T$, it follows from the projection formula that $g_{*}\omega=\omega$. In particular, the section $h_w^\Sigma$ descends to a section $h_T\in H^0(T,\omega^{N_w})$. The non-vanishing locus $U\subset T$ of $h_T$ satisfies $g^{-1}(U)=S^{\Sigma}_w$. Using again \eqref{fiberstormin}, we deduce $U=f^{-1}(S^{\min}_w)$. Since $\omega$ is ample on $T$, the open $U$ is affine, and hence $S^{\min}_w=f(U)$ is affine by a theorem of Chevalley.
\end{proof}

\subsection{The length stratification of Hodge-type Shimura varieties}
\label{sec-length-hodge-type}
We apply \S\ref{sec-length-strata} to $\zeta_{\Kcal}^{\Sigma} :S_{\Kcal}^{\Sigma}\to \GZip^\mu$. Write $S^{\Sigma}_j$ and $S^{\Sigma,*}_j$ for the length strata in $S_{\Kcal}^{\Sigma}$. 

Assumption \ref{item-length-equidim} of \Prop~\ref{prop-length-strata} follows from \cite[\Th~2]{MadapusiHodgeTor}. Assumption \ref{item-length-nonempty} was proved in the PEL case in \cite[\Th~2]{Viehmann-Wedhorn}, and has recently been generalized to general Hodge-type Shimura varieties, see \cite{Yu-basic-locus-non-empty,Kisin-Honda-Tate-theory-Shimura-varieties,Lee-newton-strata-nonempty}. This also follows from \cite{Vasiu-Manin-problems} using a different language. As for Assumption \ref{item-length-strata3}: 
\begin{lemma} 
\label{lem-zero-dim-pel}
Suppose $\varphi$ \textnormal{(\S\ref{sec-integral-symplectic-embed})} is a PEL embedding . Then $S_e^{\Sigma}$ intersects the boundary trivially, \ie $S_e=S_e^{\Sigma}$. In particular, $S_e^{\Sigma}$ is zero-dimensional.  \end{lemma}
\begin{proof}

The lemma is trivial when $\gx$ is of compact type, so assume it is not.
By \cite[Lemma 3.2(b)]{Baily-Borel}, this implies that $\GG^{\ad}(\RR)$ has no compact factors. 
Since $S_e \subset S_e^{\Sigma}$, 
the latter is nonempty as well. 
The image $\varphi^{\Sigma/\tilde \Sigma}(S_e^{\Sigma}) \subset S_{g, \tilde \Kcal}^{\tilde \Sigma}$ is contained in a unique Siegel stratum $S_{g,w}^{\tilde \Sigma}$ (for a unique $w \in {}^IW_g$). So it suffices to show that $S_{g,w}^{\tilde \Sigma}$ does not meet the boundary of $S_{g, \tilde K}^{\tilde \Sigma}$. Recall that the multiplicative rank of a group scheme $H$ of $p$-power order is $\mult \rk(H):=\dim \Hom (\mmu_p, H)$. The desired inclusion $S_{g,w}^{\tilde \Sigma} \subset S_{g,\tilde \Kcal}$ follows from the following claims:    \begin{enumerate}[label=(\alph*)]
\item \label{item-p-rank-semi-abelian}
For every boundary point $x$ of $S_{g, \tilde \Kcal}^{\tilde \Sigma} \setminus S_{g, \tilde \Kcal}$, the underlying semi-abelian variety $\tilde{A}$ \eqref{eq-semi-abelian-def} has $\mult \rk \tilde{A}[p]>0$.
\item \label{item-p-rank-constant-along-strata}
The multiplicative rank is constant along the strata $S_{g,w}^{\tilde \Sigma}$ for $w \in W_g$.
\item \label{item-p-rank-zero-se}
The stratum $S_{g,w}^{\tilde \Sigma}$ containing $S_e^{\Sigma}$ has multiplicative rank $0$.
\end{enumerate}
Claim~\ref{item-p-rank-semi-abelian} is clear: A boundary point $x \in S_{g, \tilde \Kcal}^{\tilde \Sigma}\setminus S_{g, \tilde \Kcal}$ necessarily has a nontrivial torus part $T$; if the torus part has dimension $\dim T=j>0$, one has $\mult \rk \tilde{A} \geq j$. 

Now consider \ref{item-p-rank-constant-along-strata}. It follows from \cite[\S15]{Mumford-Abelian-Varieties-book} that $\mult \rk \tilde {A}$ is the semisimple rank of $V: H^0(\tilde{A}, \Omega^1_{\tilde{A}}) \to H^0(\tilde{A}, \Omega^1_{\tilde{A}_x})^{(p)}$, since $V$ induces an isomorphism on the differentials of the torus part $T$. In view of the transpose relation between $F$ and $V$, $s$ is also the semisimple rank of $F:(\ker V)^{(p)}\to \ker V$.  Since the semisimple rank is equal to the stable rank (\ie the rank of all sufficiently large powers of $F$), one sees that $s$ is an invariant of the $GSp(2g)$-Zip $I^{\tilde \Sigma}_x$ associated to $x$ by \Th~\ref{th-tor-ext-Gzip}.

To prove \ref{item-p-rank-zero-se}, by \ref{item-p-rank-constant-along-strata} it is enough to show that the multiplicative rank of the underlying abelian variety of a point of $S_e$ is zero, under the assumption that $\GG^{\ad}(\RR)$ has no compact factors. This can be checked case-by-case using Moonen's ``standard objects'' \cite[\S\S4.9,5.8]{moonen-gp-schemes}. (Given $w \in \iw$, the associated standard object is an explicit representative of the isomorphism class of the $G$-Zips of type $w$.)
\end{proof}
It is expected that Lemma~\ref{lem-zero-dim-pel} remains true for an arbitrary Shimura variety of Hodge type. For now we single this out as a condition; it will play an important role in \S\S\ref{sec-fact-hecke},\ref{sec-galois-reps}.
\begin{condition} \label{cond-zero-dim}
The minimal EO stratum $S_e^{\Sigma}$ has dimension zero. 
\end{condition}
\begin{rmk}
\label{rmk-zero-dimensional-EO-stratum}
As in Lemma~\ref{lem-zero-dim-pel}, Condition~\ref{cond-zero-dim} holds when $S_e^{\Sigma}$ does not meet the boundary, \ie when $S_e=S_e^\Sigma$.
\end{rmk}
Putting together \Props~\ref{prop-hj},~\ref{prop-length-strata} and Lemma~\ref{lem-zero-dim-pel} gives the following result on length strata of Hodge-type Shimura varieties and their Hasse invariants:
\begin{corollary}\label{cor-length-shimura}
Assume $\gx$ is of PEL-type, or that $S_{ \Kcal}^{ \Sigma}$ satisfies \textnormal{Condition \ref{cond-zero-dim}}. Then:
\begin{enumerate}[label=(\alph*)]
\item The length strata $S^{\Sigma}_j$ and $S^{\Sigma,*}_j$ are equi-dimensional of dimension $j$, and $S^{\Sigma}_j$ is open dense in $S^{\Sigma,*}_j$.
\item For $w\in {}^I W$, $S^{\Sigma}_w$ is equi-dimensional of dimension $\ell(w)$.
\item For every $j$, there exists $N_j \geq 1$ and a $\gofafp$-equivariant $h_j \in H^0(S_j^{\Sigma,*}, \omega^{N_j})$ satisfying $\nonvanish (h_j)=S_{j}^{\Sigma}$.  
\end{enumerate} 
\end{corollary}
We call the $h_j$ the length Hasse invariants of $S^{\Sigma}_{\Kcal}$.
\section{Geometry of Hasse regular sequences} 
\label{sec-vanish-strata}

\subsection{Cohomological vanishing on $\Shktorn$} \label{sec-vanishing-on-shktor}

The following variant of "cohomology and base change" will be used to reduce vanishing statements to the special fiber.

\begin{lemma}\label{vanishn}
Let $(R,\mfr)$ be a Noetherian local ring, $f:X\to \spec R$ a proper morphism, and $\Fcal$ a coherent $\Ocal_X$-module, flat over $R$. Denote by $x$ the point $\mfr\in \spec R$. Assume that $H^i(X_x,\Fcal_x)=0$ for $i\geq 0$. Then $H^i (X,\Fcal)=0$.

In particular, $H^i (X_n,\Fcal_n)=0$ for all $n\geq 1$, where $X_n=X\times_R R/\mfr^n$ and $\Fcal_n$ is the pullback of $\Fcal$ via $X_n\to X$.
\end{lemma}

\begin{proof}
The base change map $\varphi^i(x):R^i f_*(\Fcal)_x \otimes k(x) \longrightarrow H^i(X_x, \Fcal_x)$ is surjective, because its target is zero by assumption. By ``cohomology and base change'' (see \cite[\Th~12.11(a)]{Hartshorne-Alg-Geom} in the projective case and \cite[\S7]{EGAIII.2} in general), the map $\varphi^i(x)$ is an isomorphism. Hence $R^i f_*(\Fcal)_x \otimes k(x)=0$. Since $\spec R$ is affine, $R^i f_*(\Fcal)=H^i(X,\Fcal)^\sim$, and  $H^i(X,\Fcal)$ is a finite type $R$-module since $f$ is proper. So the first statement follows from Nakayama's lemma.

The second part of the lemma follows from the first one applied to the base change $X_n \to \spec R/\mfr^n$.
\end{proof}

The next condition on the vanishing of higher direct images is key to studying cohomological vanishing on $\Shktorn$ and its subschemes. Recall that $S_{\Kcal}^{\Sigma}:=\Sscr_{\Kcal}^{\Sigma,1}$. Denote by $\pi:\Shktoro \to \Shkmino$ the natural map.
\begin{condition} 
\label{cond-hdi} Let $\eta \in \chargpldom$.  The pair $(\Sscr_{\Kcal}^{\Sigma},\eta)$ satisfies $R^i\pi_* \vsubeta=0$ for all $i>0$.  
\end{condition}
\begin{rmk} 
\label{rmk-hdi}
In the PEL case, Condition~\ref{cond-hdi} is a theorem of Lan \cite[Theorem 8.2.1.2]{Lan-Ordinary-Loci}. See also \cite{Lan-Stroh-HDI} for a simpler argument under a mild additional hypothesis on $p$.
In the general Hodge case, ~\ref{cond-hdi} is a theorem of Stroh\footnote{Stroh's result is unpublished, but it follows directly from the combination of \cite{Lan-Stroh-HDI} and \cite{Lan-Suh-lifting-cusp-forms}.} in \cite{Stroh-relative-cohomology-cuspidal-forms} when $\eta \in \mathbf Q \eta_{\omega}$. 

\end{rmk}

Recall the convention regarding notation established in \S\ref{sec-reduct-mod-pn}: $\Sscr_{\Kcal}^{\Sigma, +}:=\Sscr_{\Kcal}^{\Sigma}$ and $\Sscr_{\Kcal}^{\Sigma, 0}$ is its generic fiber.
\begin{lemma}
\label{lem-tor-min-v1} 
Assume $(S_{\Kcal}^{\Sigma},\eta)$ satisfies \textnormal{Condition~\ref{cond-hdi}}. Then there exists $m_0=m_0(\eta) \geq 1$ such that, for all $n \in \zgeqz \cup \{+\}$, all $m \geq m_0$ and all $i>0$, one has \begin{equation} H^i(\Shktorn, \mathcal V^{\sub}_{\eta} \otimes \omega^m)=0. \label{eq tor min v1} 
 \end{equation} 
 \end{lemma}

\begin{proof}  By cohomology and base change (Lemma~\ref{vanishn}), it suffices to prove~\eqref{eq tor min v1} when $n=1$. By Condition~\ref{cond-hdi} and the projection formula, the Leray spectral sequence degenerates for $\pi$ and $\vsubeta \otimes \omega^m$. Hence
\begin{equation} \label{eq tor min}
H^i(\Shktoro, \vsubeta  \otimes \omega^m)=H^i(\Shkmino, \pi_* (\vsubeta \otimes \omega^m)) =H^i(\Shkmino,  (\pi_* \vsubeta) \otimes \omega^m). 
\end{equation}
 Since $\pi$ is proper, $\pi_*\vsubeta$ is coherent. Since $\omega$ is ample on $\Shkmin$, \eqref{eq tor min} is zero for all sufficiently large $m$ by Serre's cohomological criterion for ampleness.\end{proof}

\begin{corollary} 
Assume $(\Shko, \eta)$ satisfies \textnormal{Condition~\ref{cond-hdi}}. Let $m_0=m_0(\eta)$ be as in  {\rm Lemma~\ref{lem-tor-min-v1}}. Then  \begin{equation} H^0(\Shktor, \vsubeta \otimes \omega^m) \longrightarrow H^0(\Shktorn, \vsubeta \otimes \omega^m) \label{eq reduction mod pn} \end{equation} is surjective for all $m \geq m_0$. \label{cor surjective mod pn} \end{corollary}

\begin{proof} Let $\varpi \in \Ocal_\pfr$ be a uniformizer. Multiplication by $\varpi^n$ yields a short exact sequence \begin{equation}
0  \longrightarrow  \mathcal O_{\pfr} \xrightarrow{\varpi^n} \mathcal O_{\pfr} \longrightarrow \mathcal O_{\pfr}/\pfr^n  \longrightarrow  0\label{eq reduction mod pn short seq} \end{equation}
and a long exact sequence of cohomology
\begin{equation}
\xymatrix@C=1pc{ 
H^0(\Shktor, \vsubeta \otimes \omega^m) \ar[r] & 
H^0(\Shktorn, \vsubeta \otimes \omega^m) \ar[r] & 
H^{1}(\Shktor, \vsubeta \otimes \omega^m) } \label{eq reduction mod pn long seq}. \end{equation} By Lemma~\ref{lem-tor-min-v1}, the right-hand term in~\eqref{eq reduction mod pn long seq} is zero when $m \geq m_0$. \end{proof}
\subsection{Hasse-regular sequences} 
\label{sec hasse-reg}
The next definition singles out those $\gofafp$-equivariant subschemes of a Shimura variety (modulo a prime power) which are obtained as successive zero schemes of length Hasse invariants (\S\ref{sec-length-strata}, \S\ref{sec-length-hodge-type}). 
Let $(\Sscr_{\Kcal}^{\Sigma,n})_{\Kcal^p, \Sigma}$ be the reduction modulo $\pfr^n$ of the tower introduced in \S\ref{sec-gofafp-toro}.

\begin{definition}
\label{def-hasse-reg-seq}
A  \emph{Hasse-regular sequence of  length} $r$ in $(\Shktorn)_{\Kcal^p, \Sigma}$ is the data, for each pair $(\Kcal^p, \Sigma)$ of a sequence  $(Z_j, a_j, f_j)_{j=0}^r$ satisfying: \begin{enumerate}[label=(\alph*)]
\item One has $Z_0=\Sscr_{\Kcal}^{\Sigma, n}$ and $(a_j)_{j=0}^r$ is a sequence of positive integers independent of $\Kcal^p,\Sigma$.
\item 
For all $j>0$, $Z_{j} \subset Z_{j-1}$ is a closed subscheme with $(Z_{j})_{\red}=S_{d-j}^{\Sigma,*}$, the $(d-j)$th length stratum \textnormal{(\S\ref{sec-length-hodge-type})}. 
\item For all $j$, $f_j\in H^0(Z_j,\omega^{a_j})$ is a section constructed by applying \textnormal{\Th~\ref{th-glue}} to the length Hasse invariant $h_{d-j}$ \textnormal{(\Cor~\ref{cor-length-shimura})}.
\item For all $j<r$, the zero scheme of $f_j$ is $Z_{j+1}$. \end{enumerate}
\end{definition}
\begin{rmk} Let $(Z_j, a_j, f_j)_{j=0}^r$ be a Hasse-regular sequence.
\label{rmk-hasse-reg}
\begin{enumerate}[label=(\alph*)]
\item Given $0 \leq s \leq r$, the truncation $(Z_j, a_j, f_j)_{j=0}^s$ is a Hasse regular sequence of length $s$.
\item
\label{item-hasse-reg-gofafp-equiv}
Since the length Hasse invariants $h_{d-j}$ are $\gofafp$-equivariant, it follows inductively from the uniqueness in \Th~\ref{th-glue}\ref{item-lift1} that the systems of $Z_j$ and $f_j$ in a Hasse-regular sequence are $\gofafp$-equivariant and compatible with refinement of rpcd $\Sigma$.
\item
\label{item-hasse-reg-dim}
By \Cor~\ref{cor-length-shimura}, one has $\dim Z_j=d-j$ for all $j$.
\item
\label{item-hasse-reg-hecke-action}
In \S\ref{sec-hecke-hasse-regular}, it is shown that the coherent cohomology of $Z_j$ admits an action of the unramified Hecke algebra $\Hcal$ (away from $p$), compatible with the inclusions $Z_j \hookrightarrow Z_{j-1}$. 
\item It is possible (and indeed happens in our arguments) that the last section $f_r$ is nowhere vanishing. 

\end{enumerate}
\end{rmk}
\begin{definition}
\label{def-Hasse-reg-subscheme}
A system of subschemes  $Z \subset \Sscr_{\Kcal}^{\Sigma,n}$ is {\em Hasse regular} if $Z=Z_r$ for some Hasse regular sequence $(Z_j, a_j, f_j)_{j=0}^r$.  
\end{definition}

%\begin{rmk} We shall sometimes denote $Z_j$ by $Z^{d_j}$ where $d_j=\dim Z_j$. This notation is unambiguous because no two subschemes pertaining to the same Hasse-regular sequence have the same dimension \label{rmk hecke reg dim notation} \end{rmk}

\begin{lemma} Let $(Z_j, a_j, f_j)_{j=0}^r$ be a Hasse-regular sequence in $\Shktorn$. Then $Z_j$ is Cohen-Macaulay for all $j$. 
\label{lem-cohen-macaulay-regular-sequence} 
\end{lemma}

\begin{proof} 
First note that the scheme $\Shktorn$ is Cohen-Macaulay by \cite[\Th~23.9]{Matsumura-commutative-ring-theory}, since the ring $\mathcal O_E/\pfr^n$ is Cohen-Macaulay and $\Shktorn$ is smooth over $\mathcal O_E/\pfr^n$.
Then the result follows from \cite[\Th~17.3]{Matsumura-commutative-ring-theory}, which implies that an effective Cartier divisor in a Cohen-Macaulay scheme is again Cohen-Macaulay.
\end{proof}

\begin{corollary}
\label{cor-hasse-reg-injective}
Let $(Z_j, a_j, f_j)_{j=0}^r$ be a Hasse-regular sequence in $\Shktorn$. Then $f_j$  is injective for all $j$.
\end{corollary}
\begin{proof} By Macaulay's "Unmixedness Theorem", a Cohen-Macaulay scheme has no embedded components, \cf \cite[\Cor~18.14]{eisenbud-commutative-algebra}. In view of Lemma~\ref{lem-cohen-macaulay-regular-sequence}, $Z_j$ has no embedded components. Then the result follows from \Th~\ref{th-glue}\ref{item-lift2} and \Cor~\ref{cor-length-shimura}.
\end{proof}

\subsection{Cohomological vanishing on Hasse-regular subschemes of $\Shktorn$} \label{sec coh van hasse-reg}

\begin{lemma}
\label{lem hecke regular vanishing} Suppose $\Zcal=(Z_j, a_j, f_j)_{j=0}^r$ is a Hasse-regular sequence and $\eta \in \chargpldom$. Assume $(\Shko, \eta)$ satisfies \textnormal{Condition~\ref{cond-hdi}}. Then there exists an integer $m_0$ such that, for all $i>0$, $m\geq m_0$, and $j\in \{0,...,r\}$, one has 
\begin{equation} H^i(Z_j, \mathcal \vsubeta \otimes \omega^m)=0 \label{eq strata v1}. \end{equation}  \label{lem vanishing for strata}\end{lemma}

\begin{proof} By induction on the length $r$. The base case $r=0$ is given by Lemma~\ref{lem-tor-min-v1}. For all $s \geq 0$, multiplication by $f_{r-1}$ and twisting by $\omega^s$ gives a short exact sequence of sheaves on $Z_{r-1}$:
$$0 \to \vsubeta \otimes \omega^s \to \vsubeta \otimes \omega^{a_{r-1}+s} \to \vsubeta \otimes \omega^{a_{r-1}+s}|_{Z_r}  \to  0.$$
 It induces a long exact sequence in cohomology:
\begin{equation} H^i(Z_{r-1}, \vsubeta \otimes \omega^{a_{r-1}+s}) \to H^i(Z_{r}, \vsubeta \otimes \omega^{a_{r-1}+s}) \to H^{i+1}(Z_{r-1}, \vsubeta \otimes \omega^s). \end{equation}
By induction, the extremal terms are zero for sufficiently large $s$, hence so is the middle term.
\end{proof}

\part{Hecke algebras and Galois representations} \label{part-galois}
Part~\ref{part-galois} contains our applications of group-theoretical Hasse invariants to the action of the Hecke algebra on coherent cohomology and the existence of `automorphic' Galois representations.
Our most fundamental application to coherent cohomology Hecke algebras is the factorization theorem (\Th~\ref{th-reduction-to-h0}, compare \Th~\ref{th-intro-factor}). \S\ref{sec-fact-hecke} is concerned with the statement and the majority of the proof of this result. The flag space of a Hodge-type Shimura variety is introduced in \S\ref{sec-flag-space}. It is used to complete the proof of \Th~\ref{th-reduction-to-h0} in \S\ref{sec-increased-regularity}. Our results on Galois representations are contained in \S\ref{sec-galois-reps}. The application to `Serre's letter to Tate' constitutes \S\ref{sec-serre-letter}.

\section{Construction of Hecke factorizations} \label{sec-fact-hecke}

\subsection{Hecke algebras} \label{sec-hecke-algebras}

\subsubsection{The abstract Hecke algebra}
\label{sec-hecke-abstract}
 Fix a prime $p \not \in \Ram(\GG)$ (\S\ref{sec-def-ramification}).
For every $v \not \in \Ram(\GG) \cup \{p\}$, let $\Kcal_v$ be a hyperspecial subgroup of $\gofqv$. 
Let $\Hcal_v=\Hcal_v(\gv, \Kcal_v;\ZZ_p)$ be the unramified (or spherical) Hecke algebra of $\GG$ at $v$, with $\zp$-coefficients, normalized by the unique Haar measure which assigns the hyperspecial subgroup $\Kcal_v$ volume $1$. 
Define the unramified, global Hecke algebra by \begin{equation}
\label{eq-def-hecke-alg}
\Hcal=\Hcal(\GG)= \bigotimes_{v \not \in \Ram(\GG) \cup\{p\}}\mathcal H_v  \hspace{2cm} \mbox{(restricted tensor product).} \end{equation}

\subsubsection{Systems of Hecke eigenvalues}
\label{sec-hecke-systems-e-values}
Suppose $\kappa$ is a field and $M$ is a finite dimensional $\kappa$-vector space which is also an $\Hcal$-module. Recall that a system of Hecke eigenvalues $(b_T)_{T\in \Hcal}$ appears in $M$ if there exists a finite extension $\kappa'/\kappa$ and $m \in M\otimes \kappa'$ such that $Tm=b_Tm$ in $M \otimes \kappa'$ for all $T \in \Hcal$.

\subsubsection{Coherent cohomology Hecke algebras}
\label{sec-hecke-min} Let $\gx$ be a Shimura datum of Hodge type. 
Recall the associated structure theory (\S\ref{sec-shimura-varieties-hodge-type}) and the  symplectic embedding of toroidal compactifications  $\varphi^{\Sigma/\tilde \Sigma}:\Sscr_{\Kcal}^{\Sigma} \to \Sscr_{g, \tilde \Kcal}^{\tilde \Sigma}$ \eqref{eq-varphi-tor}. We choose $\tilde \Sigma$ and $\Sigma$ smooth, so that $\Sscr_{\Kcal}^{\Sigma}$ is smooth. Let $\eta \in \chargpldom$.  
By \S\S\ref{sec-torsors}, \ref{sec-toroidal-bundles}, one has the coherent cohomology modules $H^i(\Shkn, \veta)$, $H^i(\Sscr_{\Kcal}^{\Sigma,n}, \vcaneta)$ and $H^i(\Sscr_{\Kcal}^{\Sigma,n}, \vsubeta)$ for every $n \in \zgeqz \cup \{+\}$ and every $i \geq 0$ (notation as in \S\ref{sec-reduct-mod-pn}). 
We explain how the Hecke algebra $\Hcal$ acts on each of these modules. 
This is standard for $H^i(\Shkn, \veta)$, but for $H^i(\Sscr_{\Kcal}^{\Sigma,n}, \vcaneta)$ and $H^i(\Sscr_{\Kcal}^{\Sigma,n}, \vsubeta)$ the definition of the trace maps is less obvious. 
For $\Sscr_{\Kcal}^n$, we also describe an action of $\Hcal$ on $H^i(\Sscr_{\Kcal}^n, \Fscr)$ where $\Fscr$ is an arbitrary (not necessarily locally free) $\gofafp$-equivariant coherent sheaf. 
This will be used in \S\ref{sec-serre-letter} with $\Fscr$ the ideal sheaf of $S_e$ in 
$\Sscr^1_{\Kcal}=S_{\Kcal}$. 

Let $g \in \gofqv$. For each of the spaces above, we define below an associated Hecke-operator $T_g$.
Since $\Hcal_v$ is generated by the characteristic functions of double cosets represented by $g \in \gofqv$, one deduces an action of all of $\Hcal$.

\subsubsection{Hecke action I: $\Sscr_{\Kcal}^n$}
\label{sec-hecke-open-sh}
Put $\Kcal_g=\Kcal \cap g\Kcal g^{-1}$. 
Then $g^{-1} \Kcal_g g=\Kcal \cap g^{-1} \Kcal g$. 
Let $\pi_1:=\pi_{\Kcal_g/\Kcal}$ and $\pi_2:=\pi_{ g^{-1} \Kcal_g g/\Kcal} \circ g$, where $g: \Sscr^n_{\Kcal_g} \stackrel{\sim}{\to} \Sscr^n_{g^{-1}\Kcal_g g}$ (\S\ref{sec-shimura-integral}).  
Then the  $\pi_j:\Sscr_{\Kcal_g}^{n} \to \Sscr_{\Kcal}^{n}$ are finite \'etale projections. Thus there is a trace map  $$\tr \pi_j : \pi_{j,*} \Ocal_{\Sscr_{\Kcal_g}^n} \to \Ocal_{\Sscr_{\Kcal}^n}. $$ 
Let $\Fscr$ be a $\gofafp$-equivariant coherent sheaf on the tower $(\Sscr_{\Kcal}^n)_{\Kcal^p}$ (\S\ref{sec-gofafp-objects-1}). Then $\pi_1^*\Fscr=\pi_2^*\Fscr$. For a finite morphism, the projection formula is valid for arbitrary coherent sheaves since then the direct image is exact.
\begin{comment}
(\cf \cite[Lemma~5.7]{Arapura-Frobenius-ampleness})
\end{comment}
Hence the Hecke operator $$T_g:H^i(\Shkn, \Fscr) \to H^i(\Shkn, \Fscr) $$ 
is defined as usual by push-pull: $T_g:=\tr \pi_2  \circ \pi^*_1$.

\subsubsection{Hecke action II: $\Sscr_{\Kcal}^{\Sigma,n}$}
\label{sec-hecke-tor-sh} For $j \in \{1,2\}$, set $\Sigma_g^j:= \pi_j^*\Sigma$. Then the $\Sigma_g^j$ are both admissible, finite rpcd's for $(\GG, \XX, \Kcal_g)$.  
 Let $\Sigma_g$ be a common, smooth refinement of $\Sigma^1_g$ and $ \Sigma^2_g$. Let $\pi'_j:\Sscr_{\Kcal_g}^{\Sigma_g,n} \to \Sscr_{\Kcal}^{\Sigma,n}$ be the corresponding projections. 
 Each is a composition of a refinement morphism  $r^j:\Sscr_{\Kcal_g}^{\Sigma_g,n} \to \Sscr_{\Kcal_g}^{\Sigma^j_g,n}$  corresponding to $\Sigma_g \subset \Sigma_g^j$ and a projection $\pi_j^0:\Sscr_{\Kcal_g}^{\Sigma_g^j,n} \to \Sscr_{\Kcal}^{\Sigma,n}$. 
By \cite[Proof of 5.1.3]{MadapusiHodgeTor}, the $r^j$ satisfy 
 $r^j_* \Ocal_{\Sscr_{\Kcal_g}^{\Sigma_g,n}}=\Ocal_{\Sscr_{\Kcal_g}^{\Sigma_g^j,n}}$ and $R^i r^j_* \Ocal_{\Sscr_{\Kcal_g}^{\Sigma_g,n}}=0$ for all $i>0$. 
 \begin{comment}
 Madapusi ref 2018 checked
 \end{comment}
 Since $\vcaneta$ and $\vsubeta$ are locally free and $\gofafp$-equivariant (\S\ref{sec-gofafp-toro}), combining the properties of $r^j$ with the projection formula gives 
 \begin{equation}
 \label{eq-cohom-refinement}
 H^i(\Sscr_{\Kcal_g}^{\Sigma_g,n}, \Vscr^?(\eta))
 =
 H^i(\Sscr_{\Kcal_g}^{\Sigma^j_g,n}, \Vscr^?(\eta))
 \end{equation}
 for $? \in \{\can, \sub\}$.

 The  $\pi_j^0$ are finite and we claim they are also flat (we learned this fact from J. Tilouine \cite[\S8.1.2, p. 1409]{tilouine-dual-bgg} and we thank B. Stroh and D. Rydh for discussions about why it is true). It suffices to check the flatness over $\Ocal_{E,\pfr}$; the corresponding statement over $\Ocal_{E,\pfr}/\pfr^n$ then follows by base change.
 
Observe that all of the toroidal compactifications considered here are Cohen-Macaulay over $\Ocal_{E,\pfr}$; specifically we need to justify why the $\Sscr_{\Kcal_g}^{\Sigma_g^j}$, which are associated to potentially non-smooth rpcd, are Cohen-Macaulay. Recall that, given any finite rpcd $\Sigma'$ and any base scheme $S$, one has an associated torus embedding $X_{\Sigma'}(S)$ over $S$, functorial in $\Sigma'$, as described in \cite[\Chap~IV, \Th~2.5]{Chai-Faltings-book}. By construction, the toroidal compactifications $\Sscr_{\Kcal_g}^{\Sigma_g^j}$ are \'etale locally torus embeddings of this type over $\Ocal_{E,\pfr}$. Thus we are reduced to checking that a torus embedding $X_{\Sigma'}(R)$ over a DVR $R$ is Cohen-Macaulay (the argument which follows works equally well over regular local $R$). By \cite[\Chap~1, \S3, \Th~14]{toroidal-embeddings} a torus embedding over a field is Cohen-Macaulay. Since a torus embedding over $R$ is flat over $R$, the result follows from \cite[\Th~23.9]{Matsumura-commutative-ring-theory}. Thus $\Sscr_{\Kcal_g}^{\Sigma_g^j}$ is Cohen-Macaulay. 

Since the target $\Sscr_{\Kcal}^{\Sigma}$ of $\pi_j^0$ is regular (recall that $\Sigma$ was assumed smooth), we conclude that $\pi^0_j$ is flat by \cite[\Th~23.1]{Matsumura-commutative-ring-theory}, which implies that a quasi-finite morphism between equi-dimensional, locally Noetherian schemes of the same dimension is flat if the domain is Cohen-Macaulay and the target is regular.

The finite flatness of $\pi_j^0$ gives a trace map $$\tr \pi_j^0 : \pi^0_{j,*} \Ocal_{\Sscr_{\Kcal_g}^{\Sigma^j_g,n}} \to \Ocal_{\Sscr_{\Kcal}^{\Sigma,n}}. $$ 
The trace gives a map $H^i(\Sscr_{\Kcal_g}^{\Sigma^j_g,n}, \Vscr^?(\eta)) \to H^i(\Sscr_{\Kcal}^{\Sigma,n}, \Vscr^?(\eta))$. Together with~\eqref{eq-cohom-refinement}, this yields $$\tr \pi'_j: H^i(\Sscr_{\Kcal_g}^{\Sigma_g,n}, \Vscr^?(\eta)) \to H^i(\Sscr_{\Kcal}^{\Sigma,n}, \Vscr^?(\eta)).$$ 
Having constructed $\tr \pi'_j$, our desired Hecke operators $$T_g^{\Sigma}:H^i(\Sscr_{\Kcal}^{\Sigma,n}, \Vscr^?(\eta)) \to H^i(\Sscr_{\Kcal}^{\Sigma,n}, \Vscr^?(\eta))$$ are given as before: $T_g^{\Sigma}:=\tr \pi'_2 \circ \pi'^*_1$. It follows again from the properties of the refinement morphisms $r^j$ that  $T_g^{\Sigma}$ is independent of the choice of $\Sigma_g$ (given another choice of smooth refinement $\Sigma'_g$ of both $\Sigma_g^1$ and $\Sigma_g^2$, choose a common, smooth refinement $\Sigma''_g$ of $\Sigma_g, \Sigma'_g$).

Finally, note that the definition of $T_g^{\Sigma}$ reduces to that of $T_g$ away from the boundary. The inclusion $\Sscr_{\Kcal}^n \to \Sscr_{\Kcal}^{\Sigma, n}$ is $\gofafp$-equivariant; it induces an $\Hcal$-equivariant restriction map $H^0(\Sscr_{\Kcal}^{\Sigma, n}, \vcaneta) \to H^0(\Sscr_{\Kcal}^{n}, \veta)$. The latter is an isomorphism when $\gx$ satisfies the Koecher principle. 

\subsubsection{Hecke action III: Hasse-regular subschemes}
\label{sec-hecke-hasse-regular}
For $j \in \{1,2\}$, let  $r^j$ and $\pi^0_j$ be as in \textnormal{\S\ref{sec-hecke-tor-sh}}. Let $(Z^0_{t}, a_t, f_{t,0})_{t=0}^r$ (resp. $(Z_{t}, a_t, f_{t})_{t=0}^r$) denote the component of a Hasse regular sequence for $(\Kcal, \Sigma)$ (resp. $(\Kcal^g, \Sigma_g)$). 
Even though the $\Sigma_g^j$ may not be smooth, we may still define a component $(Z^j_{t}, a_t, f_{t,j})_{t=0}^r$ of the Hasse regular sequence for $(\Kcal_g, \Sigma_g^j)$ by pullback: 
Put $Z^j_{t}:=(\pi_j^0)^{-1} (Z^0_{t})$ and $f_{t,j}:=(\pi^{0}_j)^* f_{t,0}$. Then $(r^j)^{-1}(Z^j_t)=Z_t$ and $(r^{j})^* f_{t,j}=f_t$ are the components for $\Sigma_g$.    
Let $y^{t}:Z_{t} \to \Sscr_{\Kcal_g}^{\Sigma_g}$, $y^{t,0}:Z_t^0 \to \Sscr_{\Kcal}^{\Sigma}$ and $y^{t,j}:Z_{t}^{j} \to \Sscr_{\Kcal_g}^{\Sigma^j_g}$ denote the inclusions. 
\begin{lemma} 
\label{lem-hecke-def-hasse-regular}
 For all $t$, one has:
\begin{enumerate}[label=(\alph*)]
\item 
\label{item-hasse-reg-refine}
For all $i>0$, $R^i r^j_* (y^t_* \Ocal_{Z_t})=0$ and $ r^j_* (y^t_* \Ocal_{Z_t})=y^{t,j}_*
\Ocal_{Z_t^j}$.
\item
\label{item-hasse-reg-finite-flat}
The pullback of $\pi_j^0$ along the inclusion  $y^{t,0}:Z_t^0 \to \Sscr_{\Kcal}^{\Sigma}$ is a finite flat map $Z_t^j \to Z_t^0$.
\end{enumerate}
\end{lemma}
\begin{proof} Part~\ref{item-hasse-reg-finite-flat} is a direct consequence of the finite flatness of $\pi_j^0$.
We prove~\ref{item-hasse-reg-refine} by induction on $t$: The case $t=0$ was done in \S\ref{sec-hecke-tor-sh}. Assume the two statements hold up to $t$ and consider them for $t+1$. By \Def~\ref{def-hasse-reg-seq} and the affineness of $y^{t}$, multiplication by $f_t$ gives a short exact sequence:
\begin{equation}
0 \to y^t_*(\omega^{-a_t} |_{Z_t}) 
\to 
y^t_* \Ocal_{Z_t} \to y^{t+1}_* \Ocal_{Z_{t+1}} \to 0.
\end{equation}
Apply $r^j_*$ and consider the long exact sequence of higher direct images. By the induction hypothesis,
$R^i r^j_* (y^t_* \Ocal_{Z_t})=0$ for all $i>0$. Since $(r^{j})^*(y^{t,j}_* \omega^{-a_t}|_{Z^j_t})=y^t_* \omega^{-a_t}|_{Z_t}$, the projection formula gives
$$R^i r^j_* (y^t_* \omega^{-a_t}|_{Z_t})=y^{t,j}_*\omega^{-a_t}|_{Z_t^j} \otimes R^i r^j_* (y^t_* \Ocal_{Z_t})=0$$ for  all $i>0$.
Thus $R^i r^j_* (y^{t+1}_* \Ocal_{Z_{t+1}})=0$ for all $i>0$, as it lies between two terms that are zero. Similarly, the induction hypothesis and projection formula also give \begin{equation}
0 \to y^{t,j}_*\omega^{-a_t}|_{Z_t^j} \to  y^{t,j}_* \Ocal_{Z_t^j} \to r^j_* y^{t+1}_* \Ocal_{Z_{t+1}} \to 0, 
\end{equation}
since $R^1 r^j_*( y^{t}_*\omega^{-a_t}|_{Z_t})=0$. Thus 
$r^j_* (y^{t+1}_* \Ocal_{Z_{t+1}})= y^{t+1,j}_* \Ocal_{Z_{t+1}^j}$. This proves~\ref{item-hasse-reg-refine}.
\end{proof}
 Using Lemma~\ref{lem-hecke-def-hasse-regular}, we may apply the method of \S\ref{sec-hecke-tor-sh} to $Z^0_t$ for every $t$. 
 This gives a Hecke operator 
 $T_g^{Z^0_t,\Sigma}:H^i(Z^0_t, \Vscr^?(\eta)) \to H^i(Z^0_t, \Vscr^?(\eta))$ for every $t$ and $? \in \{\can, \sub\}$. By construction, the Hecke actions on $Z^0_t$ and $Z^0_{t-1}$ are compatible: If $\alpha:Z^0_t \to Z^0_{t-1}$ denotes the inclusion, then the $\gofafp$-equivariant short exact sequence 
\begin{equation}
\label{eq-def-hecke-short-exact}
0 \to \omega^{-a_t} \to \Ocal_{Z^0_t} \to \alpha_* \Ocal_{Z^0_{t+1}} \to 0 \end{equation} induces an 
$\Hcal$-equivariant long exact sequence in cohomology. By the projection formula, the same is true if we tensor~\eqref{eq-def-hecke-short-exact} with either $\vcaneta$ or $\vsubeta$. 
\subsubsection{Notation for coherent cohomology Hecke algebras}
\label{sec-hecke-notation}
Let $\Hcal^{i,n}(\eta)$ denote the image of $\Hcal$ in $\End (H^i(\Shktorn, \vsubeta))$; we omit $\Kcal$ and $\Sigma$ from the notation. If $Z$ is a Hasse regular subscheme (\Def~\ref{def-Hasse-reg-subscheme}) of $\Shktorn$, then write $\Hcal^{i,n}_Z(\eta)$ for the image of $\Hcal$ in $\End (H^i(Z, \vsubeta))$. An arrow $\Hcal \to  \Hcal^{i,n}(\eta)$ or $\Hcal \to  \Hcal_Z^{i,n}(\eta)$ will always signify the defining projection.

\subsection{Statement of the factorization theorem} \label{sec-main-results-fact-hecke}
From now on, in addition to assuming $\Sigma$ to be smooth, we assume $\Sigma$ is a refinement of $\varphi^*\tilde \Sigma$ for some $\tilde \Sigma$ which admits $\beta$ log integral (\S\ref{sec-univ-semi-abel}). Recall that this can always be achieved by refining $\Sigma$. Then we may apply \Th~\ref{th-tor-ext-Gzip} and the results of \S\ref{sec-length-hodge-type},~\S\ref{sec-vanish-strata} which depend on it.

For every triple $(i, n, \eta)$ with $i \in \zgeqz$, $n \in \zgeqo$ and $\eta \in \chargpldom$, let 
\begin{equation}
\label{eq-reduction-to-h0}
F(i, n, \eta)=\{\eta' \in \chargpldom \ | \ \Hcal \longrightarrow \mathcal H^{i, n}(\eta) \mbox{ factors through } \Hcal \longrightarrow \Hcal^{0,n}(\eta') \}.
\end{equation}
It is desirable to know that $F(i,n,\eta)$ is large. This is accomplished by the following "factorization theorem":
\begin{theorem}[Reduction to $H^0$] \label{th-reduction-to-h0} Let $\eta \in \chargpldom$ Assume that $\gx$ is either of PEL type, or of compact type, or that $(S_{\Kcal}^{\Sigma}, \eta)$ satisfies Conditions ~\textnormal{\ref{cond-zero-dim}} and~\textnormal{\ref{cond-hdi}}. Then
\begin{enumerate}[label=(\alph*)]
\item 
\label{item-arith-progr} 
There exists an arithmetic progression $A$ such that $\eta+a\eta_{\omega} \in F(i, n, \eta)$ for all $a \in A \cap \zgeqo$. 
\item 
\label{item-cone-maximal}
Let $\Ccal \subset X^*(\TT)$ be the "global sections cone" defined in~\eqref{eq-flag-cone}. 
For all $\nu \in \Ccal$ and $\eta_1 \in F(i, n, \eta)$ there exists $m=m(\nu,n) \in \zgeqo$ such that $\eta_1+jm\nu \in F(i,n,\eta)$ for all $j \in \zgeqo$. 
\item 
\label{item-cone-phi-regular}
For all $\delta \in \rgeqz$, 
$F(i,n,\eta)$  contains a $\delta$-regular character  \textnormal{(\Def~\ref{def-delta-reg})}.
\end{enumerate} 
\end{theorem}

\begin{rmk}[Concerning our assumptions] We expect it will soon be shown Conditions~\ref{cond-zero-dim} and ~\ref{cond-hdi} hold for all Shimura varieties of Hodge type and all $\eta \in X^*_{+,L}(\TT)$. These conditions are vacuous in the compact case and are known to hold in the PEL case and in general when $\eta$ is a multiple of $\eta_{\omega}$; see Lemma~\ref{lem-zero-dim-pel} and \Rmk~\ref{rmk-hdi}.
 \label{rmk-assumptions-in-main-theorem} \end{rmk}

\begin{rmk} \Th~\ref{th-reduction-to-h0}\ref{item-cone-phi-regular} illustrates merely one of many regularity conditions that follow from \ref{item-cone-maximal}. Following the proof that \ref{item-cone-maximal} imples \ref{item-cone-phi-regular} in \S\ref{sec-increased-regularity}, it should be clear to the reader that "$\delta$-regular" may be replaced by any regularity condition whose singular locus is a finite union of hyperplanes. In particular, when $\GG=GSp(2g)$, this applies to the condition "spin-regular" studied in \cite{Kret-Shin-spin-valued-Galois-reps}.
\end{rmk}

\begin{rmk}
\label{rmk-boxer-pilloni-stroh}
For PEL Shimura varieties of type A and C (resp. Scholze's integral models of general Hodge-type Shimura varieties) Part \ref{item-arith-progr} of \Th~\ref{th-reduction-to-h0} was also obtained simultaneously and independently in \cite{Boxer-thesis} (resp. \cite{Pilloni-Stroh-CoherentCohomology}). 
The assumptions  
\ref{cond-zero-dim}, \ref{cond-hdi} are not needed in \cite{Pilloni-Stroh-CoherentCohomology}. 
By contrast, as far as we can tell, parts \ref{item-cone-maximal}-\ref{item-cone-phi-regular} do not follow by the  methods of \cite{Boxer-thesis} and \cite{Pilloni-Stroh-CoherentCohomology}.
\end{rmk}

\begin{rmk}[Reduction to $H^0$, modulo $p$] 
\label{rmk-reduction-to-h0-mod-p}
In the case $n=1$ (action of the Hecke algebra on the cohomology of the special fiber), a standard argument gives a more elementary statement in terms of systems of Hecke eigenvalues. The set $F(i,1,\eta)$ consists of all $\eta'$ such that every system of Hecke eigenvalues that appears in $H^i(\Shktoro, \vsubeta)$ also appears in $H^0(\Shktoro, \vsub(\eta'))$.

\end{rmk}

\subsubsection{Strategy of the proof} \label{sec-hecke-fact-strategy-proof}
The proof of \Th~\ref{th-reduction-to-h0}\ref{item-arith-progr} naturally breaks down into three steps:  
(i) "Descent" (Lemma \ref{lem-going-down}), 
(ii) "Weight Increase" (Lemma~\ref{lem-weight-increase}) and 
(iii) "Ascent" (Lemma~\ref{lem-going-up}).
"Descent" shows that $\Hcal \to \Hcal^{i,n}(\eta)$ factors through $\Hcal \to \Hcal_Z^{0,n}(\eta+b\eta_{\omega})$ for some $b \in \zgeqo$ and some Hecke-equivariant, nilpotent thickening $Z$ of a length stratum.
"Weight increase" then shows that the latter factors through $\Hcal \to \Hcal_Z^{0,n}(\eta+s\eta_{\omega})$ for all sufficiently large $s$ in an arithmetic progression containing $b$. By choosing $s$ large enough, the  vanishing lemma~\ref{lem vanishing for strata}  applies. 
Finally, "Ascent" yields the factorization claimed in \ref{item-arith-progr}. 

The flag space is introduced in \S\ref{sec-flag-space} to prove \ref{item-cone-maximal}-\ref{item-cone-phi-regular}. 
The latter are proved in \S\ref{sec-increased-regularity} by applying results of \S\ref{subsection global sec cone} about $\GF^{\mu}$ and $\Ccal$.

\subsection{Descent, Weight increase and going up} 
\label{sec-going-down-Weight-increase-going-up}
Throughout 
\S\ref{sec-going-down-Weight-increase-going-up}, adopt the assumptions of
\Th~\ref{th-reduction-to-h0}. If $H^i(\Sscr^{\Sigma,n}_{\Kcal}, \vsubeta)=0$, then \Th~\ref{th-reduction-to-h0} is vacuous for $(i,n,\eta)$. So assume for the rest of \S\ref{sec-fact-hecke} that $H^i(\Sscr^{\Sigma,n}_{\Kcal}, \vsubeta)\neq 0$.
The following linear algebra lemma will be applied repeatedly:
\begin{lemma} 
\label{lem-fact-hecke} Suppose  $R$ is a $\zp$-algebra and $\epsilon:M'\to M$ is a morphism of $R\mathcal H$-modules. 
Write $\mathcal H_M$ (resp. $\mathcal H_{M'}$) for the image of $\mathcal H$ in $\End_R(M)$ (resp. $\End_R(M')$). 
If $\epsilon$ is injective (resp. surjective), then the map $\Hcal \twoheadrightarrow \Hcal_{M'}$ (resp. $\Hcal \twoheadrightarrow \Hcal_{M}$) factors through $\Hcal_{M}$ (resp. $\Hcal_{M'}$).
\end{lemma}

\begin{proof}
Both factorizations are equivalent to the corresponding inclusions of kernels, which are trivial.
\end{proof}

\begin{lemma}[Descent] 
\label{lem-going-down}
Suppose $(i, n, \eta)$ is a triple as in \textnormal{\S\ref{sec-main-results-fact-hecke}}.
Then there exists $b \in \zgeqo$ and a Hasse-regular sequence $\Zcal=(Z_j, a_j, f_j)_{j=0}^{i}$  such that \begin{equation} \label{eq-going-down} \Hcal \twoheadrightarrow \Hcal^{i,n}(\eta) \mbox{ factors through } \Hcal \twoheadrightarrow \Hcal_{Z_i}^{0,n}(\eta+b\eta_{\omega}). \end{equation} \end{lemma}
\begin{proof}  We prove by induction on $J$ that for every $0 \leq J \leq i$ there exists a Hasse-regular sequence $\Zcal_J=(Z_j, a_j, f_j)_{j=0}^{J}$, an integer $b_{J}$ and a factorization \begin{equation} \label{eq-going-down-2}
\Hcal \twoheadrightarrow \Hcal^{i,n}(\eta) \mbox { factors through } \Hcal \twoheadrightarrow \Hcal_{Z_J}^{i-J,n}(\eta+b_J\eta_{\omega}). \end{equation} 
It suffices to show that a Hasse-regular sequence $\Zcal_J$ satisfying~\eqref{eq-going-down-2} can be extended by some $(Z_{j+1}, a_{j+1}, f_{j+1})$ and $b_{J+1} \geq 1$ to a Hasse regular sequence $\Zcal_{J+1}$ satisfying 
\begin{equation} 
\label{eq-going-down-3}
\Hcal \twoheadrightarrow \Hcal^{i,n}(\eta) \mbox { factors through } \Hcal \twoheadrightarrow \Hcal_{Z_{J+1}}^{i-(J+1),n}(\eta+b_{J+1}\eta_{\omega}). \end{equation} 

 The factorization~\eqref{eq-going-down-2} implies that $H^{i-J}(Z_J, \vsubeta \otimes \omega^{b_J}) \neq 0$. Recall that $\dim Z_J=d-J$ (\Rmk~\ref{rmk-hasse-reg}\ref{item-hasse-reg-dim}). By Lemma~\ref{lem vanishing for strata}, there exists $r \geq 1$ such that \begin{equation} H^{i-J}(Z_J, \vsubeta \otimes \omega^{b_J+ra_J}) = 0 .\label{eq going down vanishing} \end{equation}

 By \Cor~\ref{cor-hasse-reg-injective}, $f_J \in H^0(Z_J, \omega^{a_J})$ is injective and by \Rmk~\ref{rmk-hasse-reg}\ref{item-hasse-reg-gofafp-equiv} it is also $\gofafp$-equivariant. Thus multiplication by $f_J^r$ yields a $\gofafp$-short exact sequence of sheaves on $Z_J$:

\begin{equation}
0\to \vsubeta \otimes \omega^{b_J} \to \vsubeta \otimes \omega ^{b_J+ra_J} \to (\vsubeta \otimes \omega ^{b_J+ra_J})|_{Z(f_J^r)}\to 0
\label{eq going down short exact seq}.\end{equation}

Set $Z_{J+1}=Z(f_J^r)$ and $b_{J+1}=b_J+ra_J$. 
Applying~\eqref{eq going down vanishing}, and Lemma~\ref{lem-fact-hecke} to the associated long exact sequence in cohomology yields~\eqref{eq-going-down-3}. 
By \Cor~\ref{cor-length-shimura}, $(Z_{J+1})_{\red}=S_{d-(j+1)}^*$.
Applying \Th~\ref{th-glue} to the length Hasse invariant $h_{d-(J+1)}$ gives $a_{J+1} \geq 1$ and $f_{J+1} \in H^0(Z_{J+1}, \omega^{a_{J+1}})$ satisfying the conditions of that theorem. By construction, $\Zcal_{J+1}=(Z_j,a_j,f_j)_{j=0}^{J+1}$ is a Hasse regular sequence extending $\Zcal_J$. This completes the induction. 
\end{proof}

For the rest of 
\S\ref{sec-going-down-Weight-increase-going-up}, fix the notation and the Hasse-regular sequence of Lemma~\ref{lem-going-down}.

\begin{lemma}[Weight increase]
\label{lem-weight-increase}
 For all sufficiently large $s \in b+a_i\ZZ$, one has: 
\begin{equation} 
\label{eq Weight increase fact}
\Hcal \to \mathcal H_{Z_i}^{0,n}(\eta+b\eta_{\omega}) \mbox{ factors through } 
\Hcal \to \Hcal_{Z_i}^{0,n}(\eta+s\eta_{\omega})   \end{equation} and for all $0\leq j \leq i-1$,  
\begin{equation} H^1(Z_j, \vsubeta \otimes \omega^{s-a_j})=0. \label{eq Weight increase vanishing}\end{equation}\label{lem Weight increase} \end{lemma}
\begin{proof} Consider the injective, $\gofafp$-equivariant section $f_i \in H^0(Z_i, \omega^{a_i})$ pertaining to $(Z_j, a_j, f_j)_{j=0}^i$. 
For all $r \geq 1$, multiplication by $f_i^r$ induces a Hecke-equivariant injection \begin{equation} H^0(Z_i, \vsub(\eta+b\eta_{\omega})) \hookrightarrow H^0(Z_i, \vsub(\eta+(b+ra_i)\eta_{\omega})). \label{eq Weight increase injection} \end{equation} 
Lemma~\ref{lem-fact-hecke} gives the factorization~\eqref{eq Weight increase fact}. It satisfies the vanishing~\eqref{eq Weight increase vanishing} for  $r \gg 0$  by Lemma~\ref{lem vanishing for strata}.
\end{proof}

\begin{lemma}[Ascent]
\label{lem-going-up}
Suppose $s$ satisfies~\eqref{eq Weight increase fact} and~\eqref{eq Weight increase vanishing}. Then \begin{equation}
\label{eq going up}
\Hcal \to \Hcal^{0,n}_{Z_i}(\eta+s\eta_{\omega}) \mbox{  factors through } \Hcal \to \Hcal^{0,n}(\eta+s\eta_{\omega}).
\end{equation}  
 \end{lemma}
\begin{proof} We prove by downward induction on $J$ that  
\begin{equation} 
\label{eq-going-up-2} 
\Hcal \to  \Hcal_{Z_i}^{0,n}(\eta+s\eta_{\omega}) \mbox{ factors through } \Hcal \to \Hcal_{Z_J}^{0,n}(\eta+s\eta_{\omega})    
\end{equation} for all $0 \leq J \leq i$. For $J=i$ there is nothing to prove. So we assume~\eqref{eq-going-up-2} and prove that also \begin{equation} 
\label{eq-going-up-3} 
\Hcal \to  \Hcal_{Z_J}^{0,n}(\eta+s\eta_{\omega}) \mbox{ factors through } \Hcal \to \Hcal_{Z_{J-1}}^{0,n}(\eta+s\eta_{\omega}).    
\end{equation}
Combining the factorizations~\eqref{eq-going-up-2} and~\eqref{eq-going-up-3} will then complete the induction.

Consider the short exact sequence of sheaves on $Z_{J-1}$ given by multiplication by $f_{J-1}$:

\begin{equation}
0\to \vsub(\eta+(s-a_{J-1})\eta_{\omega}) \to \vsub(\eta+ s\eta_{\omega}) \to \vsub(\eta+s\eta_{\omega})|_{Z_J}\to 0
\label{eq going up short exact seq}.\end{equation}
The associated long exact sequence gives

\begin{equation}
 H^0(Z_{J-1},\vsub(\eta+s\eta_{\omega})) \to H^0(Z_J, \vsub(\eta+s\eta_{\omega}))\to H^1(Z_{J-1}, \vsub(\eta+(s-a_{J-1})\eta_{\omega}))
\label{eq going up long exact seq}.\end{equation}

\noindent By~\eqref{eq Weight increase vanishing}, the right-most term in~\eqref{eq going up long exact seq} is zero (given our choice of $s$). 
Hence Lemma~\ref{lem-fact-hecke} gives~\eqref{eq-going-up-3}.
\end{proof}

\begin{proof}[Proof of \textnormal{\Th~\ref{th-reduction-to-h0}\ref{item-arith-progr}:}] Apply Lemmas~\ref{lem-going-down}, ~\ref{lem Weight increase} and ~\ref{lem-going-up} successively. \end{proof}
\section{Increased regularity via the flag space} 
\label{sec-regularity-flag-space}

\subsection{The flag space} \label{sec-flag-space}
The quotient of the $\Pcal$-torsor $I_{\Pcal}$ on $\Sscr_{\Kcal}$ by the Borel subgroup $\Bcal$ is represented by a smooth, projective $\Sscr_{\Kcal}$-scheme denoted $\Flk$, which we call the flag space of $\Shk$. Similarly, the flag space of $\Shktor$ is $\Flktor:=I_{\Pcal}^{\Sigma}/\Bcal$. Let $Fl_{\Kcal}$ (resp. $\Flktoro$) denote the special fiber of $\Flk$ (resp. $\Flktor$). 
There are natural isomorphisms
\begin{equation}
\label{eq-flag-space-fiber-product}
Fl_{\Kcal} \simeq \Shko \times_{\GZip^\mu}\GF^\mu
\mbox{ and }
\Flktoro \simeq \Shktoro \times_{\GZip^\mu}\GF^\mu. \end{equation}

By construction, the tower $(\Flktor)_{\Kcal^p, \Sigma}$ admits an action of $\gofafp$. The system of maps $(\Flktor \to \Shktor)$  is $\gofafp$-equivariant. 
The construction of \S\ref{Gvarstacks} assigns a $\gofafp$-equivariant line bundle  $\leta$ on $\Flk^{\Sigma}$ (resp. $\lcaneta$ on $\Flktor$)  to every $\eta \in X^*(\TT)$. The restriction of $\leta$ (resp. $\lcaneta$) to a fiber is the line bundle previously called $\leta$ in 
\S\ref{sec-torsors}. Write $\pi: \Flktor \to \Shktor$ for the natural projection.
For all $\eta \in \chargpldom$ one has canonical, $\gofafp$-equivariant identifications \begin{equation}
\label{eq-leta-veta}
\pi_*\leta=\veta
\mbox{ and } 
\pi_*\lcaneta=\vcaneta.
\end{equation}

\begin{rmk}\label{Ldom}
Note that the $k$-fibers of $\pi$ are flag varieties isomorphic to $L/B_L$. If the character $\eta$ is not in the cone $\chargpldom$, then $H^0(L/B_L,\leta)=0$, so we deduce $\pi_*\leta=\veta=0$. This implies that if $\leta$ admits a nonzero global section on $\Flktoro$, then $\eta\in \chargpldom$. In particular, one has $\Ccal \subset \chargpldom$ (\S\ref{subsection global sec cone}).
\end{rmk}

Recall (\S\ref{sec-toroidal-bundles}) that $D$ denotes the boundary divisor of the toroidal compactification $\Sscr_{\Kcal}^{\Sigma}$.  Put $\lsubeta:=\leta \otimes \pi^*\mathcal O(-D)$. Since $\mathcal O(-D)$ is a line bundle, the projection formula and~\eqref{eq-leta-veta} gives $\pi_*\lsubeta=\vsubeta$.

\begin{rmk}[Motivation for the flag space]
\label{rmk-GS}
The $\Ocal_{\pfr}$-scheme $\Flktor$ is a simultaneous generalization of aspects of work of Griffiths-Schmid over $\CC$ \cite{Griffiths-Schmid-homogeneous-complex-manifolds} (see also the works Carayol cited in the introduction) and work of Ekedahl-van der Geer \cite{Ekedahl-Geer-EO} in characteristic $p>0$. 
We were inspired by both of these works. 

Griffiths-Schmid studied homogenous complex manifolds of the form $\Gamma\backslash G/T$, where $G$ is a connected, semisimple, real Lie group, $T$ is a compact Cartan subgroup, $G/T$ is endowed with a complex structure and $\Gamma$ is an arithmetic subgroup. 
Carayol termed these Griffiths-Schmid manifolds \cite{Carayol-LDS-1}. 

Griffiths and his school refer to those complex structures on $G/T$ which fiber holomorphically over a Hermitian symmetric domain as the {\em classical case}. In the classical case, the Griffiths-Schmid manifolds $\Gamma \backslash G /T$ are algebraic. By contrast, it has recently been shown that all Griffiths-Schmid manifolds in the non-classical case are not algebraic \cite{Griffiths-Robles-Toledo-not-algebraic}.

The complex manifolds $\Flktor(\CC)$ are adelic versions of algebraic Griffiths-Schmid manifolds. They provide a moduli interpretation of every algebraic Griffiths-Schmid manifold which lies over a Shimura variety of Hodge-type. 

For Siegel modular varieties, Ekedahl-van der Geer defined a flag space $Fl_g$ over $\fp$ by more elementary means than ours, namely in terms of full symplectic flags refining the Hodge filtration in $H^1_{\dR}$ of an abelian scheme. 
They went on to define a stratification of $Fl_g$ and study its relation to the EO stratification of $\Shgk \otimes \fp$. 
The stratification obtained above via~\eqref{eq-flag-space-fiber-product} agrees with that of \cite{Ekedahl-Geer-EO} for Siegel varieties. 
\label{rmk griffiths schmid ekedahl geer}\end{rmk}

\begin{rmk}[Comparison of cohomology] Suppose $\eta \in \chargpldom$, $n \in \zgeqz \cup \{+\}$ and $\pi: \Flktorn \to \Shktorn$ is the natural projection. 
An application of Kempf's vanishing theorem \cite[II, \Chap 4]{jantzen-representations} and Lemma~\ref{vanishn} yields $R^i \pi_* \mathcal \lsubeta=0$ for all $i>0$. 
Consequently, for all $i \in \zgeqz \cup\{+\}$ there is a Hecke-equivariant isomorphism \begin{equation} H^i(\Flktorn, \lsubeta)\cong H^i(\Shktorn, \vsubeta). \label{eq-kempf} \end{equation} \label{rmk-kempf}
In this paper~\eqref{eq-kempf} is only used with $i=0$, when it follows directly from~\eqref{eq-leta-veta}. 
However,~\eqref{eq-kempf} illustrates another advantageous property of the flag space, which is useful in other situations \cf the forthcoming work \cite{Brunebarbe-Goldring-Koskivirta-Stroh-ampleness}.\end{rmk}

\subsection{Increased Regularity}
\label{sec-increased-regularity}
The flag space affords additional Hecke factorizations as follows:

\begin{proof}[Proof of \textnormal{\Th~\ref{th-reduction-to-h0}\ref{item-cone-maximal}:}]   By \S\ref{subsection global sec cone}, for all $\nu \in \Ccal$ there exists a $\gofafp$-equivariant, injective section $h_{\nu} \in H^0(\Flktoro, \lcan(\nu))$. 
By \Th~\ref{th-glue}, there exists $m=m(\nu,n) \geq 1$ such that $h_{\nu}^m$ lifts to an injective, $\gofafp$-equivariant section $\tilde h_{\nu}^m \in H^0(\Flktorn, \lcan(m\nu))$. 
For all $j \geq 1$, multiplication by $\tilde h_{\nu}^{jm}$ induces a Hecke-equivariant injection 
$$H^0(\Flktorn, \lsub(\eta_1)) \hookrightarrow H^0(\Flktorn, \lsub(\eta_1+jm\nu)).$$ 
If $\eta_1 \in F(i,n,\eta)$, then it follows from Lemma~\ref{lem-fact-hecke} and ~\eqref{eq-leta-veta} that $\eta_1+jm\nu \in F(i,n,\eta)$.
\end{proof}

\begin{proof}[Proof of \textnormal{\Th~\ref{th-reduction-to-h0}\ref{item-cone-phi-regular}:}]
By \S\ref{subsection global sec cone}, the cone $\Ccal$ spans $X^*(\TT)_{\QQ}$. On the other hand, the singular locus in $X^*(\TT)_{\QQ}$ is a finite union of hyperplanes \viz the root hyperplanes $\alpha^{\perp}$ for $\alpha \in \Phi(\GG, \TT)$ (\S\ref{sec-weyl-chamber}). Therefore $\Ccal$ contains a regular element $\nu \in \chargpldom$. 

By \Th~\ref{th-reduction-to-h0}\ref{item-arith-progr}, $F(i,n,\eta)$ is nonempty. So let $\eta_1 \in F(i,n,\eta)$. By \Th~\ref{th-reduction-to-h0}\ref{item-cone-maximal}, there exists $m=m(\nu, n)$ such that $\eta_1+jm\nu \in F(i,n,\eta)$ for all $j \in \zgeqo$; since $\nu$ is regular, $\eta_1+jm\nu$ is  $\delta$-regular for all sufficiently large $j$. 
\end{proof}

\section{Galois representations}

\label{sec-galois-reps}
This section is devoted to \Th~\ref{th-intro-galois} on the association of Galois representations (resp. pseudo-representations) to automorphic representations whose archimedean component is a non-degenerate limit of discrete series (LDS) (resp. coherent cohomology modulo $p^n$). \S\S\ref{sec-arch-reps}-\ref{sec-satake} recall facts and introduce notation regarding archimedean and non-archimedean representations respectively. To simplify the ensuing statements, we formalize the existence of a Galois representation associated to an automorphic representation as a condition in \S\ref{sec-regular-galois-reps}. Our results are then stated as four theorems in \S\S\ref{sec-torsion-th}-\ref{sec-nondeg-lds}. These theorems are proved in \S\ref{sec-pf-torsion-lds}.
We hope that by carefully stating our hypotheses in a general context, it will be easy to apply them as soon as new results are established for cohomological representations (\eg see \Rmk~\ref{rmk-kret-shin}). 

Recall our notation and conventions regarding structure theory of reductive groups, (based) root data and related objects (\S\ref{sec-structure-theory-root-data}) and Shimura varieties of Hodge type \S\ref{sec-shimura-varieties-hodge-type}. Fix a Shimura datum of Hodge type $\gx$ together with an integral symplectic embedding $\varphi$~\eqref{eq-integral-symplectic-embedding}. 
Recall from \S\ref{sec-shimura-root-data-mu} that associated to $(\GG, \XX, \varphi)$ we have: The conjugacy class of cocharacters $[\mu]_{\overline{E}_\pfr}$ over ${\overline{E}_\pfr}$, a representative $\mu$ defined over $E_{\pfr}$, a Borel pair $(\BB, \TT)$ of $\GG_{\QQ_p}$, opposite parabolic subgroups $\PP, \PP^+$ such that $\BB_{E_p}\subset \PP$ and the Levi $\LL=\cent(\mu)=\PP \cap \PP^+$ over $E_{\pfr}$ containing $\TT_{E_{\pfr}}$.

Throughout \S\ref{sec-galois-reps}, fix an isomorphism (of abstract fields) $\iota:\qpbar \stackrel{\sim}{\longrightarrow} \mathbf C$. We choose $\iota$ compatibly with $\PP, \PP^-$ and $\Delta$ as explained in \S\ref{sec-shimura-compatible-complex}, so that $\iota \PP_{\qpbar}$ is the stabilizer of the Hodge filtration for the $\RR$-Hodge structure $\Ad \circ h$, where $h \in \XX$ gives rise to $\iota \mu_{\qpbar}$ via \S\ref{sec-notn-cocharacter}.
\subsection{Archimedean representation theory} \label{sec-arch-reps}
We briefly review notation and results about real groups and Lie algebra cohomology needed later in \S\ref{sec-galois-reps}.  
\subsubsection{Infinitesimal Character} \label{harish-chandra iso}
 Following Buzzard-Gee \cite{Buzzard-Gee-conjectures}, we say that $\chi \in X^*(\TT)_{\CC}$ is $L$-algebraic (resp. $C$-algebraic) if $\chi \in \chargpbft$ (resp. $\chi \in \chargpbft+\rho$).

Let $\gfr$ (resp. $\tfr$) be the complexified Lie algebra of $\GG$ (resp. $\TT$) and let $\zfr$ be the center of the universal enveloping algebra of $\gfr$. Let $\Sym(\tfr)$ denote the symmetric algebra of $\tfr$.
We normalize the Harish-Chandra isomorphism
\begin{equation}\label{eq harish-chandra iso} 
\zfr \xrightarrow{\sim} (\Sym(\tfr))^{W} 
\end{equation}
in the usual way, meaning that the infinitesimal character of the trivial representation is identified with $\rho$.

Let $K$ be a maximal compact subgroup of $\gofr$. Let $\pi_{\infty}$ be an irreducible Harish-Chandra module \ie an irreducible $(\gfr, K_{\CC})$-module. Write $\chi_{\infty}$ for the infinitesimal character of $\pi_{\infty}$, identified with an element of $X^*(\TT)_{\CC}/W$ via~\eqref{eq harish-chandra iso}.

Given a property $\Pcal$ of elements of $X^*(\TT)_{\CC}$ which is stable under the action of $W$, we say that $\pi_{\infty}$ has property $\mathcal P$ if $\chi_{\infty}$ does. We say that an automorphic representation has property $\Pcal$ if its archimedean component does.

\subsubsection{Limits of discrete series} \label{sec-lds} The parametrization of limits of discrete series (LDS) for $\gofr$ is complicated by the fact that $\gofr$ is often neither semisimple nor connected in the classical topology. For a semisimple real Lie group which is connected in the classical topology, LDS were parametrized in \cite[\S1]{Knapp-Zuckerman-Annals}. In the generality considered here, a full parametrization of LDS was given by the first author in both \cite{Goldring-Thesis} and \cite{Goldring-LDS-nondeg-deg-functoriality}, based on explanations by Schmid and Vogan.
Here we restrict attention to $C$-algebraic LDS, since they are the only LDS considered in this paper.

A $C$-algebraic LDS Harish-Chandra parameter is a pair $(\lambda, \Ccal)$, where $\Ccal$ is a Weyl chamber for the root datum $\Rcal\Dcal(\GG,\TT)$ (see \S\S\ref{sec-root-data},\ref{sec-weyl-chamber}),  $\lambda \in \partial \Ccal \cap (X^*(\TT) +\rho)$ and $\langle \lambda , \alpha^{\vee} \rangle \neq 0$ for all $\Ccal$-simple $\alpha \in \Phi_{\LL}$.\footnote{Our definition implicitly exploits the relationship between $K_{\CC}$ and $\LL_{\CC}$ which exists due to $\gx$ being a Shimura datum. If $G$ is a real group which does not arise as $\GG_{\RR}$ for some Shimura datum $\gx$, in particular if $G$ admits discrete series but does not admit holomorphic discrete series, then one should define Harish-Chandra parameters differently in terms of roots of $K_{\CC}$.} For every $C$-algebraic, LDS Harish-Chandra parameter $(\lambda,\Ccal)$, there exists a $C$-algebraic LDS Harish-Chandra module $\pi(\lambda, \Ccal)$. The LDS $\pi(\lambda, \Ccal)$ may be constructed as the image of a discrete series $\pi_{\rho}$ by the Zuckerman translation functor $\transl_{W\lambda}$ which goes from Harish-Chandra modules of infinitesimal character $W\rho$ to ones of infinitesimal character $W\lambda$ (\cf the discussion preceding \Th~2.1 in \cite{Soergel-LDS}).

The fundamental dichotomy of non-degenerate versus degenerate was introduced for LDS by Knapp-Zuckerman in their classification of tempered representations \cite{Knapp-Zuckerman-Annals}.
An LDS $\pi(\lambda, \Ccal)$ is degenerate if $\lambda$ is orthogonal to $\alpha^{\vee}$ for some $\alpha \in \Phi_L$ (which is then necessarily not $\Ccal$-simple); otherwise, it is called non-degenerate.

Following \cite[\S2.3]{Goldring-Galois-reps-HLDS}, 
we say that an LDS is $\XX$-holomorphic if it is isomorphic to some $\pi(\lambda, \Ccal)$ with $\Ccal$ the $\Delta$-dominant chamber (here $\XX$ pertains to the Shimura datum $\gx$). As noted in \loccitn, every $\XX$-holomorphic LDS is non-degenerate.

\begin{rmk}[Degeneracy via Langlands Parameters]

The first author has shown  that an equivalent formulation of degeneracy is that the image of the restriction to $\ccross$ of the Langlands parameter of $\pi(\lambda, \Ccal)$ contains a simple group of rank at least two \cite{Goldring-LDS-nondeg-deg-functoriality}.
In turn, this characterization allows to generalize the non-degenerate/degenerate dichotomy to arbitrary Harish-Chandra modules, (\cite{Goldring-LDS-nondeg-deg-functoriality,Goldring-Paris-Talk-Functoriality-2015}).
\end{rmk}
\subsubsection{Lie algebra cohomology} \label{sec lie algebra cohomology}
Let $\pfr$ (resp. $\lfr$) be the complexified Lie algebra of $\PP$ (resp. $\LL$).  
For $\eta \in \chargpldom$, let $V_{\eta}$ be the irreducible finite-dimensional representation of $\LL_{\CC}$ of highest weight $\eta$.
Given a Harish-Chandra module $\pi_{\infty}$ for $\gofr$, write $H^i(\pfr, \lfr, \pi_{\infty} \otimes V_{\eta})$ for the relative Lie algebra cohomology of the pair ($\pfr$, $\lfr$) with coefficients in $\pi_{\infty} \otimes V_{\eta}$ (\cf \cite{Vogan-Zuckerman,Guichardet-book-Lie-algebras,Borel-Wallach-book}, though note that the latter reference's notation is somewhat different).

Define the \emph{cohomology degree} $\cd(\Ccal)$ of a Weyl chamber $\Ccal$ by \begin{equation} \cd(\Ccal)=\card (\{\alpha \in \Phi^{+}|\langle  \Ccal , \alpha^{\vee} \rangle<0 \})
\label{eq cohomology degree}. \end{equation}   

We shall need the following result about the $(\pfr,\lfr)$-cohomology of non-degenerate LDS.

\begin{theorem}[Schmid-Williams-Harris] 
\label{th-lie-algebra-cohomology}
Let $\pi(\lambda, \Ccal)$ be a non-degenerate LDS, normalized by requiring that $\Ccal$ is $\Delta_{\LL}$-dominant.
Then 
\begin{equation}
\label{eq lie algebra cohomology}
\dim H^{\cd(\Ccal)}(\pfr, \lfr, \pi(\lambda, \Ccal) \otimes V_{-w_{0,\LL}\lambda-\rho})=1.
\end{equation}
\end{theorem}
\begin{rmk} Schmid computed the $\nfr$-cohomology of discrete series \cite{Schmid-L2-cohomology-discrete-series}. 
The computation was generalized to non-degenerate LDS by Williams \cite{Williams-nondeg-LDS}. Harris translated their results to the setting of $(\pfr, \lfr)$-cohomology \cite[\Th~3.4]{Harris-Ann-Arbor}.   
In ~\Th~3.4 of \loccitn, Harris makes the additional claim that the $(\pfr, \lfr)$-cohomology of $\pi(\lambda, \Ccal)$ is zero in all degrees other than $\cd(\Ccal)$. This is false already for $\GG=GL(2)$. It is true when $\gofr$ is semisimple and connected in the classical topology. We shall only use the part of \loccit stated in \Th~\ref{th-lie-algebra-cohomology}, which is correct.  
\end{rmk}

Following \cite[\S\S2-3]{Harris-Ann-Arbor}, we recall in the next corollary how \Th~\ref{th-lie-algebra-cohomology} leads to embedding the finite part $\pi_f$ of a cuspidal automorphic representation $\pi$ of $\GG$ with non-degenerate LDS archimedean component $\pi_{\infty}$ in the coherent cohomology of the Shimura variety $\shgx$. To this end, it seems unfortunately necessary to use the following analogue for coherent cohomology  of "interior cohomology" for local systems: Given $\lambda \in X^*_{+,\LL}(\TT)$, let 
\begin{equation}
\label{eq-interior-cohom}
\bar{H}^{i}(\Shktor, \Vscr(\lambda)):=\imag [H^{i}(\Shktor, \vsub(\lambda))
 \to
 H^{i}(\Shktor, \Vscr^{\can}(\lambda))],      
    \end{equation} 
the image of the map induced by the short exact sequence $$0 \to \Vscr^{\sub}(\lambda)=\Vscr^{\can}(\lambda)(-D) \to \Vscr^{\can}(\lambda) \to \Vscr^{\can}(\lambda)|_D \to 0 $$
associated to the boundary divisor $D$. Write $\Sscr^{\Sigma}_{\Kcal, \CC}:=\Sscr^{\Sigma}_{\Kcal} \otimes_{\Ocal_E,\pfr}\CC$. By \cite[\Prop~2.2]{Harris-Ann-Arbor}, the (double) direct limits $$\varinjlim_{\Kcal, \Sigma} H^{i}(\Shktorc, \vsub(\lambda)) \hspace{.1pc}
\mbox{ and }
\hspace{.1pc}
\varinjlim_{\Kcal, \Sigma} H^{i}(\Shktorc, \vcan(\lambda))$$ 
are both admissible $\gofaf$-modules and the natural map from the first to the second is $\gofaf$-equivariant. Thus  $\varinjlim_{\Kcal, \Sigma} \bar{H}^{i}(\Shktorc, \vlambda)$ is also an admissible $\gofaf$-module.

\begin{corollary}
\label{cor-gofaf-emebedding}
Suppose $\pi=\pi_f \otimes \pi(\lambda, \Ccal)$ is a cuspidal automorphic representation of $\mathbf G$, with $\Ccal$ normalized to be $\Delta_{\LL}$-dominant. Then there is a $\gofaf$-equivariant embedding
\begin{equation} 
\label{eq-gofaf-embedding}
\pi_f 
\hookrightarrow 
\varinjlim_{\Kcal, \Sigma} 
\bar H^{\cd(\Ccal)}(\Shktorc, \vsub(-w_{0,\LL}\lambda-\rho)). 
\end{equation}   
\end{corollary}
\begin{rmk}
\label{rmk-lds-cohom-shimura}
Though it will not be used in this paper, note that \Cor~\ref{cor-gofaf-emebedding} is purely a result over $\CC$ and is valid for any Shimura datum, even if not of Hodge type and not known to admit integral models.
\end{rmk}

\begin{proof}[Proof of \textnormal{\Cor~\ref{cor-gofaf-emebedding}:}] This is one of the key upshots of  \cite[\S\S2-3]{Harris-Ann-Arbor} but not explicitly stated there in the above form, so we review the key steps: For every dominant $\lambda$, \S2.6 of \loccit defines a space of harmonic cusp forms $\Hcal^*_{\cusp, \lambda}$, which is a $\gofaf$-module (not to be confused with our notation for Hecke algebras). By \Th~2.7 of \loccitn, there is a $\gofaf$-equivariant injection  
\begin{equation}
\label{eq-harmonic-cusp-into-shimura}
\Hcal^*_{\cusp, \lambda} 
\hookrightarrow 
\varinjlim_{\Kcal, \Sigma} 
\bar{H}^{*}(\Shktorc, \vlambda).
\end{equation}
 On the other hand, let $\Acal_{\cusp}(\GG)$ be the space of cuspidal automorphic forms of $\GG$. Then  (3.0.1) of \loccit  states that \begin{equation}
\label{eq-harmonic-cusp-lie-alg-cohom}
 \Hcal^*_{\cusp, \lambda} 
 \cong 
H^*(\pfr, \lfr, \Acal_{\cusp}(\GG) \otimes V_{\lambda})
\end{equation} By a theorem of Gelfand and Piatetski-Shapiro, the space $\Acal_{\cusp}(\GG)$ is a semisimple, admissible $\gofaf \times (\gfr, \LL_{\CC})$-module. Since $(\pfr, \lfr)$ acts trivially on the finite parts of automorphic representations, decomposing inside the right-hand side of ~\eqref{eq-harmonic-cusp-lie-alg-cohom} yields a sum over all cuspidal automorphic representations $\pi:=\pi_f \otimes \pi_{\infty}$: 
\begin{equation}
\label{eq-space-auto-forms}
    H^*(\pfr, \lfr, \Acal_{\cusp}(\GG) \otimes V_{\lambda})
    = 
    \bigoplus_{\pi}m_{\cusp}( \pi) \pi_f \otimes H^*(\pfr, \lfr, \pi_{\infty} \otimes V_{\lambda}), 
\end{equation} 
where $m_{\cusp}(\pi)$ is the finite multiplicity with which $\pi$ occurs in $\Acal_{\cusp}(\GG)$. Putting together~\eqref{eq-harmonic-cusp-into-shimura}, ~\eqref{eq-harmonic-cusp-lie-alg-cohom},~\eqref{eq-space-auto-forms} and \Th~\ref{th-lie-algebra-cohomology} yields the desired embedding~\eqref{eq-gofaf-embedding}.
\end{proof}

\subsection{Applications of the Satake isomorphism} 
\label{sec-satake}  Let $v$ be a non-archimedean place of $\QQ$. Assume $v \not \in \Ram(\GG) \cup \{p\}$ and let $\frobv$ be a geometric Frobenius at $v$. Recall $\iota:\qpbar \stackrel{\sim}{\to} \CC$ and write $\Hcal_{v,\CC}:=\Hcal_v \otimes_{\zp, \iota} \CC$.

Let $\lbfgz$ be the dual group of $\GG$ over $\CC$, \ie the connected, reductive $\CC$-group whose root datum is dual to $\Rcal \Dcal(\GG, \TT)$ (\S\ref{sec-root-data}). Let $\lbfg=\lbfgz \rtimes \galq$ (resp. $\lbfgv=\lbfgvz \rtimes \galqv$) denote the Galois form of the $L$-group of $\GG$ (resp. $\GG_v$), \cf \cite[\S2]{Borel-Automorphic-L-functions}. Similar to \cite[\S2.1]{Buzzard-Gee-conjectures}, we view $\lbfg$ and $\lbfgv$ as group schemes over $\CC$ with component groups $\galq$ and $\galqv$, respectively\footnote{The group denoted $\lbfg$ in \cite{Borel-Automorphic-L-functions} is the topological group we denote $\lbfgc$.}. A representation of $\lbfg$  will always mean a morphism of $\CC$-group schemes $r:\lbfg \to GL(m)$ for some $m\geq 1$, which factors through $\lbfgz \rtimes \gal(F/\QQ)$ for some finite, Galois extension $F$ over which $\GG$ splits; representations of $\lbfgv$ are defined analogously (\cf \cite[2.6]{Borel-Automorphic-L-functions}). Given a representation $r:\lbfg \to GL(m)$ of $\lbfg$, we denote by $r_v:\lbfgv \to GL(m)$ the representation of $\lbfgv$ obtained by restriction.

Let $\rfdlbfgv$ denote the subalgebra of the group algebra $\zplbfgv$ generated by the characters  representations of $\lbfgv$ (as just defined). 
Let $\rfdsslbfgv$ be the algebra obtained by considering elements of $\rfdlbfgv$ as functions on the set of $\lbfgvz(\CC)$-conjugacy classes in the coset $\lbfgvz \rtimes \frobv$ and then restricting to the subset of semisimple $\lbfgvz(\CC)$-conjugacy classes. 
As explained in  \cite[\S\S6-7]{Borel-Automorphic-L-functions}, composing the Satake isomorphism with the isomorphism of \Prop~6.7 of \loccit gives a canonical identification\footnote{The isomorphism from \loccit is with $\CC$-coefficients, \ie $\Hcal_{v,\CC} \stackrel{\sim}{\to} \rfdsslbfgv_{\CC}$. However the integral version~\eqref{eq-satake} holds because the Satake transform is defined over $\zp[\sqrt{v}]$, \cf \cite[Proof of Lemma~V.1.6]{Scholze-torsion}.}
\begin{equation}
\label{eq-satake}
\Hcal_v[\sqrt{v}] \xrightarrow{\sim}  \rfdsslbfgv[\sqrt{v}].
\end{equation}

We state two immediate, well-known consequences of~\eqref{eq-satake} which will play a crucial role in characterizing the Galois (pseudo-)representations we shall construct.

First, there is a canonical bijection between the set of (complex) characters of $\Hcal_{v,\CC}$ and the set of $\lbfgvzc$-conjugacy classes of semisimple elements in $\lbfgvz \rtimes \frobv$. 
Since the former set is in bijection with the set of unramified, admissible, complex representations of $\GG(\qv)$, so is the latter. 
If $\pi_v$ is an unramified, admissible, complex representation of $\GG(\qv)$, we denote by $\class(\pi_v)$ the corresponding $\lbfgvzc$-conjugacy class. 
Recall that the correspondence $\pi_v \leftrightarrow \class(\pi_v)$ is given as follows: 
The Hecke algebra $\Hcal_{v,\CC}$ acts on the line $\pi_v^{\Kcal_v}$ by a character $\chi(\pi_v):\Hcal_{v,\CC} \to \CC$. Then $\class(\pi_v)$ is characterized as the unique $\lbfgvzc$-conjugacy class such that the character of $\rfdsslbfgv$ corresponding to $\chi(\pi_v)$ via~\eqref{eq-satake} is given by evaluation at $\class(\pi_v)$.

In addition, for $r:\lbfg \to GL(m)$ as above, let $\class(\pi_v, r_v)$ be the conjugacy class in $\glmc$ generated by the image $r_v(\class(\pi_v))$. Let $\class_{p,\iota}(\pi_v, r_v)$ denote the conjugacy class in $\glmqpbar$ obtained from $\class(\pi_v, r_v)$ via $\iota$.

For all $j \geq 1$, the function $\tr^j(r):\lbfgv(\CC) \to \CC$ defined by $\tilde g\mapsto \tr(r(\tilde g)^j)$ lies in $\rfdsslbfgv$. 
\begin{definition} 
\label{def-hecke-operators} 
Denote by $T^{(j)}_v(r)\in \Hcal_v[\sqrt{v}]$ the Hecke operator associated to $\tr^j(r)$ by \eqref{eq-satake}. 
Write $T^{(j)}_v(r;i,n,\eta)$ for the image of $T^{(j)}_v(r)$ in the coherent cohomology Hecke algebra $\Hcal^{i,n}(\eta)$ \textnormal{(\S\ref{sec-hecke-notation})}. \end{definition}
\begin{rmk} By definition, if $\Hcal \to \Hcal^{i,n}(\eta)$ factors through $\Hcal \to \Hcal^{i',n'}(\eta')$, then $T_v^{(j)}(r;i',n',\eta')$ maps to $T_v^{(j)}(r;i,n,\eta)$ via the induced map $\Hcal^{i',n'}(\eta') \to \Hcal^{i,n}(\eta)$.
\label{rmk-hecke-op-fact}
\end{rmk} 

\subsection{Existence of automorphic Galois representations} 
\label{sec-regular-galois-reps}
Let $\pi$ be an automorphic representation of $\mathbf G$.  Fix a prime $p \not \in \Ram(\pi)$  (\S\ref{sec-def-ramification}). Let $r:\lbfg \longrightarrow GL(m)$ be a representation as defined in \S\ref{sec-satake}. 

It will be convenient to introduce a condition which says that the pair $(\pi, r)$ admits a $p$-adic Galois representation with weak local-global compatibility. If $X$ is a conjugacy class in $\glmqpbar$, then write $X^{\sesi}$ for its semi-simplifcation \ie $X^{\sesi}$ is the unique semisimple conjugacy class with the same characteristic polynomial as $X$. 
\begin{condition}[GalRep-$p$] The pair $(\pi,r)$ satisifies \textnormal{(GalRep-$p$)} if there exists a (necessarily unique) continuous, semisimple Galois representation \begin{equation} R_{p,\iota}(\pi, r):\galq \longrightarrow \glmqpbar \label{eq gal rep pi r}\end{equation} such that,
for every $v \not \in \Ram(\pi) \cup \{p\}$, one has $R_{p,\iota}(\pi, r)(\frobv)^{\sesi}=\class_{p,\iota}(\pi_v, r_v)$ as $\glmqpbar$-conjugacy classes. 
\label{cond-galrep-p} 
\end{condition}
\begin{rmk}
\label{rmk-buzard-gee-galrep-p}
Our formulation of the Langlands correspondence in Condition~\ref{cond-galrep-p} is compatible with the $L$-algebraic conjecture of Buzzard-Gee \cite[3.2.1]{Buzzard-Gee-conjectures} as follows (see also \cite[\Rmk~V.1.5]{Scholze-torsion}):
\begin{enumerate}[label=(\alph*)]
\item Suppose $\pi$ is $L$-algebraic. Assume that, as predicted by Langlands and Buzzard-Gee (and mentioned in \S\ref{sec-intro-galois-to-autom}) there exists
$\rpipi:\galq \to \lbfg(\qpbar)$ such that for every $v \notin \rampi \cup \{p\}$, one has $\rpipi(\frobv)^{\sesi}=\class_{p,\iota}(\pi_v)$  as $\lbfgz(\qpbar)$-conjugacy classes in $\lbfgz(\qpbar) \rtimes \frob_v$. Then 
$R_{p,\iota}(\pi, r)=r \circ \rpipi$. 
\item
\label{item-rmk-functoriality}
Assume (weak) Langlands functoriality holds for $(\pi,r)$ as in \cite[\Conj~6.1.1]{Buzzard-Gee-conjectures}, \ie that there exists an automorphic representation $r_*\pi$ of $GL(m)$ which is a transfer of $\pi$ at infinity and at all unramified places. If $(\pi, r)$ satisfies Condition~\ref{cond-galrep-p}, then $R_{p,\iota}(\pi,r)=R_{p, \iota}(r_*\pi)$.
\end{enumerate}
\end{rmk}

\subsection{Galois pseudo-representations associated to coherent cohomology modulo a prime power} \label{sec-torsion-th}
 Retain the notation of \S\ref{sec-regular-galois-reps} for the rest of \S\ref{sec-galois-reps}. 

\subsubsection{The general Hodge-type case} 
\label{sec-main-results-general} Let $\gx$ be a Hodge-type Shimura datum $\gx$ and $p \notin \Ram(\GG) \cup \{2\}$. Fix an integral model $\Sscr_{\Kcal}$ over $\Ocal_{E,\pfr}$ as in \S\ref{sec-integral-symplectic-embed} and a toroidal compactification $\Sscr_{\Kcal}^{\Sigma}$ as in \S\ref{sec-toroidal-review}. We assume that $\Sigma$ (hence also $\Sscr_{\Kcal}^{\Sigma}$) is smooth, so that we have the action of the Hecke algebra $\Hcal^{i,n}(\eta)$ on $H^i(\Sscr_{\Kcal}^{\Sigma}, \vsubeta)$ as defined in \S\ref{sec-hecke-tor-sh}.  Recall the notion of a $\delta$-regular character (\Def~\ref{def-delta-reg}).

Let $\eta \in \chargpldom$. If $\gx$ is neither of PEL-type nor of compact type, assume   $S_{\Kcal}^{\Sigma}$ satisfies 
Condition~\textnormal{\ref{cond-zero-dim} and $(S_{\Kcal}^{\Sigma}, \eta)$ satisfies Condition \ref{cond-hdi}}.
\begin{theorem}[Torsion, general case]  
\label{th-torsion-general}
 Let $\delta \in \rgeqz$ and $r:\lbfg \to GL(m)$ a representation \textnormal{(\S\ref{sec-satake})}. Suppose that, for every $\delta$-regular, cuspidal, $C$-algebraic automorphic representation $\pi$ with $\pi_{\infty}$ discrete series, the pair $(\pi, r)$ satisfies \textnormal{Condition~\ref{cond-galrep-p} (GalRep-$p$)}. 
Then, associated to every triple $(i, n, \eta)$,  
with  $i \geq 0$ and $n \geq 1$,   
there is a unique, continuous pseudo-representation
\begin{equation} \rpiinetar:\galq \to \mathcal H^{i, n}(\eta), \end{equation} satisfying, for all $j \geq 1$ and all $v \not \in \Ram(\GG) \cup \{p\}$, \begin{equation} \rpiinetar(\frobv^j)=T_v^{(j)}(r; i,n, \eta) 
\label{eq-hecke-compatibility-general}. \end{equation}  
\end{theorem}

Upon consulting the proof of \Th~\ref{th-torsion-general}, it should be clear to the reader that the proof can easily be adapted to treat several variants.
\begin{rmk}[Variant I: Twisting]
\label{rmk-twisting-variant}
Under additional assumptions, one can introduce twisting between $L$-algebraic and $C$-algebraic in several ways, \eg on the "domain" $\GG$ or on the "target" $GL(m)$. 

On the domain: Assume $\GG$ admits a twisting element $\theta$ as in \cite[\Def~5.2.1]{Buzzard-Gee-conjectures};
    for simplicity assume $\GG$ is $\QQ$-split and that the restriction of $r$ to $\lbfgz$ is irreducible of highest weight $\chi \in X^*(\TT)$. Let $T_v^{(j), C}(r):=v^{\langle \chi, \rho-\theta \rangle}T_v^{(j)}(r)$ be the twists of our old Hecke operators (the Tate normalization in \cite[\S8]{Gross-satake-iso}). Mimicking the proof of \Th~\ref{th-torsion-general}, one can prove the twisted variant where we assume instead that for all $(\pi, r)$ as in the theorem, $(\pi \otimes |\cdot |^{\rho-\theta}, r)$ satisfies Condition~\ref{cond-galrep-p} and in the conclusion we replace the Hecke operators $T_v^{(j)(r)}$ with the twists $T_v^{(j),C}(r)$.
    
On the target: Assume a weak transfer $r_*\pi$ exists for all $C$-algebraic, discrete series and non-degenerate LDS (\Rmk~\ref{rmk-buzard-gee-galrep-p}\ref{item-rmk-functoriality}), and suppose that $r_*$ maps $C$-algebraic to $C$-algebraic. Then one may again adapt the proof of \Th~\ref{th-torsion-general} to show that, if for all $(\pi, r)$ as in the theorem Condition~\ref{cond-galrep-p} holds for $(r_*\pi \otimes |\cdot|^{(1-m)/2}, \id)$, then the conclusion holds where we now replace $T_v^{(j)}(r)$ with $v^{(1-m)/2}T_v^{(j)}(r)$.

In terms of twisting, the optimal generalization of \Th~\ref{th-torsion-general} may be one formulated in terms of a $z$-extension of $\GG$ which admits a twisting element, following \cite[\S5]{Buzzard-Gee-conjectures} (and what Buzzard-Gee call the "$C$-group" of $\GG$). However, it seems this may require extending some of our results on Shimura varieties to $z$-extension of $\GG$ and we have not attempted to carry this out.

\begin{rmk}[Variant II: $\galf$]
\label{rmk-galf-variant}
Another notable variant is to replace $\galq$ by $\galf$ for some number field $F$ which plays a special role for the Shimura variety $\shgx$. An example which incorporates both variants for unitary similitude groups is given by \Th~\ref{th-torsion-unitary} below. In the Hilbert case, taking $F$ to be the totally real field recovers the main result of \cite{Emerton-Reduzzi-Xiao}. \label{rmk changing the field} \end{rmk}

\end{rmk}

\begin{rmk}
\label{rmk-gsp4}
In the special case where $\GG=GSp(4)$ and $r$ is the `standard' four-dimensional representation (which coincides with the spin representation since $g=2$), both twisted variants of  \Th~\ref{th-torsion-general} apply unconditionally (a twisting element exists and $C$-algebraic is mapped to $C$-algebraic). The hypothesis holds with $\delta=0$. This can be deduced in two ways: By using the work of Arthur \cite{arthur-gsp4} to transfer to $GL(4)$ and then applying Shin's result \cite{Shin-Galois-reps-cpt-Shimura-varieties}, or by the work of Weissauer \cite{weissauer-four-dim} and Laumon \cite{Laumon-siegel-cohomology,Laumon-siegel-zeta}. (The Conditions~\ref{cond-zero-dim},~\ref{cond-hdi} hold because $\gx$ is of PEL-type.)
\end{rmk}

\subsubsection{Unitary similitude groups}
\label{sec-torsion-unitary}
Let $\gx$ be an arbitrary Shimura datum of PEL type A. Then $\GG$ is an inner form of the quasi-split unitary group associated to a quadratic extension of an imaginary CM field $F$ over its totally real subfield $F^+$.

Given $v \not \in \Ram(\GG) \cup \{p\}$, a prime $w$ of $F$ above $v$ and $j \geq 1$, let $T_w^{(j), U}$ be the Hecke operator defined between Lemma 6.1 and Lemma 6.2 of \cite{HLTT} and denoted $T_v^{(i)}$ there. Let
$T_w^{(j), U}(i,n, \eta)$ be the image of $T_w^{(j), U}$ in $\Hcal^{i,n}(\eta)$.

\begin{theorem}[Torsion, unitary case] 
\label{th-torsion-unitary}
  For every triple $(i, n, \eta)$, with $i \in \zgeqz$, $n \in \zgeqo$ and $\eta \in \chargpldom$, there exists a unique pseudo-representation \begin{equation} \rpiineta:\galf \to \mathcal H^{i,n}(\eta), \end{equation} satisfying, for all $j \geq 1$ and all $w$ above some $v \not \in \Ram(\GG) \cup \{p\}$, \begin{equation} \rpiineta(\frob_w^j)=T_w^{(j), U}(i,n, \eta) \label{eq hecke compatibility unitary}. \end{equation}
\label{th torsion unitary} \end{theorem}

\subsection{Galois representations associated to non-degenerate LDS}
\label{sec-nondeg-lds}
\subsubsection{The Hodge-type case}
Return to the setting of \S\ref{sec-main-results-general}. If $\gx$ is neither of PEL-type nor of compact type, assume   $S_{\Kcal}^{\Sigma}$ satisfies 
Condition~\textnormal{\ref{cond-zero-dim}} for all sufficiently small $\Kcal^p$.
\label{sec-nondeg-LDS-th}
\begin{theorem}[LDS, Hodge case] 
\label{th-nondeg-lds-general} 
 Suppose there exists $\delta \in \rgeqz$ and $r: \lbfg \to GL(m)$ such that, for every $\delta$-regular, cuspidal, $C$-algebraic automorphic representation $\pi'$ with $\pi'_{\infty}$ discrete series, the pair $(\pi', r)$ satisfies \textnormal{Condition~\ref{cond-galrep-p} (GalRep-$p$)}. 
 Let $\pi$ be a cuspidal automorphic representation of $\GG$ with $\pi_{\infty}=\pi(\lambda, \Ccal)$  a $C$-algebraic, non-degenerate LDS. If $p \notin \Ram(\pi)$ and $(S_{\Kcal}^{\Sigma}, -w_{0, \LL}\lambda-\rho)$ satisfies \textnormal{Condition~\ref{cond-hdi}}, then the pair $(\pi, r)$ satisfies \textnormal{(GalRep-$p$)}.
\end{theorem}

As in \Rmks~\ref{rmk-twisting-variant}-\ref{rmk-gsp4}, one has variants with twisting between $L$ and $C$-algebraic, replacing $\galq$ with $\galf$ for suitable $F$ related to $\gx$ and one deduces an unconditional result for $\GG=GSp(4)$, $r$ the `standard' four-dimensional representation.

\begin{rmk}
\label{rmk-kret-shin}
It is interesting whether \Th~\ref{th-nondeg-lds-general} may be used to extend \cite{Kret-Shin-spin-valued-Galois-reps} to produce $GSpin(2g+1)$-valued Galois representations for non-degenerate LDS.
\end{rmk}

\subsubsection{Unitary Similitude Groups} \label{sec main results unitary} Return to the setting of \S\ref{sec-torsion-unitary}: $\gx$ is of PEL-type $A$ associated to a CM extension $F/F^+$. 

We have found useful the description of local base change from unitary similitude groups to $GL(m)$ at unramified primes in \cite[\S1.3]{HLTT}. \footnote{This seems to be one of many things that has been potentially ``well-known to experts'' for a long time, but difficult to extract from the literature prior to \cite{HLTT}.}
Let $\pi$ be a cuspidal automorphic representation of $\mathbf G$. For every $v \not \in \Ram(\pi)$ and $w$ a place of $F$ above $v$, define the base change of $\pi_v$, denoted $\bc(\pi_v)$, and its $w$-part $\bc(\pi_v)_w$ as in \cite[\S1.3]{HLTT}. Write $\rec_{F_w}$ for the local Langlands correspondence, normalized as in {\rm \cite{harris-taylor-book}}. Let $W_{F_{w}}$ be the Weil group of $F_{w}$. A superscript $(-)^\sesi$ will denote semi-simplification.

\begin{theorem}[LDS, unitary case]
Suppose $\pi$ is a cuspidal, $C$-algebraic automorphic representation of $\mathbf G$, such that $\pi_{\infty}$ is a non-degenerate  LDS. Then, for all $p \not \in \Ram(\pi)$, there exists a unique continuous, semisimple Galois representation
\begin{equation}\rpipi:\galf \longrightarrow  \glmqpbar\label{eq gal rep lds} \end{equation} such that, for all primes $w$ of $F$ which lie over some prime $v \not \in \Ram(\pi)$,
 the representation $\rpipi$ is unramified at $w$ and there is an isomorphism of Weil-Deligne representations
 \begin{equation} (\rpipi | _{W_{F_{w}}})^{\sesi}\cong \iota^{-1}{\rec_{F_w}}\big( \bc(\pi_{v})_w \otimes |\cdot|_{w}^{\frac{1-m}{2}}\big). \end{equation}

\label{th-nondeg-lds-unitary} \end{theorem}
\subsubsection{Remarks about Hodge-Tate weights}
\label{sec-hodge-tate}
Let $\pi$ be as in \Th~\ref{th-nondeg-lds-general} or \Th~\ref{th-nondeg-lds-unitary}. Assume $\pi$ is $L$-algebraic. The Hodge-Tate cocharacter of $\rpipi$ is conjectured to be the infinitesimal character $\chi_{\infty}$ of $\pi_{\infty}$, when viewed as a cocharacter of $\lbfgz$, \cite[\Rmk~3.2.3]{Buzzard-Gee-conjectures}. When $\pi$ is $C$-algebraic but not $L$-algebraic, after possibly passing to a $z$-extension of $
\GG$, one must apply twisting to reduce to the $L$-algebraic case, as explained in
\loccit

By the first author's work \cite{Goldring-LDS-nondeg-deg-functoriality}, the condition "$\pi_{\infty}$ is a non-degenerate LDS" is equivalent to the following condition on the Hodge-Tate cocharacter: The adjoint group of $\cent_{\lbfgz}(\chi_{\infty})$ is a product of copies of $PGL(2)$. 

Concretely, assume that $\GG$ is a form of $GL(n)$ (resp. $SO(2n)$, $SO(2n+1)$, $Sp(2n))$ and that $r$ is the `standard' representation of $\lbfg$ of dimension $n$ (resp. $2n$, $2n$, $2n+1$). Then "$\pi_{\infty}$ non-degenerate LDS" conjecturally means that every Hodge-Tate weight of $R_{p,\iota}(\pi, r)$ has multiplicity $\leq 2$, except that for $\GG_{\CC}=Sp(2n)$ (resp. $\GG_{\CC}=SO(2n)$) the weight $0$ may have multiplicity $3$ (resp. $4$). Recall that, by the standard dictionary between Hodge-Tate weights and Hodge numbers, this means that if $R_{p,\iota}(\pi, r)$ is the $p$-adic realization of a Motive $M$, then the Hodge numbers $h^{p,q}$ of the Betti realization of $M$ are all $\leq 2$ (with the same exceptions for $Sp(2n)$ and $SO(2n)$ as before).

By contrast, in the same examples, the condition "$\pi_{\infty}$ is a holomorphic LDS" conjecturally corresponds to the stricter restriction that all Hodge-Tate weights have multiplicity 1, except possibly for one weight of multiplicity $2$ if $\GG_{\CC}=GL(n)$, the weight $0$ with multiplicity $3$ if $\GG_{\CC}=Sp(2n)$ and two weights of multiplicity $2$ if $\GG_{\CC}=SO(2n+1)$ or $\GG_{\CC}=SO(2n)$.  In the other direction, "$\pi_{\infty}$ an arbitrary LDS" (conjecturally) imposes no restriction on the multiplicities of the Hodge-Tate weights.

\subsection{Proof of \textnormal{\Th~\ref{th-torsion-general}} and \textnormal{\Th~\ref{th-nondeg-lds-general}} } 
\label{sec-pf-torsion-lds}
The results of \S\S\ref{sec-torsion-th}-\ref{sec-nondeg-lds} will be deduced from our factorization theorem (\Th~\ref{th-reduction-to-h0}) coupled with the previously known \Cor~\ref{cor-gofaf-emebedding} on which automorphic representations appear in the coherent cohomology of Shimura varieties. The proofs in the case of general groups and in that of unitary groups are almost identical. The only difference is that in the former case we assume that some very regular $\pi$ satisfy  Condition~\ref{cond-galrep-p}, while  in the latter case this is given to us by \Cor~1.3 of \cite{HLTT}. Moreover, the argument we use is analogous to the one introduced by Taylor \cite{Taylor-GSp4}, and then applied in 
\cite{Jarvis-hilbert-LDS, Goldring-Galois-reps-HLDS,Goldring-Nicole-mu-Hasse,HLTT}. 
For these reasons, we only treat the case of general groups to avoid repetition.

\begin{rmk} Note that  \cite[\Cor~1.3]{HLTT} is a concise combination of Shin's results \cite[\Th~1.2]{Shin-Galois-reps-cpt-Shimura-varieties} and 
\cite[\Th~A.1]{Shin-appendix-Goldring-HLDS}. These build on the work of many people, see also \Rmk~\ref{rmk shin labesse chenevier harris}. 

\end{rmk} 
Consider first the case of torsion. 

\begin{proof}[Proof of \textnormal{\Th~\ref{th-torsion-general}:}]
Fix $(r;i,n,\eta)$ and $\delta \in \rgeqz$ as in \Th~\ref{th-torsion-general}. Put  \[\delta'=\delta+1+\max\{|\langle \rho, \alpha^{\vee} \rangle | | \alpha \in \Phi(\GG,\TT)\}.\]
By \Th~\ref{th-reduction-to-h0}\ref{item-cone-phi-regular} there exists $\nu \in F(i,n,\eta)$ which is $\delta'$-regular. Hence $-w_{0,\LL}(\nu+\rho)$ is $\max\{\delta,1\}$-regular.

Since $\eta_{\omega} \in \Ccal$ (\S\ref{subsection global sec cone}) and $\nu \in F(i,n,\eta)$, 
\Th~\ref{th-reduction-to-h0}\ref{item-cone-maximal}  implies that there exists $k \in \zgeqo$ such that $\nu+ak \eta_{\omega} \in F(i, n,\eta)$ for all $a \in \zgeqo$. 
Since $\nu$ is $\delta'$-regular, there exists $a_0 \in \zgeqo$ such that  
$\nu+ak\eta_{\omega}$ is $\delta'$-regular for all $a \geq a_0$. Thus 
$ -w_{0,\LL}(\nu+\rho+ak\eta_{\omega})$  is $\max\{\delta,1\}$-regular for all $a \geq a_0$. By \Cor~\ref{cor surjective mod pn}, there exists $a_1 \in \zgeqo$ such that the reduction map \begin{equation}
\label{eq-pf-tor-red-surj}
H^0(\Shktor, \vsub(\nu+ak\eta_{\omega})) \to H^0(\Shktorn, \vsub(\nu+ak\eta_{\omega}))
\end{equation} is surjective for all $a \geq a_1$. 

From now on, fix $k$ as above and assume $a \geq \max\{a_0,a_1\}$. Then 
\begin{enumerate}[label=(\roman*)] 
\item $\nu+ak\eta_{\omega} \in F(i,n,\eta)$,
\item $-w_{0,\LL}(\nu+\rho+ak\eta_{\omega})$  is $\max\{\delta,1\}$-regular, and  
\item the map ~\eqref{eq-pf-tor-red-surj} is surjective. 
\end{enumerate}

Suppose $f \in H^0(\Shktor, \vsub(\nu+ak\eta_{\omega}))$ is a Hecke eigenform and let $\pi(f)$ be the cuspidal automorphic representation of $\mathbf G$ that it generates. 
By the Casselman-Osborne Theorem (\cf \cite[\Prop~ 3.1.4]{Harris-Ann-Arbor}), 
the infinitesimal character $\chi_{\infty}(f)$ of $\pi(f)$ is given by $\chi_{\infty}(f)=-w_{0,L}(\nu+ak\eta_{\omega}+\rho)$. Hence  $\chi_{\infty}(f)$ is $1$-regular by construction.  Since $\pi(f)$ (hence also $\pi(f)_{\infty}$) is unitary, it follows from a result of Salamanca-Riba \cite[\Th~1.8]{Salamanca-Riba-Aqlambda} that  the archimedean component $\pi(f)_{\infty}$ is discrete series.
Therefore the hypothesis of \Th~\ref{th-torsion-general} concerning (GalRep-$p$) applies to $f$. 

Applying this hypothesis to all eigenforms in 
$H^0(\Shktor, \vsub(\nu+ak\eta_{\omega}))$ yields a unique continuous, semisimple Galois representation 
\begin{equation} \rho_{a}:\galq \longrightarrow GL(m, \Hcal^{0,+}(\nu+ak\eta_{\omega}) \otimes \qpbar) \label{eq gal rep char zero hecke algebra}\end{equation} 
such that $(\tr \rho_{a})(\frobv^j)=T^{(j)}_v(r;0,+,\nu+ak\eta_{\omega})$  for all $j \in \zgeqo$ (\S\ref{sec-satake}). Since $\tr \rho_a \subset \Hcal^{0,+}(\nu+ak\eta_{\omega})$, there is an induced pseudo-representation
\begin{equation} \rho_{a}^{\pseudo}:\galq \longrightarrow \Hcal^{0,+}(\nu+ak\eta_{\omega}) \label{eq pseudo rep}.\end{equation} 

Since $a$ is chosen so that~\eqref{eq-pf-tor-red-surj} is surjective, Lemma~\ref{lem-fact-hecke} implies that  $\Hcal \to \Hcal^{0,n}(\nu+ak\eta_{\omega})$ factors through $\Hcal \to \Hcal^{0,+}(\nu+ak\eta_{\omega})$. Since $\nu+ak\eta_{\omega} \in F(i,n,\eta)$, $\Hcal \to \Hcal^{i,n}(\eta)$ factors through $\Hcal \to \Hcal^{0,n}(\nu+ak\eta_{\omega})$.
Composing the two factorizations shows that $\Hcal \to \Hcal^{i,n}(\eta)$  factors through $\Hcal \to \Hcal^{0,+}(\nu+ak\eta_{\omega})$, yielding a map $$\Hcal^{0,+}(\nu+ak\eta_{\omega}) \to \Hcal^{i,n}(\eta).$$ Further composing with $\rho_{a}^{\pseudo}$ gives a continuous pseudo-representation $\rpiinetar$ satisfying~\eqref{eq-hecke-compatibility-general}(see \Rmk~\ref{rmk-hecke-op-fact}). Uniqueness follows from the Tchebotarev density theorem.
\end{proof}

The case of LDS will now be deduced from that of torsion.

\begin{proof}[Proof of \textnormal{\Th~\ref{th-nondeg-lds-general}:}] 
Let $\pi$ be as in \Th~\ref{th-nondeg-lds-general}, with $\pi_{\infty}=\pi(\lambda, \Ccal)$.
Let $\eta:=-w_{0,\LL}\lambda-\rho$ and $i:=\cd(\Ccal)$.
Since $(\lambda, \Ccal)$ is non-degenerate, $\eta \in \chargpldom$. 
For every $v \not \in \Ram(\pi)$, let $\Kcal_v$ be a hyperspecial subgroup of $\gofqv$. Since $\pi$ is unramified at $p$, this includes a hyperspecial subgroup $\Kcal_p \subset \gofqp$. Let $\gamma_{\pi}' \in \bigcap_{v \not \in \Ram(\pi)} \pi_f^{\Kcal_v}$; thus $\gamma_{\pi}'$ is a Hecke eigenform for $\Hcal$ satisfying $\gamma_{\pi}' \in \pi_f^{\Kcal_p}$. 
Choose a sufficiently small open compact subgroup $\Kcal^p \subset \gofafp$ of the form $\Kcal^p=\prod_{v \neq p} \Kcal'_v$ with $\Kcal'_v \subset \gofqv$  so that \begin{enumerate}[label=(\alph*)]
    \item we may apply \S\ref{sec-toroidal-review} to the integral model $\Sscr_{\Kcal}$,
    \item for all $v \not \in \Ram(\pi) \cup \{p\}$, one has $\Kcal_v' \subset \Kcal_v$.
\end{enumerate}  
By~\Cor~\ref{cor-gofaf-emebedding}, the vector $\gamma_{\pi}'$ is a Hecke eigenclass in $\bar{H}^i(\Shktor \otimes \qpbar, \veta)$ for some smooth toroidal compactification $\Sscr_{\Kcal}^{\Sigma}$ (recall from \S\ref{sec-hecke-tor-sh} that refining $\Sigma$ does not change the cohomology). 

First we explain how to pass from the "interior cohomology"~\eqref{eq-interior-cohom} to that of the subcanonical extension.
Let $\bar\Hcal^{i,0}(\eta)$ be the image of the Hecke algebra $\Hcal$ in $\End (\bar{H}^i(\Shktor \otimes \qpbar, \veta))$ By definition, one has the surjection $H^i(\Shktor \otimes \qpbar, \vsubeta) \twoheadrightarrow \bar{H}^i(\Shktor \otimes \qpbar, \veta)$. Thus $\Hcal \to \bar\Hcal^{i,0}(\eta)$ factors through $\Hcal^{i,0}(\eta)$. Therefore there exists an eigenclass $\gamma_{\pi} \in H^i(\Shktor \otimes \qpbar, \vsubeta)$ with the same eigenvalues as $\gamma_{\pi}'$.
After replacing $\Ocal_{\pfr}$ by a finite extension, we may assume that $\gamma_{\pi}$ lies in $H^i(\Shktor, \vsubeta)$. 

Let $\theta_+: \Hcal^{i,+}(\eta) \longrightarrow \Ocal_{\pfr} $ be the eigenvalue map of $\gamma_{\pi}$ and $ \theta_n: \Hcal^{i,n}(\eta) \longrightarrow \mathcal O_{\pfr}/\pfr^n $  its reduction modulo $\pfr^n$. 
By definition of the correspondence $\pi_v \leftrightarrow \class(\pi_v)$ (recalled in \S\ref{sec-satake}), $$\theta_+(T^{(j)}_v(r;i,+,\eta))=\iota^{-1}\tr (r_v(\class(\pi_v)^j))=\tr(\class_{p,\iota}(\pi_v, r_v)^j).$$ 
By the assumption that $S_{\Kcal}^{\Sigma}$ satisfies Condition~\ref{cond-zero-dim} and that $(S_{\Kcal}^{\Sigma}, \eta)$ satisfies Condition~\ref{cond-hdi}, \Th~\ref{th-torsion-general} applies to $H^i(\Sscr_{\Kcal}^{\Sigma,n}, \vsubeta)$ and yields a continuous pseudo-representatation $\rpiinetar$ satisfying~\eqref{eq-hecke-compatibility-general}. Since $\theta_n(T_v^{(j)}(r;i,n,\eta))$ is the reduction modulo $\pfr^n$ of $\theta_+(T_v^{(j)}(r;i,+,\eta))$,
the $\theta_n \circ \rpiinetar$ form a $\pfr$-adic system of pseudo-representations satisfying $$\theta_n \circ \rpiinetar(\frobv^j) \equiv \tr (\class_{p,\iota}(\pi_v, r_v)^j) \pmod{\pfr ^n}.$$ Their limit gives a continuous pseudo-representation $\chi:\galq \to \Ocal_{\pfr} \hookrightarrow \qpbar$ satisfying $$\chi(\frobv^j)=\tr (\class_{p,\iota}(\pi_v, r_v)^j)$$
for all $j \geq 1$.
By \cite[\Th~1]{Taylor-GSp4}, the $\qpbar$-valued pseudo-representation $\chi$ is the trace of a unique, semisimple (true) representation $R_{p,\iota}(\pi, r)$ satisfying Condition~\ref{cond-galrep-p} (GalRep-$p$) (for a conjugacy class $X \subset \glnqpbar$, the sequence of traces $(\tr (X^j))_{j \geq 1}$ uniquely determines the semi-simplification $X^{\sesi}$). 
Uniqueness follows again from the Tchebotarev density theorem. \end{proof}

\begin{rmk} In view of the ``change of weight'' afforded by 
\Th~\ref{th-reduction-to-h0}\ref{item-cone-phi-regular}, in the case of unitary similitude groups we only use a weak version of \cite[\Cor~1.3]{HLTT}, where the archimedean component is $\delta$-regular. By applying \Th~\ref{th-reduction-to-h0}\ref{item-cone-phi-regular} with $\delta=1$, it suffices for us to combine Shin's earlier work \cite[\Th~1.2]{Shin-Galois-reps-cpt-Shimura-varieties} with Labesse's restricted base change \cite{Lab}. 
In particular our results use neither Shin's extended base change \cite[\Th~A.1]{Shin-appendix-Goldring-HLDS}, 
nor the work of Chenevier-Harris \cite{Chenevier-Harris-automorphic-galois}. \footnote{To be completely precise, Labesse's result has the disadvantage of being stated for unitary groups rather than unitary similitude groups, and to assume $F^+ \neq \mathbf Q$, so in that respect we do use Shin's \cite[\Th~A.1]{Shin-appendix-Goldring-HLDS}, but not concerning the regularity of the archimedean component.  }  \label{rmk shin labesse chenevier harris}\end{rmk}

\section{Systems of Hecke eigenvalues on the generalized superpecial locus} 
\label{sec-serre-letter}
This section is concerned with the generalization of Serre's letter to Tate, \Th~\ref{th-intro-serre-letter}.
\subsection{Statement of the result and a corollary} \label{sec-serre-letter-statement}
 Recall that $S_{\Kcal}$ is the special fiber of the integral model $\Sscr_{\Kcal}$ (\S\ref{sec-univ-Gzip}) at a prime $p>2$ and that the stratum $S_e$ is nonempty and zero-dimensional. "System of Hecke eigenvalues" will always refer to the Hecke algebra $\Hcal$ defined in \S\ref{sec-hecke-abstract}. By \S\S\ref{sec-hecke-open-sh}-\ref{sec-hecke-tor-sh}, for $? \in \{\can, \sub\}$ one has the $\Hcal$-modules
 $$ M^i:=\bigoplus_{\eta\in \chargpldom}H^i(\Shko, \veta), \hspace{.25cm} M^i_?:=\bigoplus_{\eta\in \chargpldom}H^i(S_{\Kcal}^{\Sigma}, \Vscr^?(\eta)), 
\hspace{.25cm} \mbox{ and } \hspace{.25cm}
M_e:=\bigoplus_{\eta \in \chargpldom}H^0(S_e, \veta). $$
\begin{theorem}
\label{th-serre-letter}
If $\gx$ is neither of compact-type, nor of PEL-type, then assume:
\begin{enumerate}
    \item[($\star$)] There exists a $\gofafp$-equivariant Cartier divisor $D'$ such that $D'_{\red}=D$ and $\omega^k(-D')$ is ample on $\Sscr_{\Kcal}^{\Sigma}$ for all $k \gg 0$. 
\end{enumerate} 

Then each of the $\Hcal$-modules $M^0, M^0_{\sub}, M^0_{\can}$ and $M_e$ admits precisely the same systems of Hecke eigenvalues. \end{theorem}
\begin{rmk} \
\label{rmk-cartier-ample}
\begin{enumerate} [label=(\alph*)]
\item If $\gx$ is of compact type, then $(\star)$ is trivially satisfied with $D=D'=0$, since then $\omega$ is ample on $\Sscr_{\Kcal}$.
\item
\label{item-Cartier-ample}
If $\gx$ is of noncompact, PEL type, then \cite[\Th~7.3.3.4]{Lan-book-thesis} (explained further in \cite[(2.1)]{Lan-Suh-lifting-cusp-forms})  produces a $\gofafp$-equivariant Cartier divisor $D'$ satisfying ($\star$) for $\Sigma$ as above which is also {\em projective} \cite[\Def~7.3.1.1]{Lan-Suh-lifting-cusp-forms}.  The construction of $D'$ should apply in the general Hodge case to the compactifications of \cite{MadapusiHodgeTor} (still associated to projective $\Sigma$). Since we were unable to find a precise reference to this effect, the existence of such $D'$ appears as hypothesis ($\star$) in \Th~\ref{th-serre-letter}.
\item 
\label{item-serre-no-koecher}    
Note that \Th~\ref{th-serre-letter} assumes neither Koecher's Principle nor Conditions~\ref{cond-zero-dim}, \ref{cond-hdi} for $\Shko$.

\end{enumerate}
\end{rmk}

An immediate consequence of \Rmk~\ref{rmk-reduction-to-h0-mod-p} and \Th~\ref{th-serre-letter} is:
\begin{corollary} 
\label{cor-serre-letter-finiteness}
If $i=0$, make the assumptions of \textnormal{\Th~\ref{th-serre-letter}}; if $i>0$, then in addition make the assumptions of \textnormal{\Th~\ref{th-reduction-to-h0}}. Then only a finite number of systems of Hecke eigenvalues appear in \begin{equation}
\label{eq-serre-letter-finiteness}
\bigoplus_{i \geq 0} M^i_{\sub} \oplus M^i_{\can} \oplus M^0.
\end{equation} 

\end{corollary}

\begin{rmk} Using Lan's Higher Koecher Principle \cite{Lan-higher-koecher} one obtains results analogous to \Cor~\ref{cor-serre-letter-finiteness} for $M^j$, $j>0$, under additional hypotheses about the codimension of the boundary of $\Shkmino$.
\end{rmk}

\subsection{Proof of \Th~\ref{th-serre-letter} and \Cor~\ref{cor-serre-letter-finiteness}} The hypotheses of \Th~\ref{th-serre-letter} are assumed throughout this section. Given $\Hcal$-modules $M, M'$, write $M \rightsquigarrow_{\Hcal} M'$ to signify that every system of Hecke eigenvalues appearing in $M$ also appears in $M'$.
\label{sec-pf-serre-letter} We shall prove the following cycle of relations: $$M^0 \rightsquigarrow_{\Hcal} M_e \rightsquigarrow_{\Hcal} M^0_{\sub} \rightsquigarrow_{\Hcal} M^0_{\can} \rightsquigarrow_{\Hcal} M^0.$$ 

The relations $M^0_{\sub} \rightsquigarrow_{\Hcal} M^0_{\can} \rightsquigarrow_{\Hcal} M^0$ are trivial (Lemma~\ref{lem-sub-can-0}). Besides that, the argument is in three steps,  analogous to simplified versions of the steps ``Descent'', ``Weight increase'' and ``Ascent'' in the proof of 
\Th~\ref{th-reduction-to-h0}: 
First we use Ghitza's argument \cite{Ghitza-Siegel-Mod-p-Algebraic} to show that $M^0 \rightsquigarrow_{\Hcal} M_e$ (Lemma~\ref{lem-serre-letter-going-down}). 
Second, Lemma~\ref{lem-serre-letter-increase-weight} uses the nowhere vanishing Hasse invariant $h_e \in H^0(S_e, \omega^{N_e})$ to exhibit $m \geq 1$ such that, for all $\eta$,  the systems appearing in $H^0(S_e, \veta)$ are the same as those appearing in $H^0(S_e, \veta \otimes \omega^{m})$. 
Finally, Lemma~\ref{lem-serre-letter-going-up} shows that if $m'$ is a sufficiently large multiple of $m$, then every system appearing in $H^0(S_e, \veta \otimes \omega^{m'})$ also appears in $H^0(S_{\Kcal}^{\Sigma}, \vsubeta \otimes \omega^{m'})$ ). The one and only step which uses the hypothesis on the existence of a Cartier divisor $D'$ as in \Th~\ref{th-serre-letter} is Lemma~\ref{lem-serre-letter-going-up}.

\begin{rmk}
\label{rmk-Ghitza-mistake}
The argument given in \cite[\Prop~24]{Ghitza-Siegel-Mod-p-Algebraic} and cited in~\cite[\Prop~5.17]{Reduzzi-PEL-Mod-p} for Lemma~\ref{lem-serre-letter-going-up} seems to contain a serious error: 
If $j:S_{\Kcal} \to \pn$ is the immersion associated to a very ample power $\omega^m$ of $\omega$, 
then it is claimed in \loccit that every $\Vscr(\lambda)$ admits a locally free extension to $\pn$. Already for the Siegel scheme $S_{g, \Kcal}$ with $g \geq 2$  this cannot be: It would contradict the well-known fact that most $\Vscr(\lambda)$  fail to admit a locally free extension to $S_{\Kcal}^{\min}$. 

\end{rmk}

\begin{lemma}
\label{lem-sub-can-0} 
One has  $H^0(S_{\Kcal}^{\Sigma}, \vsubeta) \rightsquigarrow_{\Hcal} H^0(S_{\Kcal}^{\Sigma}, \vcaneta)$ and $H^0(S_{\Kcal}^{\Sigma}, \vcaneta) \rightsquigarrow_{\Hcal} H^0(S_{\Kcal}, \veta)$.
\end{lemma}
\begin{proof}
The $\gofafp$-equivariant short exact sequence $$0 \to \vsubeta \to \vcaneta \to \vcaneta|_D \to 0$$ induces a Hecke-equivariant injection $
H^0(S_{\Kcal}^{\Sigma}, \vsubeta) \hookrightarrow H^0(S_{\Kcal}^{\Sigma} \vcaneta)$. Since $S_{\Kcal}$ is open dense in $S_{\Kcal}^{\Sigma}$, restriction gives another Hecke-equivariant injection $H^0(S_{\Kcal}^{\Sigma}, \vcaneta) \hookrightarrow H^0(S_{\Kcal}, \veta)$. Applying Lemma~\ref{lem-fact-hecke} to these injections of $\Hcal$-modules gives the result. 
\end{proof}

To show that $M^0 \rightsquigarrow_{\Hcal} M_e$, we shall need an auxiliary lemma about "good filtrations" of $L$-modules. Let $V$ be an $L$-module (not necessarily finite-dimensional). 
Recall that an ascending filtration $0=V_0 \subset V_1 \subset \cdots \subset V_{i} \subset \cdots \subset V$ of $V$ by $L$-submodules 
$V_i$ is called a good filtration 
if $\bigcup_i V_i =V$ and every graded piece $V_{i}/V_{i-1} \cong V_{\eta_i}$ for some 
$\eta_i \in X^*_{+,L}(T)$, \cite[II, 4.1.6]{jantzen-representations}. Recall that the assumption $p>2$ is still in force.
\begin{lemma}
\label{lem-good-filtration}
For every $n \geq 1$ and $\eta \in X^*_{+,L}(T)$, the $L$-module $\Sym^n((\Lie(G)/\Lie(P))^{\vee}) \otimes V_{\eta}$ admits a good filtration. 
\end{lemma}
\begin{proof} The proof amounts to combining a result of Andersen-Janzten \cite{Andersen-Jantzen-cohomology-induced-reps-good-filtration} on good filtrations for $\Sym^n \Lie(G)$  with closure properties for good filtrations given by Donkin \cite{Donkin-good-filtrations-LNM}.

Since $\gx$ is of Hodge type, $\GG$ (and hence also $G$ and $L$) are of classical type (\ie have no factors of exceptional type) \cite[1.3.10]{Deligne-Shimura-varieties}. Therefore every prime $p>2$ is good for $G$ in the sense of Steinberg \cite[\S4.4]{Andersen-Jantzen-cohomology-induced-reps-good-filtration}: 
The $\Delta$-dominant $T$-weights $\chi$ of $\Lie(G)$ all satisfy $0 \leq \langle \chi, \alpha^{\vee} \rangle \leq p-1$ for all $\alpha \in \Delta$. Let $\tilde G$ be the simply-connected cover of the derived subgroup of $G$ (in the sense of algebraic groups, as in \cite[II, 1.6]{jantzen-representations}).
\begin{comment}
Jantzen Reference checked.
\end{comment} 
Since $p$ is good for $G$, the $\tilde{G}$-module $\Sym^n \Lie(\tilde G)$ admits a good filtration for all $n \geq 1$ by \cite[Proof of \Prop~4.4]{Andersen-Jantzen-cohomology-induced-reps-good-filtration}.

Let $V',V''$ be two $L$-modules. By \cite[4.1(1)]{Andersen-Jantzen-cohomology-induced-reps-good-filtration}, the sum $V \oplus V'$ admits a good filtration if and only if both summands $V,V'$ admit good filtrations. Further, if $V'$ and $V''$ both admit good filtrations, so does $V' \otimes V''$, \cite[\Th~4.3.1]{Donkin-good-filtrations-LNM} (see also \cite[4.1.4(3)]{Andersen-Jantzen-cohomology-induced-reps-good-filtration})\footnote{The theorem of Donkin cited above assumed that $p>2$ when $G$ has factors of type $E_7$ or $E_8$. Mathieu later showed \cite{Mathieu-filtrations-G-modules} that the existence of good filtrations is closed under tensor product without any restriction on $p$ or $G$ (still assumed connected, reductive).}. It follows as in \cite[4.1.4]{Andersen-Jantzen-cohomology-induced-reps-good-filtration} that $\Sym^n(V' \oplus V'')$ admits a good filtration for all $n \geq 1$ if and only if $\Sym^n(V')$ and $\Sym^n(V'')$ both admit good filtrations for all $n \geq 1$.

Let $U^+$ be the unipotent radical of $P^+$. One has $\Lie(G)=\Lie(P) \oplus \Lie(U^+)$ and  $\Lie(P)^{\vee}=\Lie(P^+)$ as $L$-modules. Also $\Lie(G) \simeq \Lie(G)^{\vee}$ as $G$-modules. Thus $(\Lie(G)/\Lie(P))^{\vee}=\Lie(U)$ as $L$-modules. Since $\Lie(G)=\Lie(P^+) \oplus \Lie(U)$ as $L$-modules, we are reduced to showing that $\Sym^n \Lie (G)$ admits a good filtration as $L$-module for all $n \geq 1$. 
\begin{comment} By \cite[\Prop~3.1.1]{Donkin-good-filtrations-LNM} if a V admits a filtration and each graded has a good filtration then V does too.
\end{comment}
By reductions of Donkin, \cite[\Props~3.2.7,~3.4.3]{Donkin-good-filtrations-LNM}, 
\begin{comment}
References to Donkin checked: 3.2.7 is the reduction to the derived subgroup and 3.4.3 states a module admits a good filtration for semisimple $G$ if and only if it does so for the simply-connected cover of $G$.
\end{comment}
since  $\Sym^n \Lie(\tilde G)$ admits a good filtration as a $\tilde{G}$-module, it also admits a good filtration as $G$-module. Since $\Lie(G)$ is the direct sum of its center and $\Lie(\tilde G)$, the result used above about symmetric powers of a sum implies that $\Sym^n \Lie(G)$ admits a good filtration as a $G$-module for all $n \geq 1$. By \cite[\Th~4.3.1]{Donkin-good-filtrations-LNM}, the property of admitting a good filtration is stable under restriction to a Levi subgroup. Thus $\Sym^n \Lie (G)$ admits a good filtration as $L$-module for all $n \geq 1$.  
\end{proof}

\begin{lemma} 
\label{lem-serre-letter-going-down} One has $M^0 \rightsquigarrow_{\Hcal} M_e$. 
\end{lemma}

\begin{proof}
Let $f \in H^0(S_{\Kcal}, \veta)$ be an eigenform with system of eigenvalues $(b_T)_{T \in \Hcal}$.
Let $\Ical$ be the ideal sheaf of $S_e$ in $\Shko$. By \S\ref{sec-univ-Gzip}, the  short exact sequence \begin{equation} 0 \to \mathcal I \otimes \veta \to \veta \to \veta|_{S_e} \to 0. 
\label{eq-ideal-sheaf} \end{equation} is $\gofafp$-equivariant (\S\ref{sec-gofafp-objects-1}). Similarly all the powers $\Ical^j$ are $\gofafp$-equivariant and so are the short exact sequences \begin{equation}
\label{eq-ideal-sheaf-powers}
0 \to \Ical^{j-1} \to \Ical^j \to \Ical^j/\Ical^{j-1} \to 0.
\end{equation}
Since~\eqref{eq-ideal-sheaf} and~\eqref{eq-ideal-sheaf-powers} are $\gofafp$-equivariant, the induced long exact sequences in cohomology are $\Hcal$-equivariant (\S\ref{sec-hecke-open-sh}).
Since $\Shko$ is Noetherian and $S_e$ is nonempty, there exists a largest $j \in \zgeqz$ such that $f \in H^0(\Shko, \mathcal I^j \otimes \veta)$. 
Then the image $\bar{f}$ of $f$ in $H^0(S_e, \mathcal I^j/\mathcal I^{j+1} \otimes \veta)$ is non-zero. 
Since $S_e$ and $\Shko$ are smooth, $\Ical^j/ \Ical^{j+1}=\Sym^j(\Ical /\Ical^2)$. Since in addition $S_e$ is zero-dimensional, the "co-normal exact sequence" of differentials simplifies to a $\gofafp$-equivariant isomorphism $\Ical/\Ical^2 \cong (\Omega^1_{\Shko})|_{S_e}$ of sheaves on $S_e$. 
   Hence $\bar{f} \in H^0(S_e, \Sym^j\Omega^1_{\Shko}\otimes 
\veta)$. 
So the eigensystem $(b_T)_{T \in \Hcal}$ of $f$ also appears in $H^0(S_e, \Sym^j\Omega^1_{\Shko} \otimes \veta)$.

The automorphic vector bundle $\Omega^1_{\Shko}$ is associated to the $L$-module $(\Lie(G)/\Lie(P))^{\vee}=\Lie(U)$ via the construction of \S\S\ref{sec-torsors}, \ref{sec-toroidal-bundles}. Indeed, it is sufficient to check this for the cotangent bundle over $\Sscr_{\Kcal, \CC}$, where it follows trivially from the complex uniformization of $\Sscr_{\Kcal}(\CC)$ and the Borel embedding $\XX \hookrightarrow (\GG/\PP)(\CC)$.
Since the associated bundle construction is an exact tensor functor,  $\Sym^j(\Omega^1_{\Shko})\otimes \veta$ is associated to the $L$-module $\Sym^j((\Lie(G)/\Lie(P))^{\vee}) \otimes V_{\eta}$. 
By Lemma~\ref{lem-good-filtration}, this $L$-module admits a good filtration. Hence the system of Hecke eigenvalues of $f$ appears in $H^0(S_e, \Vscr(\eta'))$ for some $\eta' \in \chargpldom$ and some graded piece $V_{\eta'}$ of a good filtration for $\Sym^j((\Lie(G)/\Lie(P))^{\vee}) \otimes V_{\eta}$ .
\end{proof}

By \Th~\ref{GTHI} (with $w=e$), there exists a nowhere vanishing, Hecke-equivariant section $h_e \in H^0(S_e, \omega^{N_e})$ for some $N_e \geq 1$. The following is then immediate:
\begin{lemma}
\label{lem-serre-letter-increase-weight}
Multiplication by $h_e$ induces a Hecke-equivariant isomorphism 
\[H^0(S_e, \Vscr(\eta)) \stackrel{\sim}{\rightarrow} H^0(S_e, \Vscr(\eta) \otimes \omega^{N_e}).\]
\end{lemma} 
The last step, of going up from $S_e$ to $S_{\Kcal}$, is given by:
\begin{lemma} 
\label{lem-serre-letter-going-up}
For every $\eta \in \chargpldom$ and every system of Hecke eigenvalues $(b_T)_{T \in \Hcal}$ appearing in $H^0(S_e, \Vscr(\eta))$, there exists $m \geq 1$ such that $(b_T)_{T \in \Hcal}$ also appears in $H^0(S_{\Kcal}, \veta \otimes \omega^m)$. 
\end{lemma}

\begin{proof}

Let $f \in H^0(S_e, \veta)$ be an eigenform with eigenvalues $(b_T)$. By Lemma~\ref{lem-serre-letter-increase-weight}, $fh_e^{a} \in H^0(S_e, \veta \otimes \omega^{aN_e})$ is an eigenform with the same eigenvalues $(b_T)$ for all $a\geq 1$. 

Let $\Sigma$ be a smooth, projective, finite, admissible rpcd. Let $D'$ be a $\gofafp$-equivariant Cartier divisor supported on $D$ such that $\omega^k(-D')$ is ample on $\Sscr_{\Kcal}^{\Sigma}$ for all $k \gg 0$, as afforded by hypothesis ($\star$) of \Th~\ref{th-serre-letter} and \Rmk~\ref{rmk-cartier-ample}. Since $S_e \cap D=\emptyset$, one has $\omega^k(-D')|_{S_e}=\omega^k|_{S_e}$ . Let $\Ical^{\Sigma}$ be the ideal sheaf of $S_e$ in $\Shktoro$. Then we have the analogue of~\eqref{eq-ideal-sheaf} for $\Shktoro$: \begin{equation} 0 \to \Ical^{\Sigma} \otimes \vcaneta \to \vcaneta \to \veta|_{S_e} \to 0. \label{eq-ideal-sheaf-tor} \end{equation} By Serre vanishing, if $r$ is sufficently large,  then \begin{equation}
\label{eq-serre-letter-h1-vanish}
H^1(\Shktoro, \mathcal \vcaneta \otimes \Ical^{\Sigma} \otimes \omega^{kr}(-rD'))=0.
\end{equation}
Twisting~\eqref{eq-ideal-sheaf-tor} by $\omega^{kr}(-rD')$ and passing to cohomology, the vanishing~\eqref{eq-serre-letter-h1-vanish} gives a surjection 
\begin{equation} 
\label{eq-surjection-for-serre}
\xymatrix{H^0(\Shktoro, \vcaneta \otimes \omega^{kr}(-rD')) \ar@{>>}[r] & H^0(S_e,\veta \otimes \omega^{kr})}. \end{equation}
Let $j:rD' \to S_{\Kcal}^{\Sigma}$ be the inclusion. The restriction map~\eqref{eq-surjection-for-serre} is the composition of  \begin{equation}
\label{eq-serre-lift-fact-1}
H^0(\Shktoro, \vcaneta \otimes \omega^{kr}(-rD')) \to H^0(\Shktoro, \vcaneta \otimes \omega^{kr})
\end{equation}
  and 
\begin{equation}
\label{eq-serre-lift-fact-2}
H^0(\Shktoro, \vcaneta \otimes \omega^{kr}) \to H^0(\Shko, \veta \otimes \omega^{kr}),
\end{equation} where the first map is induced from $0 \to \Ocal_{S_{\Kcal}^{\Sigma}}(-rD') \to \Ocal_{S_{\Kcal}^{\Sigma}} \to \Ocal_{rD'} \to 0$ and the second is restriction. Since both~\eqref{eq-serre-lift-fact-1} and~\eqref{eq-serre-lift-fact-2} are Hecke-equivariant by the methods of \S\S\ref{sec-hecke-tor-sh}-\ref{sec-hecke-hasse-regular}, so is there composition~\eqref{eq-surjection-for-serre}. 

Therefore $(b_T)$ appears in $H^0(\Shktoro, \vcaneta \otimes \omega^{kr}(-rD'))$. Since the latter is Hecke-equivariantly embedded in $H^0(\Shktoro, \vcaneta \otimes \omega^{kr}(-D))$ and $\vcaneta \otimes \omega^{kr}(-D)=\vsub(\eta) \otimes \omega^{kr}$, we conclude that $(b_T)$ also appears in $H^0(\Shktoro, \vsubeta \otimes \omega^{kr})$.

\end{proof}

\begin{proof}[Proof of \textnormal{\Th~\ref{th-serre-letter}}] 
Apply Lemmas~\ref{lem-sub-can-0},~\ref{lem-serre-letter-going-down},~\ref{lem-serre-letter-increase-weight} and \ref{lem-serre-letter-going-up} consecutively.

\end{proof}
\begin{rmk}
Suppose that the theory of integral models and their compactifications applies to $\gx$ also for $p=2$. 
\eg this is the case if $\gx$ is of PEL-type with no factors of type D. Then it is clear that the proof of Lemma~\ref{lem-serre-letter-going-up} works for $p=2$. However, even in the Siegel case, we don't know whether Lemma~\ref{lem-serre-letter-going-down} remains valid for $p=2$. According to \cite[p.~241]{jantzen-representations}, it seems likely that Lemma~\ref{lem-good-filtration} is false when $p=2$.
\end{rmk}
\begin{proof}[Proof of \textnormal{\Cor~\ref{cor-serre-letter-finiteness}}] Let $d=\dim \Shktoro$.
The canonical bundle of $\Shktoro$ is isomorphic to $\omega^m(-D)$ for some $m \geq 1$, \cite[5.3.11]{MadapusiHodgeTor}. By Serre duality, $H^i(\Shktoro,\vcaneta) \cong H^{d-i}(\Shktoro, \vsub(m\eta_{\omega}-\eta))$. By \Th~\ref{th-reduction-to-h0}, $F(i,1,\eta) \neq \emptyset$ for all $\eta$. Hence \Rmk~\ref{rmk-reduction-to-h0-mod-p}  and \Th~\ref{th-serre-letter} imply  that every system system of Hecke eigenvalues that appears in~\eqref{eq-serre-letter-finiteness} also appears in $M_e$ (recall that $\veta$, $\vcaneta$ and $\vsubeta$ all have the same restriction to $S_e$).
 The finiteness then follows by the same argument as that given in
\cite{Serre-Two-Letters-Modular-Forms,Ghitza-Siegel-Mod-p-Algebraic,Reduzzi-PEL-Mod-p}.
\end{proof}

\begin{rmk}
The finiteness of Hecke eigen-systems appearing in \eqref{eq-serre-letter-finiteness} can also be deduced as follows: Recall that the (non-compactified) flag space $Fl_\Kcal$ introduced in \eqref{eq-flag-space-fiber-product} admits a smooth map $Fl_\Kcal \to \GF^\mu$. By pull-back, the stratification defined in \S\ref{sec-stratification-GZipFlag} induces a stratification of $Fl_\Kcal$. There is a unique zero-dimensional stratum $Fl_e$, and the projection $\pi:Fl_\Kcal \to \Shko$ induces a $\gofafp$-equivariant map $Fl_e\to S_e$. Hence, we obtain a Hecke-equivariant injection $H^0(S_e,\Vscr(\eta))\to H^0(Fl_e,\pi^*\Vscr(\eta))$. Furthermore, each $\pi^*\Vscr(\eta)$ admits a filtration with graded pieces of the form $\Lscr(\lambda)$. For each character $\chi\in X^*(T)$, there exists an integer $N\geq 1$ such that $\Lscr(\chi)^N$ admits a non-vanishing $\gofafp$-equivariant section (obtained by pull-back from the corresponding stratum in $\GF^\mu$). In particular, the lattice of weights of $\gofafp$-equivariant non-vanishing sections on $Fl_e$ has finite covolume in $X^*(T)$. The finiteness of Hecke-eigensystems follows immediately.

\end{rmk}

\section*{Acknowledgements} We thank our respective collaborators Marc-Hubert Nicole and Torsten Wedhorn: The joint projects \cite{Goldring-Nicole-mu-Hasse} and \cite{Koskivirta-Wedhorn-Hasse} were invaluable in leading us to the results of the current paper. We are grateful to  David Geraghty for the many important ideas that he contributed during the initial part of this project. In addition, we thank Joseph Ayoub, Laurent Clozel, Pierre Deligne, Phillip Griffiths, David Helm, Matt Kerr, Arno Kret, Ben Moonen, Marc-Hubert Nicole, Stefan Patrikis, Jonathan Pottharst, Sug Woo Shin, Benoit Stroh, Jacques Tilouine, Adrian Vasiu and Torsten Wedhorn for helpful conversations and correspondence. Finally, we thank the referee for his/her very thorough work, detailed comments and valuable suggestions, which we believe have led to considerable improvements in the paper.

\bibliographystyle{plain}
\bibliography{biblio_overleaf}

\end{document}